\documentclass[10pt,a4paper]{article}

\usepackage[utf8]{inputenc}
\usepackage[T1]{fontenc}
\usepackage{lmodern}
\usepackage{microtype}          
\usepackage[numbers,sort&compress]{natbib}

\usepackage{geometry}
\usepackage{fancyhdr}
\usepackage{amsmath,amssymb,amsfonts,amsthm,mathtools}
\usepackage{mathrsfs}
\usepackage{bm}
\usepackage{enumitem}
\usepackage{hyperref}
\usepackage[nameinlink,capitalize]{cleveref}
\usepackage{esint}

\usepackage{xcolor}

\definecolor{refblue}{RGB}{0, 51, 153}      
\definecolor{citered}{RGB}{178, 34, 34}     
\definecolor{tocmuted}{RGB}{102, 51, 51}    

\hypersetup{
	colorlinks=true,
	linkcolor=blue!80!black,      
	citecolor=red!80!black,       
	urlcolor=blue,        
	filecolor=blue,       
	plainpages=false,
	pdfstartview=FitH,
	bookmarksopen=true,
	bookmarksnumbered=true,
	breaklinks=true
}

\usepackage{enumitem}
\setlist{noitemsep, topsep=4pt}

\usepackage{mathtools}
\usepackage{etoolbox}
\apptocmd{\thebibliography}{\color{refblue}}{}{}

\crefdefaultlabelformat{#2\color{tocmuted}#1#3}

\numberwithin{equation}{section}
\numberwithin{figure}{section}
\numberwithin{table}{section}
\newtheorem{theorem}{Theorem}[section]
\newtheorem{lemma}[theorem]{Lemma}
\newtheorem{proposition}[theorem]{Proposition}
\newtheorem{corollary}[theorem]{Corollary}
\newtheorem{assumption}[theorem]{Assumption}
\newtheorem{definition}[theorem]{Definition}
\newtheorem{remark}[theorem]{Remark}

\newcommand{\R}{\mathbb R}

\newcommand{\eps}{\varepsilon}
\newcommand{\Om}{\Omega}
\newcommand{\Omeps}{\Omega^\eps_p}

\newcommand{\Yp}{Y_p}

\newcommand{\norm}[2]{\left\|#1\right\|_{#2}}

\newcommand{\phie}{\phi^\eps}
\newcommand{\mue}{\mu^\eps}
\newcommand{\phiz}{\phi}
\newcommand{\muz}{\mu}

\def\e{{\varepsilon}}
\def\Te{{\mathcal{T}_\e}}
\def\X{{\times}}
\def\md{{\mathrm{d}}}
\def\bn{{\mathbf{n}}}
\def\Kc{{\mathcal{K}}}
\def\O{{\Omega}}

\def\GD{{\mathbf{D}}}
\def\B{{\mathbb{B}}}
\def\per{{\mathrm{per}}}
\def\Qc{{\mathcal{Q}}}

\def\div{{\mathrm{div}}}
\def\GO{\mathcal{O}}

\def\be{{\mathbf{e}}}
\title{Quantitative Homogenization of a Cahn--Hilliard System with Source Term in Periodically Perforated Domains}

\author{Amartya Chakrabortty\thanks{
		Processes and Materials, Fraunhofer Institute for Industrial Mathematics ITWM,
		Fraunhofer-Platz 1, 67663 Kaiserslautern, Germany.
		Email: \href{mailto:amartya.chakrabortty@gmail.com}{amartya.chakrabortty@gmail.com}.
		ORCID: \href{https://orcid.org/0009-0008-1353-5246}{0009-0008-1353-5246}.
}}

\date{\today}

\begin{document}

\maketitle

\begin{abstract}
	We study qualitative and quantitative homogenization for a Cahn--Hilliard system 
	with a nonconservative  source term in a periodically perforated domain. 
Using the periodic unfolding method, we derive uniform energy estimates and 
prove convergence to a homogenized Cahn--Hilliard system whose effective 
diffusion tensor is characterized by scalar Neumann cell problems on the 
pore cell.
	For the 
	quantitative analysis, we construct first-order corrector approximations by means 
	of a scale-splitting operator, so that the cell correctors are only required to 
	belong to $H^1_\per(Y_p)$. Under $H^2$-regularity of the homogenized solution and 
	well-prepared initial data, we obtain an order \(\varepsilon^{1/2}\) corrector estimate: the corrected
	order-parameter error is controlled in \(L^2(0,T;H^1(\Omega_p^\varepsilon))\),
	while the uncorrected order parameter is controlled in \(L^2(0,T;L^2(\Omega_p^\varepsilon))\). This improves the rate 
	$\varepsilon^{1/4}$ previously established for fourth-order phase-field equations 
	in perforated media, and matches the natural rate for second-order elliptic 
	problems in perforated domains. The rate reflects the boundary layer caused by 
	incomplete cells near $\partial\Omega$ and improves to order $\varepsilon$ on the flat torus $\mathbb{T}^d$.
\end{abstract}


\smallskip


\section{Introduction}
\label{sec:intro}

The Cahn--Hilliard equation, introduced by Cahn and Hilliard~\cite{CahnHilliard1958}
to model spinodal decomposition in binary alloys, is a fundamental diffuse-interface
model for phase separation and interfacial dynamics. Its applications range from
phase transitions in materials science~\cite{Cahn1961} and two-phase
flows~\cite{HohenbergHalperin1977} to tumor growth~\cite{OdenEtAl2010}, image
inpainting~\cite{BertozziEtAl2007}. In homogeneous domains the mathematical theory is
well-developed: the Ginzburg--Landau free energy is a Lyapunov functional, the
system is a gradient flow in $H^{-1}$, and well-posedness is classical for both
regular potentials~\cite{Elliott1989,ElliottZheng1986,NicolaenkoScheurerTemam1989}
and singular ones~\cite{AbelsWilke2007,BloweyElliott1991}; see the
monograph~\cite{Miranville2019} and the survey~\cite{Wu2021}.

In many physically relevant situations, however, phase separation occurs in a
\textbf{porous medium}: a material with a periodic array of solid inclusions, the
fluid phases occupying the connected pore space. When the ratio $\varepsilon > 0$
between the pore scale and the macroscopic domain size is small, resolving the
microscopic geometry directly is computationally prohibitive. Homogenization then
provides effective macroscopic equations, and reveals how pore geometry enters
the macroscopic dynamics through an effective diffusion tensor determined by cell
problems on the reference pore cell. A further feature of practical importance is
the presence of \textbf{nonconservative source terms}: in porous electrodes,
reactive porous media, and biological tissues undergoing phase transitions, the
order parameter is subject to a volumetric source or sink rather than being
conserved. This paper addresses both the qualitative and quantitative homogenization
of a Cahn--Hilliard system with a monotone nonconservative source term in a
periodically perforated domain.

\smallskip
Let $0 < \varepsilon < 1$ denote the ratio between the pore scale and the macroscopic
domain size, and let $\Omega_p^\varepsilon \subset \Omega \subset \mathbb{R}^d$,
$d \in \{2,3\}$, be the periodically perforated pore domain obtained by removing a
periodic array of solid inclusions from a bounded Lipschitz domain $\Omega$; see
Section~\ref{sec:geometry} for the precise construction. The microscopic unknowns
are the order parameter $\phi_\varepsilon$ and the chemical potential $\mu_\varepsilon$,
solving
\begin{equation*}
	\begin{cases}
		\partial_t \phi_\varepsilon - \Delta \mu_\varepsilon + G(\phi_\varepsilon) = 0
		& \text{in } (0,T) \times \Omega_p^\varepsilon,\\[1mm]
		\mu_\varepsilon = -\Delta \phi_\varepsilon + F'(\phi_\varepsilon)
		& \text{in } (0,T) \times \Omega_p^\varepsilon,\\[1mm]
		\nabla \phi_\varepsilon \cdot \mathbf{n}_\varepsilon
		= \nabla \mu_\varepsilon \cdot \mathbf{n}_\varepsilon = 0
		& \text{on } (0,T) \times \partial\Omega_p^\varepsilon,\\[1mm]
		\phi_\varepsilon(0,\cdot) = \phi_\varepsilon^0
		& \text{in } \Omega_p^\varepsilon.
	\end{cases}
\end{equation*}
Here $F$ is a general potential satisfying polynomial growth, coercivity, and
dissipativity conditions (Assumption~\ref{ass:F}), with the classical double-well
$F(s) = \tfrac{1}{4}(s^2-1)^2$ as the model example, and $G$ is a monotone
globally Lipschitz source term with $G(0) = 0$ and $G' \geq c_G > 0$
(Assumption~\ref{ass:G}). The homogeneous Neumann conditions are imposed on the
full boundary $\partial\Omega_p^\varepsilon = \partial\Omega \cup \Gamma_s^\varepsilon$,
encoding no-flux conditions through both the outer boundary $\partial\Omega$ and
the pore walls $\Gamma_s^\varepsilon$.

The first main result, Theorem~\ref{thm:qual-hom}, establishes that as
$\varepsilon \to 0$ the solutions $(\phi_\varepsilon, \mu_\varepsilon)$ converge,
up to subsequences, to a pair $(\phi, \mu)$ solving the homogenized Cahn--Hilliard
system
\begin{equation*}
	\begin{cases}
		\partial_t \phi - \operatorname{div}(\B^{\mathrm{hom}}\nabla\mu)
		+ G(\phi) = 0
		& \text{in } (0,T)\times\Omega,\\[1mm]
		\mu = -\operatorname{div}(\B^{\mathrm{hom}}\nabla\phi) + F'(\phi)
		& \text{in } (0,T)\times\Omega,\\[1mm]
		\B^{\mathrm{hom}}\nabla\phi \cdot \mathbf{n}
		= \B^{\mathrm{hom}}\nabla\mu \cdot \mathbf{n} = 0
		& \text{on } (0,T)\times\partial\Omega,\\
		\phi(0,\cdot) = \phi^0
		& \text{in } \Omega.
	\end{cases}
\end{equation*}
The effective diffusion tensor $\B^{\mathrm{hom}}$ is symmetric and uniformly
elliptic, determined by scalar Neumann cell problems on the pore cell $Y_p$;
see~\eqref{eq:Bhom}. The proof uses the periodic unfolding method adapted to
perforated domains~\cite{CDG02,CDG08,CDG18,CDG08+,CD06}, uniform energy estimates,
and compactness. 
A structural observation driving the analysis is that $G$, being monotone with
$G(0) = 0$, produces a dissipative contribution in the energy identity
(Lemma~\ref{lem:Gdissipation-weak}): testing the chemical potential equation
with $G(\phi_\varepsilon)$ yields a positive gradient term, so the source term
reinforces rather than competes with the dissipation. As a result, the right-hand
side of the energy inequality is a fixed constant independent of $\varepsilon$,
and no Gronwall argument is needed. This stands in contrast to problems with a
general nonconservative forcing term, where the source would need to be controlled
by Gronwall at the cost of an exponential-in-time factor.

The second main result, Theorem~\ref{thm:corrector-estimate}, is a quantitative
corrector estimate. Under an additional $H^2$-regularity assumption on the
homogenized solution (Assumption~\ref{ass:quant-reg}) and well-prepared initial
data, we construct first-order corrector approximations
\begin{equation*}
	\Phi_\e=\phi + \varepsilon\,\sum_{i=1}^d \Qc_\varepsilon(\partial_{x_i}\phi)\,
	\chi_i\!\left(\tfrac{x}{\varepsilon}\right),
	\quad\text{and}\quad
	M_\e=\mu + \varepsilon\,\sum_{i=1}^d \Qc_\varepsilon(\partial_{x_i}\mu)\,
	\chi_i\!\left(\tfrac{x}{\varepsilon}\right),
\end{equation*}
where $\Qc_\varepsilon$ is the $Q_1$ scale-splitting operator from \cite{CDG08,CDG18,Griso2004} applied on the full domain $\Omega$, and
$\chi_i \in H^1_\per(Y_p)$ are the scalar pore-cell correctors solving
the Neumann cell problem~\eqref{eq:cell-problem}. The estimate reads
\begin{multline}
	\label{eq:rate-intro}
	\|\phi_\varepsilon-\phi\|_{L^2(0,T;L^2(\Omega_p^\varepsilon))}
	+
	\left\|
	\nabla\phi_\varepsilon
	-
	\left[
	\nabla\phi
	+
	\sum_{i=1}^d
	\Qc_\varepsilon(\partial_{x_i}\phi)
	\nabla_y\chi_i\!\left(\tfrac{x}{\varepsilon}\right)
	\right]
	\right\|_{L^2((0,T)\times\Omega_p^\varepsilon)}\\
	+
	\|\mu_\varepsilon-\mu\|_{L^2(0,T;H^1(\Omega_p^\varepsilon)')}
	\leq C \varepsilon^{1/2}.
\end{multline}
The three terms on the left measure: the uncorrected $L^2$-error in the order
parameter, the corrected gradient error in the order parameter, and the chemical
potential error in the dual norm $L^2(0,T;H^1(\Omega_p^\varepsilon)')$.
The proof of Theorem~\ref{thm:corrector-estimate} uses a negative-norm energy
method: the phase error equation is tested against $\mathcal{N}_\varepsilon
\widetilde{e}_\phi^\varepsilon$, where $\mathcal{N}_\varepsilon$ is the Neumann
inverse Laplacian on $\Omega_p^\varepsilon$ and $\widetilde{e}_\phi^\varepsilon$
is the zero-mean part of the phase error. This produces an evolution equation
for the $H^{-1}_\varepsilon$-norm of the error (defined in Section~\ref{subsec:rates}
via the Neumann inverse Laplacian on $\Omega_p^\varepsilon$), which is then coupled with the
chemical potential error equation tested against $\widetilde{e}_\phi^\varepsilon$
itself. The resulting system is closed by a Gronwall argument using the residual
estimate of Lemma~\ref{lem:residual}.
A gradient
estimate for $\mu_\varepsilon - \mu$ in $L^2$ is not obtained; this is sharp
within the present variational framework, since controlling
$\|\nabla(\mu_\varepsilon - M_\varepsilon)\|_{L^2}$ would require testing the
chemical potential error equation with $\partial_t e_\phi^\varepsilon$, which is
not justified at the regularity level of the residuals; see
Remark~\ref{rem:no-mu-gradient}. On the flat torus $\mathbb{T}^d$, where there
is no boundary layer of incomplete cells, the rate improves to
$\mathcal{O}(\varepsilon)$ (Remark~\ref{rem:periodic-improved}).

\smallskip
The Cahn--Hilliard equation has been extensively studied analytically.
Well-posedness for regular potentials via Faedo--Galerkin is classical;
see~\cite{Elliott1989,ElliottZheng1986,NicolaenkoScheurerTemam1989} and
the monograph~\cite{Miranville2019} and survey~\cite{Wu2021}. For singular potentials, well-posedness and
separation results are in~\cite{AbelsWilke2007,BloweyElliott1991,CherfilisMiranvilleZelik2011,ElliottLuckhaus1991}. Well-posedness with a
Lipschitz source term of the type considered here follows
from~\cite{Elliott1989,NicolaenkoScheurerTemam1989} and is recalled in
Theorem~\ref{thm:existence} for the fixed-$\varepsilon$ problem on
$\Omega_p^\varepsilon$.

Homogenization of the Cahn--Hilliard equation in heterogeneous media has been
studied both formally and rigorously. Formal upscaling via matched asymptotic
expansions appears in~\cite{SchmuckPavliotisKalliadasis2014,SchmuckPradasPavliotisKalliadasis2012}, and qualitative
convergence by evolutionary $\Gamma$-convergence in~\cite{liero2018homogenization}. Qualitative homogenization of Cahn-Hilliard system coupled with Stokes equation is present in \cite{bavnas2017homogenization,daly2015homogenization,lakhmara2022homogenization,schmuck2013derivation} and Navier-Stokes equation in \cite{bunoiu2020homogenization,chakrabortty2025navier}.
The periodic unfolding method for perforated domains, developed
in~\cite{CDG08+,CD06} and collected in~\cite{CDG18}, provides the natural
framework for identifying the two-scale limit on $\Omega_p^\varepsilon$. We use
this framework throughout Sections~\ref{sec:homogenization}--\ref{sec:quant}.

Quantitative convergence rates in periodic homogenization have a long history.
The energy method of Bensoussan--Lions--Papanicolaou~\cite{BLP1978} gives
$\mathcal{O}(\varepsilon^{1/2})$ in $H^1$ for second-order elliptic problems on
bounded domains; see also~\cite{JikovKozlovOlejnik1994}. Griso~\cite{Griso2004,
	Griso2006} removed the $W^{1,\infty}$ assumption on correctors while keeping the
$\mathcal{O}(\varepsilon^{1/2})$ rate in $L^2(\Omega)$ and in the corrected
gradient, by introducing the scale-splitting operator $\Qc_\varepsilon$; this is the
approach we adopt. Reaching $\mathcal{O}(\varepsilon)$ in $L^2$ requires more.
For scalar elliptic equations on $\mathbb{R}^n$, Griso~\cite{Griso2006} obtained
it via interior estimates. For elliptic systems on bounded $C^{1,1}$ domains with
H\"{o}lder coefficients, Kenig--Lin--Shen~\cite{KenigLinShen2012} established
$\mathcal{O}(\varepsilon)$ in $L^2$ via a duality argument, building on the uniform
$W^{1,p}$ estimates of~\cite{KenigShen2011a,KenigShen2011b}. Extensions cover Green
and Neumann functions~\cite{KenigLinShen2014}, boundary estimates in $C^{1,\alpha}$
domains~\cite{Shen2017}, and linear elasticity~\cite{ShenZhuge2017}. For bounded
measurable coefficients, the $\mathcal{O}(\varepsilon)$ rate comes instead from the
operator-estimate approach of Birman--Suslina~\cite{BirmanSuslina2004,
	BirmanSuslina2006,BirmanSuslina2007}, surveyed in~\cite{ZhikovPastukhova2016};
see also~\cite{Zhikov2006,ZhikovPastukhova2005}. For perforated domains,
$\mathcal{O}(\varepsilon^{1/2})$ in $H^1$ is in~\cite{OnofreiVernescu2007}. For
Stokes systems with Neumann conditions, rates $\mathcal{O}(\varepsilon^{1/2})$ in
$H^1$ and $\mathcal{O}(\varepsilon)$ in $L^2$ are proved in~\cite{Gu2016,Gu2018}
via the Steklov smoothing operator and the duality argument of~\cite{KenigLinShen2012}.
For parabolic systems with time-dependent periodic coefficients, $\mathcal{O}(\varepsilon)$
in $L^2$ is in~\cite{GengShen2016}. Two-scale convergence goes back to
Nguetseng~\cite{Nguetseng1989} and Allaire~\cite{Allaire1992}. For a
systematic operator-theoretic treatment of quantitative two-scale approximations,
including high-contrast and degenerating problems, see~\cite{CooperKamotskiSmyshlyaev2025}. Quantitative rates for incompressible Navier-Stokes in a perforated domain in $\R^3$ is established in \cite{hofer2023homogenization} and compressible in \cite{hofer2024quantitative}.

The only prior quantitative result for a Cahn--Hilliard system in a perforated
medium is Schmuck--Kalliadasis~\cite{Schmuck17}. They prove an
$\mathcal{O}(\varepsilon^{1/4})$ rate for the $H^1$-error in the order parameter
by the classical energy method with a boundary cutoff. The rate degrades from
$\mathcal{O}(\varepsilon^{1/2})$ because the fourth-order structure introduces
boundary terms in the corrector equation of order $\varepsilon^{-1/2}$. Their
argument requires cell correctors in $W^{1,\infty}$ and the macroscopic solution
in $C^1(0,T;W^{k,\infty}(\Omega))$ for $k \geq 4$ (Assumption~C
in~\cite{Schmuck17}). Remark~4 of that paper anticipates that an operator-estimate
approach in the spirit of Suslina~\cite{Suslina2013Neumann,Suslina2013Dirichlet}
should recover the natural $\mathcal{O}(\varepsilon)$ rate for fourth-order operators.

\smallskip
\noindent\textbf{Main contributions.}
This paper makes two contributions to the homogenization theory of phase-field
equations in perforated media.

The first is qualitative. We derive the homogenized Cahn--Hilliard
system~\eqref{eq:hom-CH} with a nonconservative monotone source term via the
periodic unfolding method, treating $G$ under Assumptions~\ref{ass:F}--\ref{ass:G}
without specializing to a particular double-well potential.

The second is quantitative. Under the additional regularity
Assumption~\ref{ass:quant-reg} on the homogenized solution, we prove the corrector
estimate~\eqref{eq:rate-intro} of order $\varepsilon^{1/2}$. Specifically, the
uncorrected order-parameter error is controlled in $L^2(0,T;L^2(\Omega_p^\varepsilon))$,
the corrected gradient error in $L^2((0,T)\times\Omega_p^\varepsilon)$, and the
chemical potential error in the dual norm $L^2(0,T;H^1(\Omega_p^\varepsilon)')$.
This improves the rate $\varepsilon^{1/4}$ of~\cite{Schmuck17} — where the
error is measured in the uncorrected $H^1$-norm and no first-order corrector
is subtracted — and matches the natural rate for second-order elliptic problems
in perforated domains~\cite{OnofreiVernescu2007}.
Two methodological points drive the improvement.
First, the scale-splitting operator $\Qc_\varepsilon$ from~\cite{CDG02,Griso2004,
	Griso2006} replaces the classical formal corrector expansion, allowing cell
correctors to stay in $H^1_\per(Y_p)$ without any $W^{1,\infty}$ assumption and
reducing the required regularity of the macroscopic solution from $W^{k,\infty}$
($k \geq 4$) to $H^2$. Second, the consistency estimate
(Lemma~\ref{lem:residual}) is derived entirely in variational form, using the
weak no-flux condition on $\Gamma_s^\varepsilon$ inherited from the cell problem;
this avoids the boundary-term contributions that cause the rate loss
in~\cite{Schmuck17}. The $\varepsilon^{1/2}$ rate on bounded domains comes from
the collar estimate for the boundary layer of incomplete cells near $\partial\Omega$,
not from the fourth-order structure. On the flat torus, where this layer is absent,
the rate improves to $\mathcal{O}(\varepsilon)$.

\noindent\textbf{Outline.}
Section~\ref{sec:problem} introduces the geometry, states the assumptions on $F$
and $G$, formulates the microscopic weak problem, derives the energy inequality,
and establishes the uniform a priori estimates. Section~\ref{sec:homogenization}
recalls the periodic unfolding operators, proves compactness and two-scale
convergence, derives the cell problems, and establishes
Theorem~\ref{thm:qual-hom}. Section~\ref{sec:quant} introduces $Q_\varepsilon$,
constructs the corrector approximations $\Phi_\varepsilon$ and $M_\varepsilon$,
proves the variational residual estimate (Lemma~\ref{lem:residual}), and
establishes Theorem~\ref{thm:corrector-estimate} together with the corollaries
on the uncorrected and corrected gradient rates. Section~\ref{rem:future-rate} concludes with possible improvements.


\section{The microscopic problem}
\label{sec:problem}

\subsection{Geometry of the periodic perforated domain}
\label{sec:geometry}

Let $\Omega\subset\mathbb{R}^d$, $d\in\{2,3\}$, be a bounded Lipschitz domain.
Let $Y := (0,1)^d$
be the reference periodic cell, and let $Y_s\subset Y$ be a nonempty open set
with Lipschitz boundary $\partial Y_s$, representing the solid (impermeable)
inclusion in the reference cell. We assume the frame condition $\overline{Y_s}\subset Y$,
i.e.\ the solid inclusion is compactly contained in the reference cell and
does not touch $\partial Y$. We define the pore part of the cell by $Y_p := Y\setminus\overline{Y_s}$.
Since $\overline{Y_s}$ is closed and $\overline{Y_s}\subset Y$, the set $Y_p$
is automatically open, and $\partial Y_p=\partial Y\,\cup\,\partial Y_s$,
where the two pieces are disjoint (since $\partial Y_s\cap\partial Y=\emptyset$
by the frame condition). We assume in addition that $Y_p$ is connected. The
porosity of the cell is
\[
\theta_p:=|Y_p|\in(0,1).
\]
Throughout, $H^1_{\per}(Y_p)$ denotes the space of $Y$-periodic $H^1$
functions on $Y_p$, and $H^1_{\per,0}(Y_p)$ denotes its mean-zero
subspace $\{v\in H^1_{\per}(Y_p) : \int_{Y_p}v\,dy=0\}$. The cell
correctors belong to $H^1_{\per,0}(Y_p)$, with the mean-zero condition
inherited from the two-scale structure of the unfolding operator.

\noindent\textbf{The perforated domain.}
For $0<\varepsilon<\varepsilon_0$, with a fixed $\varepsilon_0$, let
\[
K_\varepsilon:=\bigl\{\,k\in\mathbb{Z}^d \;:\; \varepsilon(k+Y)\subset\Omega\,\bigr\}
\]
be the set of indices of cells entirely contained in $\Omega$. By the frame
condition, the rescaled closed inclusions $\varepsilon(k+\overline{Y_s})$,
$k\in K_\varepsilon$, are pairwise disjoint and each is compactly contained in
$\varepsilon(k+Y)\subset\Omega$. We set
\[
\Omega^\e_s
:= \operatorname{int}\!\left(\bigcup_{k\in K_\varepsilon}
\varepsilon\bigl(k+\overline{Y_s}\bigr)\right),
\qquad
\O^\e_p := \Omega\setminus\overline{\Omega^\e_s}.
\]
The internal oscillating (pore-wall) boundary is
\[
\Gamma_\varepsilon^s
:= \partial\Omega^\e_s\cap\Omega
= \bigcup_{k\in K_\varepsilon}\varepsilon\bigl(k+\partial Y_s\bigr).
\]
By the frame condition, every connected component of $\Omega^\e_s$
satisfies $\operatorname{dist}(\overline{\Omega^\e_s},\partial\Omega)
\geq c\,\varepsilon$ for some $c>0$ independent of $\varepsilon$; in
particular $\O^\e_p$ contains a fixed-width boundary layer
along $\partial\Omega$ that is free of perforations, and
\[
\partial\O^\e_p=\partial\Omega\,\cup\,\Gamma_\varepsilon^s,
\]
with the two pieces disjoint. We denote by $\bn_\varepsilon$ the unit outward
normal vector field on $\partial\O^\e_p$, i.e.\ the outward
normal to $\Omega$ on $\partial\Omega$ and the normal pointing from
$\O^\e_p$ into $\Omega^\e_s$ on $\Gamma_\varepsilon^s$.

Moreover, we set
$$
\O_\e=\text{interior}\left\{\bigcup_{\kappa\in\Kc_\e}\e(\kappa+\overline{Y})\right\},$$
where $\Lambda_\e=(\O\setminus\O_\e)$ contains the part of the cells intersecting
$\partial\O$; since $\O$ is bounded with Lipschitz boundary, $|\Lambda_\e|\to0$ as
$\e\to0$. 

\begin{remark}\label{rem:frame}
	The frame condition $\overline{Y_s}\subset Y$, together with the Lipschitz
	boundary of $Y_s$ and the connectivity of $Y_p$ assumed above, is precisely
	the hypothesis under which $\O^\e_p$ is connected, has a uniform Lipschitz
	character (independent of $\varepsilon$), and admits a uniform extension
	operator $P_\varepsilon\colon H^1(\O^\e_p)\to H^1(\Omega)$ with
	$\|P_\varepsilon\|_{\mathcal L(H^1(\O^\e_p),H^1(\Omega))}$ bounded
	independently of $\varepsilon$; see~\cite{AcerbiChiadoPiatDalMasoPercivale1992,
		CioranescuDonato1999,oleinik1992mathematical}.
	This extension operator, combined with a contradiction argument
	(assuming the Poincar\'e--Wirtinger constant blows up along a sequence
	$\varepsilon_n\to0$, normalizing, and extracting a compact subsequence via
	the extension to reach a contradiction), yields a uniform
	Poincar\'e--Wirtinger inequality on $\O^\e_p$ with constant independent
	of $\varepsilon$; see~\cite{damlamian2002sequences} for details.
	This inequality will be used in the uniform energy estimates of
	Section~\ref{sec:quant}.
\end{remark}

\subsection{Problem description}

Let $T\in(0,\infty)$. The unknowns are the order parameter
$\phi_\varepsilon\colon(0,T)\times\O^\e_p\to\mathbb{R}$ and
the chemical potential
$\mu_\varepsilon\colon(0,T)\times\O^\e_p\to\mathbb{R}$, solving
\begin{equation}\label{eq:strong}
	\begin{cases}
		\partial_t\phi_\varepsilon-\Delta\mu_\varepsilon+G(\phi_\varepsilon)=0,
		& \text{in } (0,T)\times\O^\e_p,\\[2pt]
		\mu_\varepsilon=-\Delta\phi_\varepsilon+F'(\phi_\varepsilon),
		& \text{in } (0,T)\times\O^\e_p,\\[2pt]
		\nabla\phi_\varepsilon\cdot \bn_\varepsilon
		=\nabla\mu_\varepsilon\cdot \bn_\varepsilon=0,
		& \text{on } (0,T)\times\partial\O^\e_p,\\[2pt]
		\phi_\varepsilon(0,\cdot)=\phi_\varepsilon^0,
		& \text{in } \O^\e_p.
	\end{cases}
\end{equation}
The homogeneous Neumann conditions are imposed on the {whole} boundary
$\partial\O^\e_p=\partial\Omega\cup\Gamma_\varepsilon^s$: there is
no flux of $\phi_\varepsilon$ or $\mu_\varepsilon$ either through the outer
boundary $\partial\Omega$ or through the pore walls $\Gamma_\varepsilon^s$.


\begin{assumption}[General Potential]\label{ass:F}
	The potential $F\in C^2(\mathbb{R})$ satisfies the polynomial growth
	condition: there exists $C>0$ such that
	\[
	|F'(s)|\le C\bigl(1+|s|^3\bigr),
	\qquad
	|F''(s)|\le C\bigl(1+|s|^2\bigr),
	\qquad \forall s\in\mathbb{R}.
	\]
	Moreover:
	\begin{enumerate}
		\item[(i)] $F$ is bounded from below: $F(s)\ge -c_3$ for some $c_3\ge0$.
		\item[(ii)] ({Coercivity}) there exist $c_0>0$, $C_0\ge0$ such that
		\[
		F(s)\ge c_0|s|^4-C_0, \qquad \forall s\in\mathbb{R}.
		\]
		\item[(iii)] ({Dissipativity}) there exist $c_1>0$, $c_2\ge0$ such that
		\[
		s\,F'(s)\ge c_1F(s)-c_2, \qquad \forall s\in\mathbb{R}.
		\]
	\end{enumerate}
	The model example $F(s)=\tfrac14(s^2-1)^2$, $F'(s)=s^3-s$, satisfies all of
	the above with $c_1=4,c_2=1$ in (iii) (and $c_0=\tfrac14$, $C_0=\tfrac14$ in
	(ii)).
\end{assumption}

\begin{assumption}[Source term]\label{ass:G}
	The source term $G\in C^1(\mathbb{R})$ satisfies
	\[
	G(0)=0, \qquad 0<c_G\le G'(s)\le C_G \qquad \forall s\in\mathbb{R}.
	\]
	Consequently $G$ is globally Lipschitz with $|G(s)|\le C_G|s|$, and
	\[
	\bigl(G(a)-G(b)\bigr)(a-b)\ge c_G|a-b|^2 \qquad \forall a,b\in\mathbb{R}.
	\]
\end{assumption}

\begin{assumption}[Initial data]\label{ass:init}
	The initial data $\phi_\varepsilon^0\in H^1(\O^\e_p)$ satisfy
	\[
	\|\phi_\varepsilon^0\|_{H^1(\O^\e_p)}\le C,
	\]
	with $C$ independent of $\varepsilon$. Moreover, there exists
	$\phi_0\in H^1(\Omega)$ such that the unfolded initial data converge
	strongly in the sense of unfolding operator from Definition \ref{def:uo},
	\[
	\mathcal{T}^\ast_\varepsilon(\phi_\varepsilon^0)\to\phi_0
	\quad\text{strongly in } L^2(\Omega\times Y_p).
	\]
\end{assumption}


We use the standard notation $\langle\cdot,\cdot\rangle_{H^1(\O^\e_p)',H^1(\O^\e_p)}$
for the duality pairing between $H^1(\O^\e_p)'$ and
$H^1(\O^\e_p)$, and $(\cdot,\cdot)$ for the $L^2(\O^\e_p)$
inner product.

\begin{definition}[Weak solution]\label{def:weak}
	A pair
	\[
	(\phi_\varepsilon,\mu_\varepsilon)\in
	\Bigl[L^\infty(0,T;H^1(\O^\e_p))\cap
	H^1(0,T;H^1(\O^\e_p)')\Bigr]
	\times L^2(0,T;H^1(\O^\e_p))
	\]
	is called a {weak solution} of \eqref{eq:strong} if
	\begin{enumerate}
		\item for all $\zeta\in L^2(0,T;H^1(\O^\e_p))$,
		\begin{equation}\label{eq:weak1}
			\int_0^T\bigl\langle\partial_t\phi_\varepsilon,\zeta\bigr\rangle_{H^1(\O^\e_p)',H^1(\O^\e_p)}\,dt
			+\int_0^T\!\!\int_{\O^\e_p}\nabla\mu_\varepsilon\cdot\nabla\zeta\,dx\,dt
			+\int_0^T\!\!\int_{\O^\e_p}G(\phi_\varepsilon)\,\zeta\,dx\,dt=0;
		\end{equation}
		\item for a.e.\ $t\in(0,T)$ and for all $\eta\in H^1(\O^\e_p)$,
		\begin{equation}\label{eq:weak2}
			\int_{\O^\e_p}\mu_\varepsilon(t)\,\eta\,dx
			=\int_{\O^\e_p}\nabla\phi_\varepsilon(t)\cdot\nabla\eta\,dx
			+\int_{\O^\e_p}F'(\phi_\varepsilon(t))\,\eta\,dx;
		\end{equation}
		\item the initial condition is attained:
		\begin{equation}\label{eq:weakinit}
			\phi_\varepsilon(0)=\phi_\varepsilon^0 \quad\text{in } L^2(\O^\e_p).
		\end{equation}
	\end{enumerate}
\end{definition}

\begin{remark}\label{rem:weaksense}
	The space $L^\infty(0,T;H^1(\O^\e_p))\cap H^1(0,T;H^1(\O^\e_p)')$
	embeds continuously into $C([0,T];L^2(\O^\e_p))$ via the standard
	interpolation lemma for evolution triples
	$H^1(\O^\e_p)\hookrightarrow L^2(\O^\e_p)\hookrightarrow H^1(\O^\e_p)'$;
	see~\cite{Temam1997}. Hence the initial condition~\eqref{eq:weakinit}
	is meaningful pointwise at $t=0$. Condition~\eqref{eq:weak2} is imposed
	for a.e.\ $t$ rather than in integrated form, since no time derivative
	of $\mu_\varepsilon$ appears in the system.
\end{remark}

\subsection{Structure of the source term}

Two estimates for the source term $G$ are derived below from Assumption~\ref{ass:G}.

\begin{lemma}[Structure of $G$]\label{lem:Gstructure}
	Let $G$ satisfy Assumption~\ref{ass:G}. Then for every $s\in\mathbb{R}$:
	\begin{enumerate}
		\item[(i)] $c_G|s|\le |G(s)|\le C_G|s|$, and $G(s)$ has the same sign as
		$s$; equivalently,
		\[
		c_G\,s^2 \le G(s)\,s \le C_G\,s^2,\quad G(s)\,\mathrm{sgn}(s)\ge c_G|s|.\]
		\item[(ii)] Let $F$ satisfy
		Assumption~\ref{ass:F}(i),(iii). Set $C_3:=c_1c_3+c_2\ge0$. Then
		\[
		G(s)\,F'(s)\;\ge\;-C_GC_3 \qquad\forall s\in\mathbb{R}.
		\]

\item[(iii)] If $v\in
H^1(\O^\e_p)$, then $G(v)\in H^1(\O^\e_p)$, with
\[
\nabla\bigl(G(v)\bigr)=G'(v)\,\nabla v \qquad \text{a.e.\ in }\O^\e_p.
\]
Moreover $\|G(v)\|_{L^2(\O^\e_p)}\le C_G\|v\|_{L^2(\O^\e_p)}$
and $\|\nabla G(v)\|_{L^2(\O^\e_p)}\le C_G\|\nabla v\|_{L^2(\O^\e_p)}$.
	\end{enumerate}
\end{lemma}

\begin{proof}
	\textbf{(i)} Since $G(0)=0$ and $c_G\le G'\le C_G$, the fundamental theorem
	gives $G(s)=\int_0^s G'(r)\,dr$, so $c_G|s|\le|G(s)|\le C_G|s|$ and
	$\mathrm{sgn}(G(s))=\mathrm{sgn}(s)$ for all $s$. Hence
	$G(s)s=|G(s)|\,|s|\in[c_Gs^2,C_Gs^2]$ and
	$G(s)\,\mathrm{sgn}(s)=|G(s)|\ge c_G|s|$.
	
	\smallskip
	\textbf{(ii)} By Assumption~\ref{ass:F}(i),(iii),
	$sF'(s)\ge c_1F(s)-c_2\ge -c_1c_3-c_2=-C_3$ for all $s$.
	Write $G(s)=\theta(s)s$ where $\theta(s):=G(s)/s\in[c_G,C_G]$ for $s\ne0$
	(by the mean-value theorem) and $\theta(0):=G'(0)$. Then
	\[
	G(s)F'(s)=\theta(s)\,(sF'(s)).
	\]
	If $sF'(s)\ge0$, multiply the bound $sF'(s)\ge-C_3$ by $\theta(s)\ge c_G>0$
	to get $G(s)F'(s)\ge-c_GC_3$. If $sF'(s)<0$, since $\theta(s)\le C_G$
	multiplying the negative quantity by the smaller factor gives the larger product,
	so $G(s)F'(s)\ge C_G(sF'(s))\ge-C_GC_3$. Since $C_G\ge c_G$, both cases
	give $G(s)F'(s)\ge-C_GC_3$.
	
	\smallskip
	\textbf{(iii)} The pointwise bound $|G(v)|\le C_G|v|$ gives
	$G(v)\in L^2(\O^\e_p)$. The chain rule $\nabla(G(v))=G'(v)\nabla v$
	holds in $H^1(\O^\e_p)$ by approximating $v$ by smooth functions,
	using $|G'|\le C_G$ and dominated convergence, and invoking the closedness
	of the weak gradient; see e.g.~\cite{ElliottZheng1986}. The stated
	$L^2$-bounds follow immediately from $|G'|\le C_G$.
\end{proof}

\begin{lemma}[Dissipation produced by the source term]\label{lem:Gdissipation-weak}
	Let $F$ satisfy Assumption~\ref{ass:F} and $G$ satisfy Assumption~\ref{ass:G},
	and let $(\phi_\varepsilon,\mu_\varepsilon)$ be a weak solution of
	\eqref{eq:strong} in the sense of Definition~\ref{def:weak}. Then for a.e.\
	$t\in(0,T)$,
	\begin{equation}\label{eq:Gdissipation-weak}
		\int_{\O^\e_p}G(\phi_\varepsilon(t))\,\mu_\varepsilon(t)\,dx
		\;\ge\;
		c_G\|\nabla\phi_\varepsilon(t)\|^2_{L^2(\O^\e_p)}
		-C_GC_3\,|\O^\e_p|,
	\end{equation}
	where $C_3:=c_1c_3+c_2\ge0$ is the constant from
	Lemma~\ref{lem:Gstructure}.
\end{lemma}
\begin{proof}
	Fix a.e.\ $t\in(0,T)$ for which $\phi_\varepsilon(t)\in H^1(\O^\e_p)$
	and \eqref{eq:weak2} holds.
	By Lemma~\ref{lem:Gstructure}(iii), $G(\phi_\varepsilon(t))\in H^1(\O^\e_p)$
	with $\|G(\phi_\varepsilon(t))\|_{H^1}\le C_G\|\phi_\varepsilon(t)\|_{H^1}$,
	so $\eta=G(\phi_\varepsilon(t))$ is an admissible test function in
	\eqref{eq:weak2}. Testing gives
	\begin{equation}\label{eq:test-step}
		\int_{\O^\e_p}\mu_\varepsilon(t)\,G(\phi_\varepsilon(t))\,dx
		=\int_{\O^\e_p}G'(\phi_\varepsilon(t))\,
		|\nabla\phi_\varepsilon(t)|^2\,dx
		+\int_{\O^\e_p}F'(\phi_\varepsilon(t))\,G(\phi_\varepsilon(t))\,dx,
	\end{equation}
	where we used $\nabla(G(\phi_\varepsilon))=G'(\phi_\varepsilon)\nabla\phi_\varepsilon$
	from Lemma~\ref{lem:Gstructure}(iii) (no boundary term appears since the
	weak formulation already encodes homogeneous Neumann conditions for $\phi_\varepsilon$).
	By Assumption~\ref{ass:G}, $G'(\phi_\varepsilon(t,x))\ge c_G$ a.e., so
	the first term on the right satisfies
	\[
	\int_{\O^\e_p}G'(\phi_\varepsilon(t))\,|\nabla\phi_\varepsilon(t)|^2\,dx
	\ge c_G\|\nabla\phi_\varepsilon(t)\|^2_{L^2(\O^\e_p)}.
	\]
	By Lemma~\ref{lem:Gstructure}(ii), $G(s)F'(s)\ge-C_GC_3$ for all $s$, so
	the second term satisfies
	$\int_{\O^\e_p}F'(\phi_\varepsilon(t))\,G(\phi_\varepsilon(t))\,dx
	\ge-C_GC_3|\O^\e_p|$.
	Substituting into \eqref{eq:test-step} gives \eqref{eq:Gdissipation-weak}.
\end{proof}

\subsection{The energy inequality}

Recall the energy functional
\[
E_\varepsilon(t):=\frac12\|\nabla\phi_\varepsilon(t)\|^2_{L^2(\O^\e_p)}
+\int_{\O^\e_p}F(\phi_\varepsilon(t))\,dx.
\]

Testing \eqref{eq:weak1} with $\zeta=\mu_\varepsilon$ and \eqref{eq:weak2} with
$\eta=\partial_t\phi_\varepsilon$ and adding gives, for a.e.\
$t\in(0,T)$,
\begin{equation}\label{eq:energyident}
	\frac{d}{dt}E_\varepsilon(t)+\|\nabla\mu_\varepsilon(t)\|^2_{L^2(\O^\e_p)}
	+\int_{\O^\e_p}G(\phi_\varepsilon)\,\mu_\varepsilon\,dx=0.
\end{equation}

\begin{remark}
	For a weak solution in the sense of Definition~\ref{def:weak}, $\eta=
	\partial_t\phi_\varepsilon(t)\in H^1(\O^\e_p)'$ is not a
	priori an admissible test function in \eqref{eq:weak2}. Identity
	\eqref{eq:energyident} is therefore understood as holding for the Galerkin
	approximations $(\phi_\varepsilon^{(n)},\mu_\varepsilon^{(n)})$ used to
	construct the weak solution (where $\eta=\partial_t\phi_\varepsilon^{(n)}$
	is finite-dimensional and hence admissible). The estimates below are derived
	at the level of these approximations; passing to the limit $n\to\infty$,
	weak lower semicontinuity of the norms $\|\nabla\phi_\varepsilon\|_{L^2}$,
	$\|\nabla\mu_\varepsilon\|_{L^2}$ and of $\int F(\phi_\varepsilon)$ (by
	Fatou's lemma, using $F$ bounded below) turns
	\eqref{eq:energyident} into the {inequality} \eqref{eq:energyineq} for
	the weak solution $(\phi_\varepsilon,\mu_\varepsilon)$ itself.
\end{remark}

Combining \eqref{eq:energyident} with Lemma~\ref{lem:Gdissipation-weak} and
$|\O^\e_p|\le|\Omega|$ yields, with $C_3=c_1c_3+c_2$ as in
Lemma~\ref{lem:Gstructure},
\begin{equation}\label{eq:energyineq}
	\frac{d}{dt}E_\varepsilon(t)+\|\nabla\mu_\varepsilon(t)\|^2_{L^2(\O^\e_p)}
	+c_G\|\nabla\phi_\varepsilon(t)\|^2_{L^2(\O^\e_p)}
	\;\le\; C_GC_3\,|\Omega|.
\end{equation}

As in the double-well case, the right-hand side of \eqref{eq:energyineq} is a
fixed constant independent of $\e$.

\noindent\textbf{Bound on the initial energy.} By Assumption~\ref{ass:init},
$\|\phi_\varepsilon^0\|_{H^1(\O^\e_p)}\le C$. By
Assumption~\ref{ass:F}, $|F'(s)|\le C(1+|s|^3)$, so integrating gives the
growth bound $|F(s)|\le C(1+|s|^4)$ for all $s\in\mathbb{R}$. Together with
the uniform Sobolev embedding $H^1(\O^\e_p)\hookrightarrow
L^4(\O^\e_p)$ (Remark~\ref{rem:frame}), this gives
\[
E_\varepsilon(0)=\frac12\|\nabla\phi_\varepsilon^0\|^2_{L^2(\O^\e_p)}
+\int_{\O^\e_p}F(\phi_\varepsilon^0)\,dx
\le C\bigl(1+\|\phi_\varepsilon^0\|^4_{L^4(\O^\e_p)}\bigr)\le C,
\]
with $C$ independent of $\varepsilon$.

\noindent\textbf{Integration in time.} Integrating \eqref{eq:energyineq} over
$(0,t)$, $t\in[0,T]$, and using $E_\varepsilon(0)\le C$,
\begin{equation}\label{eq:energy-int}
	E_\varepsilon(t)+\int_0^t\|\nabla\mu_\varepsilon\|^2_{L^2(\O^\e_p)}\,ds
	+c_G\int_0^t\|\nabla\phi_\varepsilon\|^2_{L^2(\O^\e_p)}\,ds
	\le E_\varepsilon(0)+C_GC_3|\Omega|\,T \le C(T).
\end{equation}
\noindent\textbf{From the energy to $H^1$ and $L^4$ bounds.} Unlike the
double-well case, $F$ is here only bounded below
(Assumption~\ref{ass:F}(i)), so $E_\varepsilon(t)\ge0$ no longer holds and
$E_\varepsilon(t)$ alone does not control $\|\nabla\phi_\varepsilon(t)\|^2_{L^2}$. Instead,
by Assumption~\ref{ass:F}(i), $F(s)\ge-c_3$, so
\[
\int_{\O^\e_p}F(\phi_\varepsilon(t))\,dx\ge -c_3|\O^\e_p|\ge-c_3|\Omega|,
\]
and hence
\[
\frac12\|\nabla\phi_\varepsilon(t)\|^2_{L^2(\O^\e_p)}
\le E_\varepsilon(t)+c_3|\Omega|
\stackrel{\eqref{eq:energy-int}}{\le} C(T)+c_3|\Omega|=:C_1.
\]
This gives, uniformly in $t\in[0,T]$ and $\varepsilon$:
\begin{equation}\label{eq:gradbound}
	\|\nabla\phi_\varepsilon\|_{L^\infty(0,T;L^2(\O^\e_p))}\le C_1.
\end{equation}

For the $L^4$-bound, we now use the coercivity assumption
\ref{ass:F}(ii), $F(s)\ge c_0|s|^4-C_0$, which gives
\[
c_0\int_{\O^\e_p}\phi_\varepsilon(t)^4\,dx
\le \int_{\O^\e_p}F(\phi_\varepsilon(t))\,dx+C_0|\O^\e_p|
\le E_\varepsilon(t)+C_0|\Omega|
\le C(T)+C_0|\Omega|.
\]
Hence
\begin{equation}\label{eq:L4bound}
	\|\phi_\varepsilon\|_{L^\infty(0,T;L^4(\O^\e_p))}\le C_2,
\end{equation}
with $C_2$ independent of $\varepsilon$.

Combining \eqref{eq:energy-int},
\eqref{eq:gradbound} and \eqref{eq:L4bound}, there exists $C(T)>0$,
independent of $\varepsilon$, such that
\begin{equation}\label{eq:uniformbound1}
	\sup_{t\in[0,T]}\Bigl[\|\nabla\phi_\varepsilon(t)\|^2_{L^2(\O^\e_p)}
	+\|\phi_\varepsilon(t)\|^4_{L^4(\O^\e_p)}\Bigr]
	+\int_0^T\|\nabla\mu_\varepsilon\|^2_{L^2(\O^\e_p)}\,dt
	+\int_0^T\|\nabla\phi_\varepsilon\|^2_{L^2(\O^\e_p)}\,dt
	\le C(T).
\end{equation}

\begin{lemma}[Uniform bound on the mean]\label{lem:meanbound}
	Let $\bar\phi_\varepsilon(t):=\dfrac{1}{|\O^\e_p|}\displaystyle\int_{\O^\e_p}\phi_\varepsilon(t,x)\,dx$.
	Then there exists $C>0$, independent of $\varepsilon$, such that
	\begin{equation}\label{eq:meanbound}
		\|\bar\phi_\varepsilon\|_{L^\infty(0,T)}\le C.
	\end{equation}
\end{lemma}

\begin{proof}
	Taking the mean of equation \eqref{eq:strong}$_1$ over $\O^\e_p$,
	and using
	$\int_{\O^\e_p}\Delta\mu_\varepsilon\,dx
	=\int_{\partial\O^\e_p}\nabla\mu_\varepsilon\cdot\bn_\varepsilon\,dS=0$
	(homogeneous Neumann condition on $\mu_\varepsilon$), we obtain
	\[
	\frac{d}{dt}\bar\phi_\varepsilon(t)+\overline{G(\phi_\varepsilon)}(t)=0,
	\qquad \overline{G(\phi_\varepsilon)}(t):=\frac{1}{|\O^\e_p|}\int_{\O^\e_p}G(\phi_\varepsilon(t,x))\,dx.
	\]
	Write $\overline{G(\phi_\varepsilon)}=G(\bar\phi_\varepsilon)+r_\varepsilon$,
	where
	\[
	r_\varepsilon(t):=\frac{1}{|\O^\e_p|}\int_{\O^\e_p}
	\bigl[G(\phi_\varepsilon)-G(\bar\phi_\varepsilon)\bigr]\,dx,
	\]
	so that
	\begin{equation}\label{eq:meanODE}
		\frac{d}{dt}\bar\phi_\varepsilon(t)+G(\bar\phi_\varepsilon(t))=-r_\varepsilon(t).
	\end{equation}
	Since $G$ is Lipschitz with constant $C_G$ and $\phi_\varepsilon-\bar\phi_\varepsilon$
	has zero mean, Cauchy--Schwarz and the Poincar\'e--Wirtinger inequality on
	$\O^\e_p$ (with constant $C_P$ independent of $\varepsilon$,
	Remark~\ref{rem:frame}) give
	\begin{equation}\label{eq:rbound}
		|r_\varepsilon(t)|
		\le \frac{C_G}{|\O^\e_p|^{1/2}}\|\phi_\varepsilon(t)-\bar\phi_\varepsilon(t)\|_{L^2(\O^\e_p)}
		\le \frac{C_GC_P}{|\O^\e_p|^{1/2}}\|\nabla\phi_\varepsilon(t)\|_{L^2(\O^\e_p)}.
	\end{equation}
	Since $|\O^\e_p|\to\theta_p|\Omega|>0$, we have
	$|\O^\e_p|\ge\tfrac12\theta_p|\Omega|$ for $\varepsilon$ small
	enough; together with \eqref{eq:gradbound} this gives
	\[
	\sup_{t\in[0,T]}|r_\varepsilon(t)|\le C_1':=C\,C_1.
	\]
	By Lemma~\ref{lem:Gstructure}(ii), $G(s)\mathrm{sgn}(s)\ge c_G|s|$ for all
	$s\in\mathbb{R}$. Multiplying \eqref{eq:meanODE} by
	$\mathrm{sgn}(\bar\phi_\varepsilon(t))$ gives, for a.e.\ $t$,
	\[
	\frac{d}{dt}|\bar\phi_\varepsilon(t)| \le -c_G|\bar\phi_\varepsilon(t)|+|r_\varepsilon(t)|.
	\]
	By Gronwall's inequality, for all $t\in[0,T]$,
	\[
	|\bar\phi_\varepsilon(t)|
	\le |\bar\phi_\varepsilon(0)|\,e^{-c_Gt}
	+\int_0^te^{-c_G(t-s)}|r_\varepsilon(s)|\,ds
	\le |\bar\phi_\varepsilon(0)|+\frac{C_1'}{c_G}.
	\]
	Finally, by Cauchy--Schwarz and Assumption~\ref{ass:init},
	\[
	|\bar\phi_\varepsilon(0)|\le \frac{1}{|\O^\e_p|^{1/2}}\|\phi_\varepsilon^0\|_{L^2(\O^\e_p)}\le C.
	\]
	Combining the last two displays gives \eqref{eq:meanbound} with
	$C=C+C_1'/c_G$, independent of $\varepsilon$.
\end{proof}

\subsection{Existence for fixed \(\eps\)}

\begin{theorem}[Existence of microscopic weak solutions]
	\label{thm:existence}
	Let Assumptions \ref{ass:F}, \ref{ass:G}, and \ref{ass:init} hold. Then, for
	every \(\eps>0\), problem \eqref{eq:strong} admits a weak solution
	\((\phie,\mue)\) in the sense of Definition~\ref{def:weak}.
\end{theorem}

\begin{proof}[Sketch of the proof]
	For fixed $\varepsilon>0$, $\Omega_p^\varepsilon$ is a fixed bounded
	Lipschitz domain and \eqref{eq:strong} is the classical Cahn--Hilliard
	system with homogeneous Neumann conditions, regular potential $F$
	(Assumption~\ref{ass:F}), and globally Lipschitz source term $G$
	(Assumption~\ref{ass:G}). Existence of a global weak solution follows
	by the Faedo--Galerkin method: project onto finite-dimensional eigenspaces
	of the Neumann Laplacian, derive the uniform bound
	\eqref{eq:energy-estimate} at the Galerkin level (where
	\eqref{eq:energyident} holds exactly), and pass to the limit using
	weak-$\ast$ compactness in $L^\infty(0,T;H^1(\Omega_p^\varepsilon))$,
	weak compactness in $L^2(0,T;H^1(\Omega_p^\varepsilon))$, and strong
	compactness in $L^2(0,T;L^4(\Omega_p^\varepsilon))$ via the
	Aubin--Lions--Simon lemma ($d\le3$). The nonlinear terms
	$F'(\phi_\varepsilon^n)$ and $G(\phi_\varepsilon^n)$ converge strongly
	in $L^2(0,T;L^{4/3})$ and $L^2(0,T;L^2)$ respectively, by the growth
	bound on $F''$, the Lipschitz bound on $G$, and the $L^4$-strong
	convergence; see~\cite{Elliott1989,ElliottZheng1986,
		NicolaenkoScheurerTemam1989} for the full argument without the source
	term, which is handled identically.
\end{proof}

\subsection{Uniform estimates}

\begin{lemma}[Uniform energy estimate]
	\label{lem:energy}
	Let Assumptions \ref{ass:F}, \ref{ass:G}, and \ref{ass:init} hold, and let
	\((\phie,\mue)\) be a weak solution of \eqref{eq:strong} given by
	Theorem~\ref{thm:existence}. Then there exists \(C>0\), independent of
	\(\eps\), such that
	\begin{equation}
		\label{eq:energy-estimate}
		\norm{\phie}{L^\infty(0,T;H^1(\Omeps))}
		+
		\norm{\mue}{L^2(0,T;H^1(\Omeps))}
		+
		\norm{\partial_t\phie}{L^2(0,T;H^1(\Omeps)')}
		\leq C.
	\end{equation}
\end{lemma}

\begin{proof}
	\textbf{Step 1: $\phie\in L^\infty(0,T;H^1(\Omeps))$.}
	The gradient bound \eqref{eq:gradbound} gives
	$\|\nabla\phie\|_{L^\infty(0,T;L^2(\Omeps))}\le C_1$, and
	Lemma~\ref{lem:meanbound} gives $\|\bar\phie\|_{L^\infty(0,T)}\le C$.
	The Poincar\'e--Wirtinger inequality on $\Omeps$ (Remark~\ref{rem:frame}) then yields
	\[
	\|\phie(t)\|_{L^2(\Omeps)}
	\le C_P\|\nabla\phie(t)\|_{L^2(\Omeps)}+|\bar\phie(t)|\,|\Omeps|^{1/2}
	\le C,
	\]
	uniformly in $t$ and $\varepsilon$. Together with \eqref{eq:gradbound} this
	gives $\phie\in L^\infty(0,T;H^1(\Omeps))$.
	
	\smallskip
	\textbf{Step 2: $\mue\in L^2(0,T;H^1(\Omeps))$.}
	The gradient bound $\|\nabla\mue\|_{L^2(0,T;L^2(\Omeps))}\le C$ is already
	in \eqref{eq:uniformbound1}. For the mean, taking $\eta=1$ in \eqref{eq:weak2}
	gives $\bar\mue(t)=|\Omeps|^{-1}\int_{\Omeps}F'(\phie(t))\,dx$. By
	Assumption~\ref{ass:F} and the uniform $L^4$-bound \eqref{eq:L4bound},
	$|\bar\mue(t)|\le C(1+\|\phie(t)\|^3_{L^4(\Omeps)})\le C$
	uniformly in $t$ and $\varepsilon$.
	The Poincar\'e--Wirtinger inequality then gives
	$\|\mue\|_{L^2(0,T;L^2(\Omeps))}\le C$,
	hence $\mue\in L^2(0,T;H^1(\Omeps))$.
	
	\smallskip
	\textbf{Step 3: $\partial_t\phie\in L^2(0,T;H^1(\Omeps)')$.}
	For $\zeta\in H^1(\Omeps)$, equation \eqref{eq:weak1} gives
	\[
	\|\partial_t\phie(t)\|_{H^1(\Omeps)'}
	\le \|\nabla\mue(t)\|_{L^2(\Omeps)}+C_G\|\phie(t)\|_{L^2(\Omeps)}.
	\]
	Squaring, integrating over $(0,T)$, and using Steps~1--2 gives
	$\partial_t\phie\in L^2(0,T;H^1(\Omeps)')$, with norm bounded independently
	of $\varepsilon$. Combining Steps~1--3 gives \eqref{eq:energy-estimate}.
\end{proof}

\section{Two-scale limit and homogenized system}
\label{sec:homogenization}
\subsection{Unfolding operator for a perforated domain}\label{SSec61}
The main tool for homogenization in the perforated domain $\Omega_p^\varepsilon$
and for the derivation of the two-scale limit system is the periodic unfolding
operator adapted to perforated media, denoted by $\mathcal T_\varepsilon^\ast$.
It was first introduced in \cite{CDG02} and further developed in
\cite{CD06,CDG08,CDG08+}. For a detailed presentation we refer to
\cite[Chapter~4]{CDG18}.
We recall the definition of the periodic unfolding operators
$\mathcal T_\varepsilon$ and $\mathcal T_\varepsilon^\ast$ for functions defined
on $\Omega$ and $\Omega_p^\varepsilon$, respectively.
For \(x\in\Om\), write
\[
x=\eps\left[\frac{x}{\eps}\right]+\eps\left\{\frac{x}{\eps}\right\},
\]
where \([\cdot]\in\mathbb Z^d\) and \(\{\cdot\}\in Y\).

\begin{definition}
	\label{def:uo}
	Let $Q_T = (0,T)\times\Omega$. For every measurable function
	$\psi$ on $Q_T$ the unfolding operator
	$\mathcal T_\varepsilon : L^1(Q_T)\to L^1((0,T)\X\Omega\times Y)$ is
	defined by
	\[
	\mathcal T_\varepsilon(\psi) (t,x,y)
	\doteq
	\begin{cases}
		\psi \bigl(t,\varepsilon\bigl[\tfrac{x}{\varepsilon}\bigr] + \varepsilon y\bigr),
		& \text{for a.e. } (t,x,y)\in (0,T)\X\Omega_\varepsilon\times Y,\\[4pt]
		0, & \text{for a.e. } (t,x,y)\in (0,T)\X\Lambda_\varepsilon\times Y.
	\end{cases}
	\]
	For every measurable function $\psi$ on $(0,T)\X\Omega_p^\varepsilon$ the
	perforated unfolding operator
	$\mathcal T_\varepsilon^\ast : L^1((0,T)\X\Omega_p^\varepsilon)
	\to L^1((0,T)\X\Omega\times Y_p)$ is defined by
	\[
	\mathcal T_\varepsilon^\ast(\psi) (t,x,y)
	\doteq
	\begin{cases}
		\psi \bigl(t,\varepsilon\bigl[\tfrac{x}{\varepsilon}\bigr] + \varepsilon y\bigr),
		& \text{for a.e. } (t,x,y)\in (0,T)\X\Omega_\varepsilon\times Y_p,\\[4pt]
		0, & \text{for a.e. } (t,x,y)\in (0,T)\X\Lambda_\varepsilon\times Y_p.
	\end{cases}
	\]
\end{definition}

In particular, let $\psi$ be a measurable function defined on
$(0,T)\X\Omega_p^\varepsilon$, and let $\widetilde\psi$ denote its zero-extension
to $(0,T)\X\Omega$ (i.e.\ $\widetilde\psi=\psi$ on $(0,T)\X\Omega_p^\varepsilon$
and $\widetilde\psi=0$ on $(0,T)\X(\Omega\setminus\Omega_p^\varepsilon)$). Then
\[
\Te^\ast(\psi)=\Te(\widetilde\psi)_{|(0,T)\X\O\X Y_p}.
\]
\textbf{Unfolding criterion for integrals in $\Omega_p^\varepsilon$ (u.c.i.).}
For every $\psi\in L^1((0,T)\X\Omega_p^\varepsilon)$ one has
\begin{equation*}
	\int_{Q_T\times Y_p}
	\mathcal T_\varepsilon^\ast(\psi)(t,x,y)\,\md(y,x,t)
	=
	\int_{(0,T)\X\O_p^\varepsilon}\psi(t,x)\,\md(x,t)
	-
	\int_{(0,T)\X\Lambda_\varepsilon}\psi(t,x)\,\md(x,t).
\end{equation*}
Since $\Omega$ has Lipschitz boundary, we have $|\Lambda_\varepsilon|\to 0$ as
$\varepsilon\to 0$. Hence (for $\e$ independent $\psi$)
\[
\left|\int_{(0,T)\X\Lambda_\varepsilon}\psi(t,x)\,\md(x,t)\right|
\longrightarrow 0
\qquad\text{as }\varepsilon\to0,
\]
and therefore
\begin{equation}\label{EQUCI}
	\left|
	\int_{Q_T\times Y_p}
	\mathcal T_\varepsilon^\ast(\psi)(t,x,y)\,\md(y,x,t)
	-
	\int_{(0,T)\X\O_p^\varepsilon}\psi(t,x)\,\md(x,t)\right|
	\longrightarrow 0
	\qquad\text{as }\varepsilon\to0.
\end{equation}

If \(\{\psi_\varepsilon\}_\e\) is bounded in \(L^p((0,T)\times\Omega)\) for some \(p>1\), then
\begin{equation}\label{eq:UCI}
	\left|\int_{(0,T)\times\Lambda_\varepsilon}\psi_\varepsilon\,dxdt\right|
	\leq
	|\Lambda_\varepsilon|^{1/p'}\|\psi_\varepsilon\|_{L^p((0,T)\times\Omega)}
	\to0.
\end{equation}
Finally, the unfolding operators preserve products in the expected way: if
$u, v : (0,T)\X\Omega_p^\varepsilon\to\mathbb R$ are measurable, then
\[
\mathcal T_\varepsilon^\ast(u\,v)
= \mathcal T_\varepsilon^\ast(u)\,\mathcal T^\ast_\varepsilon(v)
\qquad\text{a.e. in } (0,T)\X\Omega\times Y_p.
\]
If in addition $H:\R\to\R$ is continuous, then
\[
\mathcal T_\varepsilon^\ast(H(u))
= H\bigl(\mathcal T_\varepsilon^\ast(u)\bigr)
\qquad\text{a.e. in } (0,T)\X\Omega^\varepsilon\times Y_p;
\]
if moreover $H(0)=0$, this identity extends to a.e.\ $(0,T)\X\Omega\times Y_p$,
since both sides vanish on $(0,T)\X\Lambda_\varepsilon\times Y_p$. This applies
in particular to $H=G$ by Assumption~\ref{ass:G}.
\begin{theorem}[Unfolding compactness with corrector decomposition]\label{thm:unfolding-compactness}
	Let $(v_\varepsilon)\subset H^1(\Omega)$ with $\sup_\varepsilon\|v_\varepsilon\|_{H^1(\Omega)}\le C$.
	Then there exist $v\in H^1(\Omega)$ and $v_1\in L^2(\Omega;H^1_{\rm per}(Y))$ such that, along a
	subsequence,
	\[
	v_\varepsilon\rightharpoonup v \text{ weakly in } H^1(\Omega),
	\qquad
	\mathcal T_\varepsilon(v_\varepsilon)\to  v \quad\text{strongly in } L^2(\Omega\times Y)	,
	\]
	\[
	\mathcal T_\varepsilon(\nabla v_\varepsilon)\rightharpoonup \nabla_x v+\nabla_y v_1
	\quad\text{weakly in } L^2(\Omega\times Y)^d.
	\]
\end{theorem}
\subsection{Compactness}
\begin{lemma}[Extension and consistency with the perforated unfolding operator]
	\label{lem:ext-consistency}
	Let
	\[
	\phi_\varepsilon\in
	L^\infty(0,T;H^1(\Omega_p^\varepsilon))
	\cap H^1(0,T;H^1(\Omega_p^\varepsilon)'),
	\qquad
	\mu_\varepsilon\in L^2(0,T;H^1(\Omega_p^\varepsilon))
	\]
	satisfy the uniform bounds \eqref{eq:energy-estimate}. Then there exist
	\[
	\widetilde\phi_\varepsilon\in
	L^\infty(0,T;H^1(\Omega))
	\cap H^1(0,T;H^1(\Omega)'),
	\qquad
	\widetilde\mu_\varepsilon\in L^2(0,T;H^1(\Omega)),
	\]
	with
	\[
	\widetilde\phi_\varepsilon=\phi_\varepsilon,
	\qquad
	\widetilde\mu_\varepsilon=\mu_\varepsilon
	\quad\text{a.e. in }(0,T)\times\Omega_p^\varepsilon,
	\]
	such that
	\[
	\begin{aligned}
		\|\widetilde\phi_\varepsilon\|_{L^\infty(0,T;H^1(\Omega))}
		+
		\|\partial_t\widetilde\phi_\varepsilon\|_{L^2(0,T;H^1(\Omega)')}
		&\le
		C\Bigl(
		\|\phi_\varepsilon\|_{L^\infty(0,T;H^1(\Omega_p^\varepsilon))}
		+
		\|\partial_t\phi_\varepsilon\|_{L^2(0,T;H^1(\Omega_p^\varepsilon)')}
		\Bigr),\\
		\|\widetilde\mu_\varepsilon\|_{L^2(0,T;H^1(\Omega))}
		&\le
		C\|\mu_\varepsilon\|_{L^2(0,T;H^1(\Omega_p^\varepsilon))},
	\end{aligned}
	\]
	with \(C>0\) independent of \(\varepsilon\). Moreover, for a.e.
	\((t,x,y)\in(0,T)\times\Omega\times Y_p\),
	\[
	\mathcal T_\varepsilon^\ast(\phi_\varepsilon)
	=
	\mathcal T_\varepsilon(\widetilde\phi_\varepsilon)
	\big|_{(0,T)\times\Omega\times Y_p},
	\qquad
	\mathcal T_\varepsilon^\ast(\nabla\phi_\varepsilon)
	=
	\mathcal T_\varepsilon(\nabla\widetilde\phi_\varepsilon)
	\big|_{(0,T)\times\Omega\times Y_p},
	\]
	and analogously,
	\[
	\mathcal T_\varepsilon^\ast(\mu_\varepsilon)
	=
	\mathcal T_\varepsilon(\widetilde\mu_\varepsilon)
	\big|_{(0,T)\times\Omega\times Y_p},
	\qquad
	\mathcal T_\varepsilon^\ast(\nabla\mu_\varepsilon)
	=
	\mathcal T_\varepsilon(\nabla\widetilde\mu_\varepsilon)
	\big|_{(0,T)\times\Omega\times Y_p}.
	\]
\end{lemma}

\begin{proof}
	Set
	\[
	V_\varepsilon:=H^1(\Omega_p^\varepsilon),
	\qquad
	V:=H^1(\Omega),
	\qquad
	H_\varepsilon:=L^2(\Omega_p^\varepsilon),
	\qquad
	H:=L^2(\Omega).
	\]
	Let
	\[
	J_\varepsilon:H_\varepsilon\to V_\varepsilon',
	\qquad
	J:H\to V'
	\]
	denote the canonical injections associated with the Gelfand triples
	\[
	V_\varepsilon\hookrightarrow H_\varepsilon\hookrightarrow V_\varepsilon',
	\qquad
	V\hookrightarrow H\hookrightarrow V'.
	\]
	
	Let $P_\varepsilon\in\mathcal L(V_\varepsilon,V)$ be the uniform extension operator from
	Remark~\ref{rem:frame}; see~\cite{AcerbiChiadoPiatDalMasoPercivale1992,
		oleinik1992mathematical}. Define its $L^2$-adjoint
	\[
	P_\varepsilon^*: L^2(\Omega)\to L^2(\Omega_p^\varepsilon)
	\]
	by the identity
	\[
	(P_\varepsilon v,\eta)_{H} = (v, P_\varepsilon^*\eta)_{H_\varepsilon}
	\qquad \forall v\in V_\varepsilon,\ \eta\in V,
	\]
	which holds by definition of the Hilbert-space adjoint. Since $P_\varepsilon$
	is constructed by reflection and local averaging near $\Gamma_s^\varepsilon$
	(see~\cite{AcerbiChiadoPiatDalMasoPercivale1992}), its adjoint $P_\varepsilon^*$
	maps $H^1(\Omega)$ to $H^1(\Omega_p^\varepsilon)$ with uniform bound
	\[
	\|P_\varepsilon^*\eta\|_{V_\varepsilon} \le C\|\eta\|_{V}
	\qquad \forall\eta\in V,
	\]
	with $C>0$ independent of $\varepsilon$. We set $P_{\varepsilon,\sharp}:=P_\varepsilon^*|_{H^1(\Omega)}$
	and define $F_\varepsilon:=(P_{\varepsilon,\sharp})^\ast:
	V_\varepsilon'\to V'$
	by
	\[
	\langle F_\varepsilon\xi,\eta\rangle_{V',V}
	:=
	\langle \xi,P_{\varepsilon,\sharp}\eta\rangle_{V_\varepsilon',V_\varepsilon}
	\qquad
	\forall \xi\in V_\varepsilon',\ \eta\in V.
	\]
	Then
	\[
	\|F_\varepsilon\xi\|_{V'}
	\le C\|\xi\|_{V_\varepsilon'}
	\qquad
	\forall \xi\in V_\varepsilon',
	\]
	with \(C>0\) independent of \(\varepsilon\).
	
	For a.e. \(t\in(0,T)\), define
	\[
	\widetilde\phi_\varepsilon(t):=P_\varepsilon\phi_\varepsilon(t),
	\qquad
	\widetilde\mu_\varepsilon(t):=P_\varepsilon\mu_\varepsilon(t).
	\]
	Since \(P_\varepsilon\) is bounded and linear, the maps
	\(t\mapsto\widetilde\phi_\varepsilon(t)\) and
	\(t\mapsto\widetilde\mu_\varepsilon(t)\) are Bochner measurable. Moreover,
	\[
	\|\widetilde\phi_\varepsilon(t)\|_{H^1(\Omega)}
	\le C\|\phi_\varepsilon(t)\|_{H^1(\Omega_p^\varepsilon)}
	\quad\text{for a.e. }t\in(0,T),
	\]
	and
	\[
	\|\widetilde\mu_\varepsilon(t)\|_{H^1(\Omega)}
	\le C\|\mu_\varepsilon(t)\|_{H^1(\Omega_p^\varepsilon)}
	\quad\text{for a.e. }t\in(0,T).
	\]
	Taking the essential supremum in time in the first estimate and integrating the second estimate
	over \((0,T)\), we obtain
	\[
	\|\widetilde\phi_\varepsilon\|_{L^\infty(0,T;H^1(\Omega))}
	\le
	C\|\phi_\varepsilon\|_{L^\infty(0,T;H^1(\Omega_p^\varepsilon))},
	\]
	and
	\[
	\|\widetilde\mu_\varepsilon\|_{L^2(0,T;H^1(\Omega))}
	\le
	C\|\mu_\varepsilon\|_{L^2(0,T;H^1(\Omega_p^\varepsilon))}.
	\]
	
	\noindent
	We now identify the time derivative of the extended phase field. We claim that
	\[
	\partial_t\widetilde\phi_\varepsilon
	=
	F_\varepsilon(\partial_t\phi_\varepsilon)
	\quad\text{in }\mathcal D'(0,T;V').
	\]
	Let \(\psi\in C_c^\infty(0,T;V)\). Since \(P_{\varepsilon,\sharp}\) is independent of time,
	\[
	P_{\varepsilon,\sharp}\psi\in C_c^\infty(0,T;V_\varepsilon),
	\qquad
	\partial_t(P_{\varepsilon,\sharp}\psi)
	=
	P_{\varepsilon,\sharp}(\partial_t\psi).
	\]
	Using the definition of \(F_\varepsilon\), the weak time-derivative identity for
	\(\phi_\varepsilon\), and the defining relation of \(P_{\varepsilon,\sharp}\), we obtain
	\begin{align*}
		\int_0^T
		\langle F_\varepsilon\partial_t\phi_\varepsilon(t),\psi(t)\rangle_{V',V}\,dt
		&=
		\int_0^T
		\langle \partial_t\phi_\varepsilon(t),
		P_{\varepsilon,\sharp}\psi(t)\rangle_{V_\varepsilon',V_\varepsilon}\,dt=
		-\int_0^T
		(\phi_\varepsilon(t),
		\partial_t(P_{\varepsilon,\sharp}\psi)(t))_{H_\varepsilon}\,dt\\
		&=
		-\int_0^T
		(\phi_\varepsilon(t),
		P_{\varepsilon,\sharp}\partial_t\psi(t))_{H_\varepsilon}\,dt=
		-\int_0^T
		(P_\varepsilon\phi_\varepsilon(t),
		\partial_t\psi(t))_{H}\,dt\\
		&=
		-\int_0^T
		(\widetilde\phi_\varepsilon(t),
		\partial_t\psi(t))_{H}\,dt.
	\end{align*}
	Hence
	\[
	\partial_t\widetilde\phi_\varepsilon
	=
	F_\varepsilon(\partial_t\phi_\varepsilon)
	\quad\text{in }\mathcal D'(0,T;V').
	\]
	Since
	\[
	\|F_\varepsilon\xi\|_{V'}
	\le C\|\xi\|_{V_\varepsilon'}
	\qquad
	\forall \xi\in V_\varepsilon',
	\]
	we get
	\[
	\|\partial_t\widetilde\phi_\varepsilon\|_{L^2(0,T;V')}
	\le
	C
	\|\partial_t\phi_\varepsilon\|_{L^2(0,T;V_\varepsilon')}.
	\]
	Therefore
	\[
	\widetilde\phi_\varepsilon\in
	L^\infty(0,T;H^1(\Omega))
	\cap H^1(0,T;H^1(\Omega)'),
	\]
	and the stated estimate for \(\widetilde\phi_\varepsilon\) follows.
	
	\noindent
It remains to prove the consistency with the perforated unfolding operator.
Since
\[
\widetilde\phi_\varepsilon=\phi_\varepsilon,
\qquad
\widetilde\mu_\varepsilon=\mu_\varepsilon
\quad\text{a.e. in }\Omega_p^\varepsilon,
\]
the locality of weak gradients gives the same identities for the gradients.
Hence, for \((t,x,y)\in(0,T)\times\Omega_\varepsilon\times Y_p\), where
\(\varepsilon[x/\varepsilon]+\varepsilon y\in\Omega_p^\varepsilon\), the definitions of
\(\mathcal T_\varepsilon\) and \(\mathcal T_\varepsilon^\ast\) yield
\[
\mathcal T_\varepsilon^\ast(\phi_\varepsilon)
=
\mathcal T_\varepsilon(\widetilde\phi_\varepsilon),
\qquad
\mathcal T_\varepsilon^\ast(\mu_\varepsilon)
=
\mathcal T_\varepsilon(\widetilde\mu_\varepsilon),
\]
and the corresponding identities with \(\nabla\phi_\varepsilon\) and
\(\nabla\mu_\varepsilon\). On
\((0,T)\times\Lambda_\varepsilon\times Y_p\), both sides vanish by definition.
Thus all four identities hold a.e. in \((0,T)\times\Omega\times Y_p\).
This completes the proof.
\end{proof}
\begin{lemma}[Compactness]
	\label{lem:compactness}
	Let $(\phie,\mue)$ satisfy \eqref{eq:energy-estimate}. Then, up to a
	subsequence, there exist
	\[
	\phiz\in L^\infty(0,T;H^1(\Om)),
	\quad
	\muz\in L^2(0,T;H^1(\Om)),
	\quad
	\phi_1,\mu_1\in L^2((0,T)\times\Om;H^1_{\rm per}(\Yp)),
	\]
	such that the following convergences hold: 
	\begin{equation}\label{eq:Con01}
		\begin{aligned}
			\mathcal T^\ast_\eps(\phie)&\to \phiz
			\quad&&\text{strongly in }L^2((0,T)\times\Om\times\Yp),\\
			\mathcal T^\ast_\eps(\mue)&\rightharpoonup \muz
			\quad&&\text{weakly in }L^2((0,T)\times\Om\times\Yp),\\
			\mathcal T^\ast_\eps(\nabla\phie)
			&\rightharpoonup
			\nabla_x\phiz+\nabla_y\phi_1
			\quad&&\text{weakly in }L^2((0,T)\times\Om\times\Yp)^d,\\
			\mathcal T^\ast_\eps(\nabla\mue)
			&\rightharpoonup
			\nabla_x\muz+\nabla_y\mu_1
			\quad&&\text{weakly in }L^2((0,T)\times\Om\times\Yp)^d.
		\end{aligned}
	\end{equation}
\end{lemma}

\begin{proof}
	The bounds $\|\phi_\varepsilon\|_{L^\infty(0,T;H^1(\Omega_p^\varepsilon))}$,
	$\|\mu_\varepsilon\|_{L^2(0,T;H^1(\Omega_p^\varepsilon))}$, and
	$\|\partial_t\phi_\varepsilon\|_{L^2(0,T;H^1(\Omega_p^\varepsilon)')}$
	are all bounded by $C$ independent of $\varepsilon$, by
	Lemma~\ref{lem:energy}.
	
	\textbf{Step 1: Extension.}
	Lemma~\ref{lem:ext-consistency} provides extensions
	\[
	\widetilde\phi_\varepsilon\in L^\infty(0,T;H^1(\Omega))
	\cap H^1(0,T;H^1(\Omega)'),
	\quad
	\widetilde\mu_\varepsilon\in L^2(0,T;H^1(\Omega)),
	\]
	coinciding with $\phi_\varepsilon$ and $\mu_\varepsilon$ on
	$(0,T)\times\Omega_p^\varepsilon$, with all norms bounded by $C$
	independently of $\varepsilon$.
	
	\textbf{Step 2: Strong compactness for $\widetilde\phi_\varepsilon$.}
	The family $(\widetilde\phi_\varepsilon)_\varepsilon$ is bounded in
	$L^\infty(0,T;H^1(\Omega))\cap H^1(0,T;H^1(\Omega)')$.
	Since $H^1(\Omega)\Subset L^2(\Omega)\hookrightarrow H^1(\Omega)'$,
	the Aubin--Lions--Simon lemma gives, along a subsequence,
	\[
	\widetilde\phi_\varepsilon\to\phi
	\quad\text{strongly in }C([0,T];L^2(\Omega)),
	\qquad
	\widetilde\phi_\varepsilon\overset{\ast}{\rightharpoonup}\phi
	\quad\text{weakly-$\ast$ in }L^\infty(0,T;H^1(\Omega)),
	\]
	for some $\phi\in L^\infty(0,T;H^1(\Omega))$.
	
	\textbf{Step 3: Weak compactness for $\widetilde\mu_\varepsilon$.}
	Boundedness in $L^2(0,T;H^1(\Omega))$ gives, along a further subsequence,
	\[
	\widetilde\mu_\varepsilon\rightharpoonup\mu
	\quad\text{weakly in }L^2(0,T;H^1(\Omega))
	\]
	for some $\mu\in L^2(0,T;H^1(\Omega))$.
	
	\textbf{Step 4: Unfolding and corrector decomposition.}
	Since $\mathcal{T}_\varepsilon: H^1(\Omega)\to L^2(\Omega;H^1(Y))$ is
	bounded and independent of $t$, the families
	$(\mathcal{T}_\varepsilon(\widetilde\phi_\varepsilon))_\varepsilon$ and
	$(\mathcal{T}_\varepsilon(\widetilde\mu_\varepsilon))_\varepsilon$ are
	bounded in $L^\infty((0,T)\times\Omega;H^1(Y))$ and
	$L^2((0,T)\times\Omega;H^1(Y))$ respectively.
	Applying Theorem~\ref{thm:unfolding-compactness} with $t$ as a Fubini
	parameter yields $\phi_1,\mu_1\in L^2((0,T)\times\Omega;H^1_{\per,0}(Y))$
	and, along a single subsequence, the convergences 
		\begin{equation}\label{eq:Con02}
		\begin{aligned}
			\mathcal T_\varepsilon(\widetilde\phi_\varepsilon)&\to\phi
			\quad&&\text{strongly in }L^2((0,T)\times\Omega\times Y),\\
			\mathcal T_\varepsilon(\widetilde\mu_\varepsilon)&\rightharpoonup\mu
			\quad&&\text{weakly in }L^2((0,T)\times\Omega\times Y),\\
			\mathcal T_\varepsilon(\nabla\widetilde\phi_\varepsilon)
			&\rightharpoonup\nabla_x\phi+\nabla_y\phi_1
			\quad&&\text{weakly in }L^2((0,T)\times\Omega\times Y)^d,\\
			\mathcal T_\varepsilon(\nabla\widetilde\mu_\varepsilon)
			&\rightharpoonup\nabla_x\mu+\nabla_y\mu_1
			\quad&&\text{weakly in }L^2((0,T)\times\Omega\times Y)^d.
		\end{aligned}
	\end{equation}
	Strong convergence of $\mathcal{T}_\varepsilon(\widetilde\phi_\varepsilon)$
	follows from the decomposition
	\[
	\mathcal{T}_\varepsilon(\widetilde\phi_\varepsilon)-\phi
	=\mathcal{T}_\varepsilon(\widetilde\phi_\varepsilon-\phi)
	+(\mathcal{T}_\varepsilon\phi-\phi),
	\]
	where the first term vanishes by $L^2$-stability of $\mathcal{T}_\varepsilon$
	and Step~2, and the second by standard unfolding convergence for fixed
	$L^2$-functions.
	
	\textbf{Step 5: Restriction to $Y_p$.}
	By Lemma~\ref{lem:ext-consistency},
	$\mathcal{T}_\varepsilon^\ast(\phi_\varepsilon)
	=\mathcal{T}_\varepsilon(\widetilde\phi_\varepsilon)|_{(0,T)\times\Omega\times Y_p}$
	and analogously for $\nabla\phi_\varepsilon$ and $\mu_\varepsilon$.
	Restricting \eqref{eq:Con02} to $(0,T)\times\Omega\times Y_p$ gives
	\eqref{eq:Con01}, with $\phi_1,\mu_1$ identified as the restrictions to
	$Y_p$ of the $Y$-periodic correctors from Step~4.
\end{proof}

\subsection{Two-scale system}

The compactness established in Lemma~\ref{lem:compactness} identifies the limit fields
$(\phiz,\muz)$ and the correctors $(\phi_1,\mu_1)$, but does not yet determine their
relationship. We derive this by passing to the limit in the weak formulation
\eqref{eq:weak1}--\eqref{eq:weak2} against an appropriate family of test functions.

We set 
$$\GD(\zeta_0,\zeta_1)=\nabla \zeta_0+\nabla_y\zeta_1,\quad \forall\,(\zeta_0,\zeta_1)\in L^2(0,T;H^1(\Omega))\X L^2((0,T)\times\Omega;H^1_{\per,0}(Y_p)).$$
\begin{proposition}[Two-scale limit system]
	\label{prop:two-scale}
	Let $(\phiz,\muz,\phi_1,\mu_1)$ be the limit fields from Lemma~\ref{lem:compactness}.
	Then, for all $(\zeta_0,\eta_0)\in [L^2(0,T;H^1(\Omega))]^2$ and
	$(\zeta_1,\eta_1)\in [L^2((0,T)\times\Omega;H^1_{\per,0}(Y_p))]^2$,
	\begin{equation}
		\label{eq:two-scale-phi}
		\theta_p\int_0^T\langle\partial_t\phiz,\zeta_0\rangle\,dt
		+\int_0^T\int_\O\int_{Y_p}\GD(\mu,\mu_1)\cdot\GD(\zeta_0,\zeta_1)\,d(y,x,t)
		+\theta_p\int_0^T\int_\O
		G(\phiz)\zeta_0\,d(x,t)=0,
	\end{equation}
	\begin{equation}
		\label{eq:two-scale-mu}
		\theta_p\int_0^T\int_\O
		\muz\,\eta_0\,d(x,t)
		=\int_0^T\int_\O\int_{Y_p}
		\GD(\phiz,\phi_1)\cdot\GD(\eta_0,\eta_1)\,d(y,x,t)
		+\theta_p\int_0^T\int_\O
		F'(\phiz)\eta_0\,d(x,t).
	\end{equation}
\end{proposition}

\begin{proof}
We construct a recovery sequence of test functions and pass to the limit in
\eqref{eq:weak1}--\eqref{eq:weak2}.

\noindent
\textbf{Recovery sequence.} For smooth $\zeta_0\in C_c^\infty((0,T)\times\Omega)$ and
$\zeta_1\in C_c^\infty((0,T)\times\Omega;C^\infty_{\per}(Y_p))$ with mean zero over $Y_p$, define
\[
\zeta^\varepsilon(t,x):=\zeta_0(t,x)+\varepsilon\zeta_1\!\left(t,x,\tfrac{x}{\varepsilon}\right),
\qquad (t,x)\in(0,T)\times\Omega_p^\varepsilon.
\]
This is the restriction to $\Omega_p^\varepsilon$ of a function defined on $(0,T)\times\Omega$,
hence $\zeta^\varepsilon\in L^2(0,T;H^1(\Omega_p^\varepsilon))$ is an admissible test function
in \eqref{eq:weak1}. Since $\zeta_0$ is compactly supported in time,
$\zeta^\varepsilon$ vanishes near $t=0$ and $t=T$. By the strong consistency
of the unfolding operator,
\[
\begin{aligned}
	\mathcal{T}_\varepsilon^\ast(\partial_t\zeta^\varepsilon)&\to\partial_t\zeta_0
	\quad&&\text{strongly in }L^2((0,T)\times\Omega\times Y_p),\\
	\mathcal{T}_\varepsilon^\ast(\zeta^\varepsilon)&\to\zeta_0
	\quad&&\text{strongly in }L^2((0,T)\times\Omega\times Y_p),\\
	\mathcal{T}_\varepsilon^\ast(\nabla\zeta^\varepsilon)&\to\nabla_x\zeta_0+\nabla_y\zeta_1
	\quad&&\text{strongly in }L^2((0,T)\times\Omega\times Y_p)^d.
\end{aligned}
\]
An analogous recovery sequence $\eta^\varepsilon$ is constructed for the chemical potential
equation using $\eta_0$ and $\eta_1$.

\noindent
\textbf{Passage to the limit.} Insert $\zeta^\varepsilon$ into \eqref{eq:weak1} and
$\eta^\varepsilon$ into \eqref{eq:weak2}.

\noindent
Since $\zeta^\varepsilon$ vanishes at $t=0$ and $t=T$,
we integrate by parts in time:
\[
\int_0^T\langle\partial_t\phi_\varepsilon,\zeta^\varepsilon\rangle_{V_\varepsilon',V_\varepsilon}\,dt
=
-\int_0^T\int_{\Omega_p^\varepsilon}\phi_\varepsilon\,\partial_t\zeta^\varepsilon\,dx\,dt.
\]
Applying the unfolding criterion \eqref{EQUCI} to the right-hand side and
using the strong convergence $\mathcal{T}_\varepsilon^\ast(\phi_\varepsilon)\to\phi$
in $L^2((0,T)\times\Omega\times Y_p)$ from Lemma~\ref{lem:compactness}, together
with the strong convergence
$\mathcal{T}_\varepsilon^\ast(\partial_t\zeta^\varepsilon)\to\partial_t\zeta_0$,
we obtain
\[
-\int_0^T\int_{\Omega_p^\varepsilon}\phi_\varepsilon\,\partial_t\zeta^\varepsilon\,dx\,dt
\;\longrightarrow\;
-|Y_p|\int_0^T\int_\Omega\phi\,\partial_t\zeta_0\,dx\,dt
=
\theta_p\int_0^T\langle\partial_t\phi,\zeta_0\rangle_{V',V}\,dt,
\]
where the last equality uses integration by parts in time and the fact that
$\zeta_0$ is compactly supported in $(0,T)$.

\noindent
Apply the unfolding criterion \eqref{EQUCI}
to rewrite each remaining integral over $\Omega_p^\varepsilon$ as an integral over
$\Omega\times Y_p$, up to a remainder vanishing as $\varepsilon\to 0$.
Using the weak convergences of
$\mathcal T_\varepsilon^\ast(\nabla\phi_\varepsilon)$ and
$\mathcal T_\varepsilon^\ast(\nabla\mu_\varepsilon)$ from
\eqref{eq:Con01}, together with the strong convergence of the unfolded
gradients of the admissible test functions, we may pass to the limit in the
linear diffusion terms.

It remains to identify the nonlinear terms.
Set $u_\varepsilon:=\mathcal T_\varepsilon^\ast(\phi_\varepsilon)$.
By compactness,
\[
u_\varepsilon\to\phiz
\quad\text{strongly in }L^2((0,T)\times\Omega\times Y_p),
\]
and the uniform estimate \eqref{eq:energy-estimate} gives
\[
\{u_\varepsilon\}_\varepsilon
\quad\text{bounded in }L^\infty(0,T;L^4(\Omega\times Y_p)).
\]
Since $G$ is globally Lipschitz and $G(0)=0$, the definition of the
perforated unfolding operator gives
\[
\mathcal T_\varepsilon^\ast(G(\phi_\varepsilon))
=
G(u_\varepsilon)
\quad\text{a.e. in }(0,T)\times\Omega\times Y_p.
\]
Therefore,
\[
\mathcal T_\varepsilon^\ast(G(\phi_\varepsilon))
\to
G(\phiz)
\quad\text{strongly in }L^2((0,T)\times\Omega\times Y_p).
\]
For the potential term, Assumption~\ref{ass:F} implies
\[
|F'(a)-F'(b)|
\le
C(1+|a|^2+|b|^2)|a-b|,
\qquad a,b\in\mathbb R.
\]
Since $u_\varepsilon\to\phiz$ a.e.\ in $(0,T)\times\Omega\times Y_p$
(along a subsequence, from the $L^2$-strong convergence), and $F'$ is
continuous, we have $F'(u_\varepsilon)\to F'(\phiz)$ a.e. To upgrade this
to strong $L^2(0,T;L^{6/5})$ convergence we apply Vitali's convergence
theorem. The family $(F'(u_\varepsilon))_\varepsilon$ is equiintegrable in
$L^2(0,T;L^{6/5}(\Omega\times Y_p))$: indeed, using the growth bound and
H\"{o}lder's inequality,
\[
\|F'(u_\varepsilon)\|_{L^{6/5}(\Omega\times Y_p)}
\le
C\bigl(|\Omega\times Y_p|^{1/2}
+\|u_\varepsilon\|^3_{L^4(\Omega\times Y_p)}\bigr),
\]
and the right-hand side is bounded in $L^\infty(0,T)$ uniformly in
$\varepsilon$ by the $L^4$-bound. Vitali's theorem therefore gives
\[
F'(u_\varepsilon)\to F'(\phiz)
\quad\text{strongly in }
L^2(0,T;L^{6/5}(\Omega\times Y_p)).
\]
We now compare $F'(u_\varepsilon)$ with
$\mathcal T_\varepsilon^\ast(F'(\phi_\varepsilon))$. Since the unfolding
operator is set equal to zero on the boundary layer
$\Lambda_\varepsilon\times Y_p$, while $u_\varepsilon=0$ there, we have
\[
\mathcal T_\varepsilon^\ast(F'(\phi_\varepsilon))
=
F'(u_\varepsilon)
-
F'(0)\mathbf 1_{\Lambda_\varepsilon}
\quad\text{a.e. in }(0,T)\times\Omega\times Y_p.
\]
Hence
\begin{multline*}
	\|\mathcal T_\varepsilon^\ast(F'(\phi_\varepsilon))
	-F'(\phiz)\|_{L^2(0,T;L^{6/5}(\Omega\times Y_p))}\\
	\le
	\|F'(u_\varepsilon)-F'(\phiz)\|_{L^2(0,T;L^{6/5}(\Omega\times Y_p))}
	+
	|F'(0)|\,\|\mathbf 1_{\Lambda_\varepsilon}
	\|_{L^2(0,T;L^{6/5}(\Omega\times Y_p))}.
\end{multline*}
The first term tends to zero by the preceding argument. The second term also
tends to zero because $|\Lambda_\varepsilon|\to 0$. Therefore,
\[
\mathcal T_\varepsilon^\ast(F'(\phi_\varepsilon))
\to
F'(\phiz)
\quad\text{strongly in }
L^2(0,T;L^{6/5}(\Omega\times Y_p)).
\]
This convergence is sufficient to pass to the limit against smooth test
functions, since $H^1(\Omega)\hookrightarrow L^6(\Omega)$ for $d\le 3$.
Therefore, passing to the limit in the unfolded variational formulation gives
\eqref{eq:two-scale-phi}--\eqref{eq:two-scale-mu} for smooth test functions.

A density argument extends the identities to the full test function spaces.
\end{proof}

\subsection{Cell problems, correctors, and the homogenized tensor}
\label{ssec:cell}

Taking $\zeta_0=0$ in \eqref{eq:two-scale-phi} and $\eta_0=0$ in \eqref{eq:two-scale-mu}
and localizing to a.e.\ $(t,x)\in(0,T)\times\Omega$, the correctors $\phi_1$ and $\mu_1$
satisfy, for a.e.\ $(t,x)$,
\[
\int_{Y_p}\!(\nabla_x\muz+\nabla_y\mu_1)\cdot\nabla_y\zeta_1\,dy=0
\quad\forall\,\zeta_1\in H^1_{\per,0}(Y_p),
\]
\[
\int_{Y_p}\!(\nabla_x\phiz+\nabla_y\phi_1)\cdot\nabla_y\eta_1\,dy=0
\quad\forall\,\eta_1\in H^1_{\per,0}(Y_p).
\]
These are scalar Neumann cell problems on $Y_p$. For $i=1,\dots,d$, let
$\chi_i\in H^1_{\per,0}(Y_p)$ be the unique solution of
\begin{equation}
	\label{eq:cell-problem}
	\int_{Y_p}(\be_i+\nabla_y\chi_i)\cdot\nabla_y\psi\,dy=0
	\qquad\forall\,\psi\in H^1_{\per,0}(Y_p),
\end{equation}
or equivalently $-\Delta_y(y_i+\chi_i)=0$ in $Y_p$ with no-flux condition
$(\be_i+\nabla_y\chi_i)\cdot\bn_y=0$ on $\partial Y_s$ and $Y$-periodicity.
Existence and uniqueness follow from the Lax--Milgram theorem on $H^1_{\per,0}(Y_p)$.

\begin{lemma}[Corrector representation]
	\label{lem:corrector-representation}
	The two-scale correctors satisfy
	\begin{equation}
		\phi_1(t,x,y)=\sum_{i=1}^d\partial_{x_i}\phiz(t,x)\,\chi_i(y),
		\qquad
		\mu_1(t,x,y)=\sum_{i=1}^d\partial_{x_i}\muz(t,x)\,\chi_i(y).
	\end{equation}
\end{lemma}
\begin{proof}
	By linearity of the cell equation in the data $\nabla_x\phiz$ and $\nabla_x\muz$, the
	functions $\sum_i\partial_{x_i}\phiz\,\chi_i$ and $\sum_i\partial_{x_i}\muz\,\chi_i$
	satisfy the respective cell problems pointwise in $(t,x)$. Uniqueness in
	$H^1_{\per,0}(Y_p)$ gives the result.
\end{proof}

The homogenized diffusion tensor is defined by
\begin{equation}
	\label{eq:Bhom}
	\B^{\mathrm{hom}}_{ij}
	:=\frac{1}{|Y_p|}\int_{Y_p}(\be_i+\nabla_y\chi_i)\cdot(e_j+\nabla_y\chi_j)\,dy.
\end{equation}
Standard arguments show that $\B^{\mathrm{hom}}$ is symmetric and uniformly elliptic;
see e.g.\ \cite{CioranescuDonato1999}.

\subsection{Homogenized system}

\begin{theorem}[Qualitative homogenization]
	\label{thm:qual-hom}
	Let Assumptions~\ref{ass:F}, \ref{ass:G}, and~\ref{ass:init} hold. Then the microscopic
	solutions $(\phi_\varepsilon,\mu_\varepsilon)$ converge, up to a subsequence, to a pair
	$(\phiz,\muz)$ satisfying the homogenized system
	\begin{equation}
		\label{eq:hom-CH}
		\begin{cases}
			\partial_t\phiz-\operatorname{div}(\B^{\mathrm{hom}}\nabla\muz)+G(\phiz)=0
			&\text{in }(0,T)\times\Omega,\\[1mm]
			\muz=-\operatorname{div}(\B^{\mathrm{hom}}\nabla\phiz)+F'(\phiz)
			&\text{in }(0,T)\times\Omega,\\[1mm]
			\B^{\mathrm{hom}}\nabla\phiz\cdot \bn=\B^{\mathrm{hom}}\nabla\muz\cdot \bn=0
			&\text{on }(0,T)\times\partial\Omega,\\[1mm]
			\phiz(0,x)=\phi_0(x)&\text{in }\Omega.
		\end{cases}
	\end{equation}
\end{theorem}
\begin{proof}
	Substituting the corrector representation of Lemma~\ref{lem:corrector-representation} into
	the two-scale system \eqref{eq:two-scale-phi}--\eqref{eq:two-scale-mu} and integrating out
	the $y$-variable using the definition \eqref{eq:Bhom} of $\B^{\mathrm{hom}}$, one obtains,
	for all $\zeta_0\in L^2(0,T;H^1(\Omega))$,
	\[
	\int_0^T\langle\partial_t\phiz,\zeta_0\rangle\,dt
	+\int_{(0,T)\times\Omega}\B^{\mathrm{hom}}\nabla\muz\cdot\nabla\zeta_0\,d(x,t)
	+\int_{(0,T)\times\Omega}G(\phiz)\zeta_0\,d(x,t)=0,
	\]
	\[
	\int_{(0,T)\times\Omega}\muz\,\eta_0\,d(x,t)
	=\int_{(0,T)\times\Omega}\B^{\mathrm{hom}}\nabla\phiz\cdot\nabla\eta_0\,d(x,t)
	+\int_{(0,T)\times\Omega}F'(\phiz)\eta_0\,d(x,t),
	\]
	for all $\eta_0\in L^2(0,T;H^1(\Omega))$. Here we used that integrals over
	$\Omega\times Y_p$ of $y$-independent functions contribute a factor of $|Y_p|$, which
	is divided through to be absorbed into $\B^{\mathrm{hom}}$ via \eqref{eq:Bhom}. The boundary conditions follow
	from the natural boundary conditions encoded in the test function space $H^1(\Omega)$.
	
The initial condition $\phi(0)=\phi^0$ is identified as follows.
By Lemma~\ref{lem:compactness}, $\widetilde\phi_\varepsilon(0)=P_\varepsilon\phi^0_\varepsilon
\to\phi(0)$ strongly in $L^2(\Omega)$.
By the unfolding criterion \eqref{EQUCI}--\eqref{eq:UCI} and the strong convergence
$\mathcal{T}^\ast_\varepsilon(\phi^0_\varepsilon)\to\phi^0$ in $L^2(\Omega\times Y_p)$
from Assumption~\ref{ass:init}, for every $\psi\in L^2(\Omega)$,
\[
\int_{\Omega_p^\varepsilon}\phi^0_\varepsilon\,\psi\,dx
=
\frac{1}{|Y|}\int_{\Omega\times Y_p}
\mathcal{T}^\ast_\varepsilon(\phi^0_\varepsilon)\,\psi\,d(x,y)
+\int_{\Lambda_\e\cap\O^\e_p}\phi^0_\e\psi\,dx
\;\longrightarrow\;
\theta_p\int_\Omega\phi^0\,\psi\,dx,
\]
so $P_\varepsilon\phi^0_\varepsilon\rightharpoonup\phi^0$ weakly in $L^2(\Omega)$.
Since the strong limit is already $\phi(0)$, we conclude $\phi(0)=\phi^0$.
\end{proof}

\section{Conditional Quantitative rates}
\label{sec:quant}

\subsection{Scale-splitting operator}
\label{Sec:scale}

We recall the scale-splitting operator used in the quantitative estimates.
The construction follows the fixed-domain scale-splitting operator of
\cite[Section~1.6]{CDG18}. Since the homogenized fields are defined on the
fixed macroscopic domain $\Omega$, we apply the scale-splitting operator to
functions on $\Omega$ and then restrict the resulting functions to
$\O^\e_p$. Thus the operator itself is not intrinsic to the
perforated domain. The perforated geometry enters through the restriction to
$\O^\e_p$, the unfolding onto $\Omega\times Y_p$, and the
cell correctors defined on $Y_p$.

\noindent\textbf{Full-cell averages.}
Let $Y=(0,1)^d$. Since $\Omega$ is Lipschitz, we fix a bounded extension
operator
\[
\mathcal E_\Omega:H^m(\Omega)\to H^m(\mathbb R^d),
\qquad m=0,1,2,
\]
where the value of $m$ is chosen according to the regularity of the function
under consideration. For $v\in L^1_{\mathrm{loc}}(\mathbb R^d)$ and
$\xi\in\mathbb Z^d$, define the local full-cell average by
\[
\mathcal M_\varepsilon v(\varepsilon\xi)
:=
\frac{1}{\varepsilon^d |Y|}
\int_{\varepsilon(\xi+Y)} v(z)\,dz .
\]
The normalization is the usual average over the full cell
$\varepsilon(\xi+Y)$.

\noindent\textbf{The $Q_1$ scale-splitting operator.}
Let $\{q_\kappa\}_{\kappa\in\{0,1\}^d}$ be the standard $Q_1$ shape
functions on $Y$, namely
\[
q_\kappa(y)
=
\prod_{\ell=1}^d
y_\ell^{\kappa_\ell}(1-y_\ell)^{1-\kappa_\ell},
\qquad y\in Y.
\]
For $v\in L^1_{\mathrm{loc}}(\mathbb R^d)$, define
\[
\widetilde{\mathcal Q}_\varepsilon v(x)
:=
\sum_{\kappa\in\{0,1\}^d}
\mathcal M_\varepsilon v
\left(
\varepsilon\left\lfloor\frac{x}{\varepsilon}\right\rfloor_Y
+\varepsilon\kappa
\right)
q_\kappa
\left(
\left\{\frac{x}{\varepsilon}\right\}_Y
\right),
\qquad x\in\mathbb R^d .
\]
For a function $v$ defined on $\Omega$, we set
\[
\mathcal Q_\varepsilon v
:=
\left.
\widetilde{\mathcal Q}_\varepsilon(\mathcal E_\Omega v)
\right|_{\Omega}.
\]
When the function is used on the perforated domain, we use the same notation
for its restriction:
\[
\mathcal Q_\varepsilon v
:=
\left.
\widetilde{\mathcal Q}_\varepsilon(\mathcal E_\Omega v)
\right|_{\O^\e_p}.
\]
By construction, $\mathcal Q_\varepsilon v$ is separately affine on each
full cell $\varepsilon(\xi+Y)$. It is important to note that this does not
mean that its gradient is constant on each cell. In dimensions $d\geq2$, a
$Q_1$-function generally has a gradient which still depends on the cell
variable.

\noindent\textbf{Properties.}
For  $\e>0$, the scale-splitting operator 
$$\Qc_\e\,:\, H^1(\O) \to W^{1,\infty}(\O),$$
satisfies the following estimates, with constants
$C>0$ that may depend on $\Omega$,
$Y$, $d$, and the chosen extension operator, but not on $\varepsilon$.

\begin{enumerate}[label=\textup{(P\arabic*)}]
	\item[(P1)] \textit{(Stability in $H^1$.)}
	For every $v\in H^1(\Omega)$,
	\begin{equation}\label{eq:Q-H1-stability}
		\begin{aligned}
			\|\mathcal Q_\varepsilon v\|_{L^2(\O^\e_p)}
			\leq
			C\|v\|_{L^2(\Omega)} ,\quad 	\|\nabla \mathcal Q_\varepsilon v\|_{L^2(\O^\e_p)}
		&\leq
			C\|\nabla v\|_{L^2(\Omega)} ,\\
			\|v-\mathcal Q_\varepsilon v\|_{L^2(\O^\e_p)}
			&\leq
			C\varepsilon\|\nabla v\|_{L^2(\Omega)} .
		\end{aligned}
	\end{equation}
	
	\item[(P2)] \textit{(Unfolded convergence.)}
	If $v\in H^1(\Omega)$, then
	\begin{equation}\label{eq:Q-unfold-L2}
		\mathcal T_\varepsilon^\ast(\mathcal Q_\varepsilon v)
		\to v
		\quad\text{strongly in }L^2(\Omega\times Y_p),\quad 
		\mathcal T_\varepsilon^\ast(\nabla\mathcal Q_\varepsilon v)
		\to \nabla v
		\quad\text{strongly in }L^2(\Omega\times Y_p)^d .
	\end{equation}
\item[(P3)] \emph{(Gradient product estimate.)} For every $v\in H^1(\Omega)$ and every
$\psi\in L^2_{\per}(Y)$,
\begin{equation}
	\left\| \nabla \Qc_\varepsilon(v)\,\psi\!\left(\frac{x}{\varepsilon}\right) \right\|_{L^2(\Omega_p^\varepsilon)}
	\;\le\; C\,\|\nabla v\|_{L^2(\Omega)}\,\|\psi\|_{L^2(Y)}.
	\label{eq:P5}
\end{equation}
\end{enumerate}
This is the gradient analogue of the product estimate for $\Qc_\varepsilon$
(Proposition~3.2 of~\cite{Griso2004}), and follows from the same cell-by-cell
argument. On each cell $\varepsilon(\xi+Y)$, $\nabla_x\Qc_\varepsilon(v)$ is
expressed via the $Q_1$ nodal formula in terms of differences of local averages
$\mathcal{M}_\varepsilon(v)$ at adjacent nodes divided by $\varepsilon$; each
difference quotient is bounded by $\|\nabla v\|_{L^2}$ on the union of
neighbouring cells via the Poincar\'e--Wirtinger inequality. Squaring,
multiplying by $|\psi(x/\varepsilon)|^2$, integrating over each cell, and
summing over $\xi$ yields \eqref{eq:P5}; see~\cite[Theorem~3.4]{Griso2004}
for the analogous computation. Cells intersecting $\partial\Omega$ are handled
via $\Qc_\varepsilon(v)=\widetilde{\Qc}_\varepsilon(\mathcal{E}_\Omega v)$,
with the Poincar\'e--Wirtinger inequality applied on neighbouring cells in
$\mathbb{R}^d$ and the bound absorbed into $C\|\nabla v\|_{L^2(\Omega)}$
by boundedness of $\mathcal{E}_\Omega:H^1(\Omega)\to H^1(\mathbb{R}^d)$.

The operator $\mathcal Q_\varepsilon$ will be applied to macroscopic
coefficients such as $\partial_{x_i}\phi$ and $\partial_{x_i}\mu$. In the
corrector construction, the regularized coefficients
$\mathcal Q_\varepsilon(\partial_{x_i}\phi)$ and
$\mathcal Q_\varepsilon(\partial_{x_i}\mu)$ are restricted to
$\O^\e_p$ and multiplied by pore-cell correctors
$\chi_i(x/\varepsilon)$. Thus the macroscopic regularization is performed on
the fixed domain $\Omega$, while the perforated geometry is encoded through
the restriction to $\O^\e_p$ and the correctors on $Y_p$.

\subsection{Corrector approximations}
\label{subsec:corrector-approx}

We now define the first-order corrector approximations used in the
quantitative estimates. From this point on, the rate result is conditional
on additional regularity of the homogenized solution. This regularity is not
a consequence of the weak compactness theory of Section~\ref{sec:homogenization};
it is an additional assumption needed to justify the quantitative estimates.
\begin{assumption}[Quantitative regularity assumption]
	\label{ass:quant-reg}
	We assume that the homogenized solution satisfies
	\begin{equation}\label{eq:quant-reg}
		\phi \in L^\infty(0,T;H^2(\Omega))
		\cap H^1(0,T;H^1(\Omega)),
		\qquad
		\mu \in L^2(0,T;H^2(\Omega)).
	\end{equation}
	In addition, we assume $\phi\in C([0,T];H^2(\Omega))$, so that
	$\phi(0)\in H^2(\Omega)$ is well-defined and $Q_\varepsilon(\partial_{x_i}\phi(0))$
	is meaningful. In particular,
	\[
	\partial_{x_i}\phi\in L^\infty(0,T;H^1(\Omega)),
	\qquad
	\partial_{x_i}\mu\in L^2(0,T;H^1(\Omega)),
	\qquad i=1,\ldots,d.
	\]
	The constants in the estimates below may depend on the norm of
	$(\phi,\mu)$ in the spaces appearing in \eqref{eq:quant-reg}, but not
	on $\varepsilon$.
\end{assumption}

\begin{remark}\label{rem:quant-reg}
	Assumption~\ref{ass:quant-reg} is an additional hypothesis on the
	homogenized solution $(\phiz,\muz)$ of~\eqref{eq:hom-CH}; it is
	\emph{not} a consequence of the weak compactness theory of
	Section~\ref{sec:homogenization}. Such assumptions are standard in
	quantitative homogenization: see Bensoussan--Lions--Papanicolaou~\cite{BLP1978},
	Griso~\cite{Griso2004,Griso2006}, and Kenig--Lin--Shen~\cite{KenigLinShen2012}.
	Establishing~\eqref{eq:quant-reg} rigorously would require compatible
	initial data $\phi^0\in H^2(\Omega)$, suitable geometry on $\Omega$
	(convexity or $C^{1,1}$ boundary), and elliptic regularity for
	$-\operatorname{div}(\B^{\mathrm{hom}}\nabla\,\cdot)$, which depends
	on the smoothness of $\partial Y_s$; see~\cite{CDG18,CioranescuDonato1999}.
	We therefore take~\eqref{eq:quant-reg} as a standing hypothesis
	throughout Section~\ref{sec:quant}.
\end{remark}

Throughout this subsection, if a function is defined on \(\Omega\), we use
the same notation for its ordinary restriction to \(\Omega_p^\varepsilon\).
The scale-splitting operator is applied on the fixed domain
\(\Omega\), and the resulting functions are then restricted to
\(\Omega_p^\varepsilon\).

\noindent\textbf{Extension of the cell correctors.}
Let \(\chi_i\in H^1_{\per,0}(Y_p)\), \(i=1,\ldots,d\), be the cell
correctors from Section~\ref{ssec:cell}. Since \(Y_s\Subset Y\) and \(Y_p\)
is Lipschitz, each \(\chi_i\) admits an \(H^1\)-extension to \(Y\). We fix
periodic extensions, still denoted by \(\chi_i\), such that
\[
\chi_i\in H^1_\per(Y),
\qquad
\chi_i|_{Y_p}\ \text{is the original pore-cell corrector}.
\]
This convention makes \(\chi_i(x/\varepsilon)\) well defined also in the
boundary layer of incomplete cells. On complete pore cells, the extension
agrees with the original corrector on \(Y_p\).

\noindent\textbf{Product estimates.}
Besides the gradient product estimate \eqref{eq:P5}, we shall use the
standard scale-splitting product estimate form \cite[Proposition~3.2]{Griso2004}
\begin{equation}\label{eq:Q-product-estimate}
	\left\|
	\Qc_\varepsilon v\,
	\psi\!\left(\frac{x}{\varepsilon}\right)
	\right\|_{L^2(\Omega_p^\varepsilon)}
	\leq
	C
	\|v\|_{L^2(\Omega)}
	\|\psi\|_{L^2(Y)}
\end{equation}
for every \(v\in L^2(\Omega)\) and every
\(\psi\in L^2_\per(Y)\). The constant \(C\) is independent of
\(\varepsilon\). This estimate is obtained by the same cell-by-cell argument
as \eqref{eq:P5}.

\noindent\textbf{Definition of the scalar correctors.}
Let \(\Qc_\varepsilon\) be the fixed-domain scale-splitting operator from
Section~\ref{Sec:scale}. We define, for
\((t,x)\in(0,T)\times\Omega_p^\varepsilon\),
\begin{align}
	\label{eq:Phiapp}
	\Phi^\varepsilon(t,x)
	&:=
	\phiz(t,x)
	+
	\varepsilon
	\sum_{i=1}^d
	\Qc_\varepsilon(\partial_{x_i}\phiz)(t,x)
	\,\chi_i\!\left(\frac{x}{\varepsilon}\right),
	\\
	\label{eq:Muapp}
	M^\varepsilon(t,x)
	&:=
	\muz(t,x)
	+
	\varepsilon
	\sum_{i=1}^d
	\Qc_\varepsilon(\partial_{x_i}\muz)(t,x)
	\,\chi_i\!\left(\frac{x}{\varepsilon}\right).
\end{align}
Note that $\Phi_\varepsilon$ is not required to satisfy any boundary condition
on $\partial\Omega$: in the variational framework employed here, the boundary
conditions on $\partial\Omega$ enter through the homogenized conormal condition
$\B^{\mathrm{hom}}\nabla\phi\cdot \bn = 0$, which is used in the integration-by-parts
argument of Lemma~\ref{lem:elliptic-consistency}. 


\noindent\textbf{Regularity of the scalar correctors.}
Since \(\partial_{x_i}\phiz\in L^\infty(0,T;H^1(\Omega))\), the stability
of \(\Qc_\varepsilon\) gives
\[
\Qc_\varepsilon(\partial_{x_i}\phiz)
\in L^\infty(0,T;H^1(\Omega_p^\varepsilon)).
\]
Similarly,
\[
\Qc_\varepsilon(\partial_{x_i}\muz)
\in L^2(0,T;H^1(\Omega_p^\varepsilon)).
\]
The product rule gives, in the weak sense on \(\Omega_p^\varepsilon\),
\begin{equation}\label{eq:grad-corrector}
	\nabla\!\left(
	\varepsilon
	\Qc_\varepsilon(\partial_{x_i}\phiz)
	\chi_i\!\left(\frac{x}{\varepsilon}\right)
	\right)
	=
	\varepsilon
	\nabla\Qc_\varepsilon(\partial_{x_i}\phiz)
	\chi_i\!\left(\frac{x}{\varepsilon}\right)
	+
	\Qc_\varepsilon(\partial_{x_i}\phiz)
	\nabla_y\chi_i\!\left(\frac{x}{\varepsilon}\right).
\end{equation}
By \eqref{eq:P5},
\[
\left\|
\nabla\Qc_\varepsilon(\partial_{x_i}\phiz)
\chi_i\!\left(\frac{x}{\varepsilon}\right)
\right\|_{L^2(\Omega_p^\varepsilon)}
\leq
C
\|\nabla\partial_{x_i}\phiz\|_{L^2(\Omega)}
\|\chi_i\|_{L^2(Y)}.
\]
By \eqref{eq:Q-product-estimate},
\[
\left\|
\Qc_\varepsilon(\partial_{x_i}\phiz)
\nabla_y\chi_i\!\left(\frac{x}{\varepsilon}\right)
\right\|_{L^2(\Omega_p^\varepsilon)}
\leq
C
\|\partial_{x_i}\phiz\|_{L^2(\Omega)}
\|\nabla_y\chi_i\|_{L^2(Y)}.
\]
Together with \eqref{eq:quant-reg}, this yields
\begin{equation}\label{eq:corrector-H1-bounds}
	\|\Phi^\varepsilon\|_{L^\infty(0,T;H^1(\Omega_p^\varepsilon))}
	\leq
	C\|\phiz\|_{L^\infty(0,T;H^2(\Omega))},
	\qquad
	\|M^\varepsilon\|_{L^2(0,T;H^1(\Omega_p^\varepsilon))}
	\leq
	C\|\muz\|_{L^2(0,T;H^2(\Omega))}.
\end{equation}

Moreover, by \eqref{eq:Q-product-estimate},
\begin{equation}\label{eq:corrector-L2-approx}
	\|\Phi^\varepsilon-\phiz\|_{L^\infty(0,T;L^2(\Omega_p^\varepsilon))}
	\leq
	C\varepsilon
	\|\phiz\|_{L^\infty(0,T;H^1(\Omega))},
	\quad
	\|M^\varepsilon-\muz\|_{L^2(0,T;L^2(\Omega_p^\varepsilon))}
	\leq
	C\varepsilon
	\|\muz\|_{L^2(0,T;H^1(\Omega))}.
\end{equation}

The time derivative of \(\Phi^\varepsilon\) is understood in the distributional
sense in time. Since \(\partial_t\phiz\in L^2(0,T;H^1(\Omega))\) and
\(\Qc_\varepsilon\) is linear in the spatial variable,
\begin{equation}\label{eq:time-derivative-Phi}
	\partial_t\Phi^\varepsilon
	=
	\partial_t\phiz
	+
	\varepsilon
	\sum_{i=1}^d
	\Qc_\varepsilon(\partial_{x_i}\partial_t\phiz)
	\chi_i\!\left(\frac{x}{\varepsilon}\right)
	\quad
	\text{in }L^2(0,T;L^2(\Omega_p^\varepsilon)).
\end{equation}
Furthermore,
\begin{equation}\label{eq:time-corrector-bound}
	\|\partial_t(\Phi^\varepsilon-\phiz)\|_{L^2(0,T;L^2(\Omega_p^\varepsilon))}
	\leq
	C\varepsilon
	\|\partial_t\phiz\|_{L^2(0,T;H^1(\Omega))}.
\end{equation}

\noindent\textbf{Auxiliary Neumann-compatible fluxes.}
For \(w\in H^2(\Omega)\), define
\begin{equation}\label{eq:aux-flux}
	\mathcal J_w^\varepsilon
	:=
	\sum_{i=1}^d
	\Qc_\varepsilon(\partial_{x_i}w)
	\left(
	\be_i+\nabla_y\chi_i\!\left(\frac{x}{\varepsilon}\right)
	\right)
	\quad\text{in }\Omega_p^\varepsilon .
\end{equation}
For \(w=\phiz(t)\) and \(w=\muz(t)\), this flux belongs to
\(L^2(\Omega_p^\varepsilon)^d\) for a.e. \(t\). 
The no-flux property is understood in the following weak sense:
for every $\psi\in H^1(\Omega_p^\varepsilon)$,
\begin{equation}
	\label{eq:aux-flux-neumann}
	\int_{\Omega_p^\varepsilon} J_{w}^\varepsilon\cdot\nabla\psi\,dx
	+
	\int_{\Omega_p^\varepsilon}\psi\,\mathrm{div}\,J_{w}^\varepsilon\,dx
	=0,
\end{equation}
which follows from the cell boundary condition
$(\mathbf{e}_i+\nabla_y\chi_i)\cdot\mathbf{n}_y=0$ on $\partial Y_s$
in weak form~\eqref{eq:cell-problem}, after rescaling to $\Omega_p^\varepsilon$
and summing over cells. Since $b_i=\mathbf{e}_i+\nabla_y\chi_i\in L^2(Y_p)^d$
only, no pointwise normal trace is available; the identity
\eqref{eq:aux-flux-neumann} is the correct variational substitute.

\begin{lemma}[Approximation properties of the correctors]
	\label{lem:corrector-approx}
	Let \(d\in\{2,3\}\). Under the regularity assumption
	\eqref{eq:quant-reg}, we have
	\begin{equation}
		\label{eq:Q-approx-phi}
		\|\phiz-\Qc_\varepsilon\phiz\|_{L^2(0,T;L^2(\Omega_p^\varepsilon))}
		\leq
		C\varepsilon
		\|\nabla\phiz\|_{L^2(0,T;L^2(\Omega))},
	\end{equation}
	and, for each \(i=1,\ldots,d\),
	\begin{equation}
		\label{eq:Q-approx-dphi}
		\|\partial_{x_i}\phiz
		-
		\Qc_\varepsilon(\partial_{x_i}\phiz)
		\|_{L^2(0,T;L^2(\Omega_p^\varepsilon))}
		\leq
		C\varepsilon
		\|\phiz\|_{L^2(0,T;H^2(\Omega))}.
	\end{equation}
	Moreover,
	\begin{equation}
		\label{eq:grad-Phi-split}
		\left\|
		\nabla\Phi^\varepsilon
		-
		\left[
		\nabla\phiz
		+
		\sum_{i=1}^d
		\Qc_\varepsilon(\partial_{x_i}\phiz)
		\nabla_y\chi_i\!\left(\frac{x}{\varepsilon}\right)
		\right]
		\right\|_{L^2((0,T)\times\Omega_p^\varepsilon)}
		\leq
		C\varepsilon
		\|\phiz\|_{L^2(0,T;H^2(\Omega))},
	\end{equation}
	and
	\begin{equation}
		\label{eq:grad-Phi-flux}
		\|\nabla\Phi^\varepsilon-\mathcal J_{\phiz}^{\varepsilon}\|_
		{L^2((0,T)\times\Omega_p^\varepsilon)}
		\leq
		C\varepsilon
		\|\phiz\|_{L^2(0,T;H^2(\Omega))}.
	\end{equation}
	The corresponding estimates for \(M^\varepsilon\) are
	\begin{equation}
		\label{eq:grad-M-split}
		\left\|
		\nabla M^\varepsilon
		-
		\left[
		\nabla\muz
		+
		\sum_{i=1}^d
		\Qc_\varepsilon(\partial_{x_i}\muz)
		\nabla_y\chi_i\!\left(\frac{x}{\varepsilon}\right)
		\right]
		\right\|_{L^2((0,T)\times\Omega_p^\varepsilon)}
		\leq
		C\varepsilon
		\|\muz\|_{L^2(0,T;H^2(\Omega))},
	\end{equation}
	and
	\begin{equation}
		\label{eq:grad-M-flux}
		\|\nabla M^\varepsilon-\mathcal J_{\muz}^{\varepsilon}\|_
		{L^2((0,T)\times\Omega_p^\varepsilon)}
		\leq
		C\varepsilon
		\|\muz\|_{L^2(0,T;H^2(\Omega))}.
	\end{equation}
\end{lemma}

\begin{proof}
	We prove the estimates for \(\phiz\). The proof for \(\muz\) is identical.
	
	By the \(L^2\)-approximation property of \(\Qc_\varepsilon\), applied on
	\(\Omega\) and then restricted to \(\Omega_p^\varepsilon\), for a.e.
	\(t\in(0,T)\),
	\[
	\|\phiz(t)-\Qc_\varepsilon\phiz(t)\|_{L^2(\Omega_p^\varepsilon)}
	\leq
	C\varepsilon
	\|\nabla\phiz(t)\|_{L^2(\Omega)}.
	\]
	After squaring and integrating in time, this gives
	\eqref{eq:Q-approx-phi}. Applying the same estimate to
	\(\partial_{x_i}\phiz(t)\in H^1(\Omega)\) yields
	\[
	\|\partial_{x_i}\phiz(t)
	-
	\Qc_\varepsilon(\partial_{x_i}\phiz)(t)
	\|_{L^2(\Omega_p^\varepsilon)}
	\leq
	C\varepsilon
	\|\nabla\partial_{x_i}\phiz(t)\|_{L^2(\Omega)}.
	\]
	Squaring and integrating in time gives \eqref{eq:Q-approx-dphi}.
	
	From the definition of \(\Phi^\varepsilon\), we have in the weak sense
	on \(\Omega_p^\varepsilon\)
	\[
	\nabla\Phi^\varepsilon
	=
	\nabla\phiz
	+
	\sum_{i=1}^d
	\Qc_\varepsilon(\partial_{x_i}\phiz)
	\nabla_y\chi_i\!\left(\frac{x}{\varepsilon}\right)
	+
	\varepsilon
	\sum_{i=1}^d
	\nabla\Qc_\varepsilon(\partial_{x_i}\phiz)
	\chi_i\!\left(\frac{x}{\varepsilon}\right).
	\]
	The last term is estimated by the gradient product estimate
	\eqref{eq:P5}:
	\[
	\left\|
	\nabla\Qc_\varepsilon(\partial_{x_i}\phiz)(t)
	\chi_i\!\left(\frac{x}{\varepsilon}\right)
	\right\|_{L^2(\Omega_p^\varepsilon)}
	\leq
	C
	\|\nabla\partial_{x_i}\phiz(t)\|_{L^2(\Omega)}
	\|\chi_i\|_{L^2(Y)}.
	\]
	Hence
	\[
	\left\|
	\varepsilon
	\sum_{i=1}^d
	\nabla\Qc_\varepsilon(\partial_{x_i}\phiz)
	\chi_i\!\left(\frac{x}{\varepsilon}\right)
	\right\|_{L^2((0,T)\times\Omega_p^\varepsilon)}
	\leq
	C\varepsilon
	\|\phiz\|_{L^2(0,T;H^2(\Omega))},
	\]
	which proves \eqref{eq:grad-Phi-split}.
	
	It remains to compare with \(\mathcal J_{\phiz}^\varepsilon\). Since
	\[
	\mathcal J_{\phiz}^{\varepsilon}
	=
	\sum_{i=1}^d
	\Qc_\varepsilon(\partial_{x_i}\phiz)\be_i
	+
	\sum_{i=1}^d
	\Qc_\varepsilon(\partial_{x_i}\phiz)
	\nabla_y\chi_i\!\left(\frac{x}{\varepsilon}\right),
	\]
	we combine \eqref{eq:grad-Phi-split} with
	\eqref{eq:Q-approx-dphi} and obtain
	\[
	\|\nabla\Phi^\varepsilon-\mathcal J_{\phiz}^{\varepsilon}\|_
	{L^2((0,T)\times\Omega_p^\varepsilon)}
	\leq
	C\varepsilon
	\|\phiz\|_{L^2(0,T;H^2(\Omega))}.
	\]
	This proves \eqref{eq:grad-Phi-flux}. The estimates
	\eqref{eq:grad-M-split} and \eqref{eq:grad-M-flux} follow in the same
	way with \(\muz\) in place of \(\phiz\).
\end{proof}

\subsection{Variational residual equations}
\label{subsec:residual}

The corrector approximations $(\Phi^\varepsilon,M^\varepsilon)$ are not
expected to satisfy the microscopic equations in the strong sense, since
the cell correctors belong only to $H^1_{\per,0}(Y_p)$.
We therefore define the residuals only through the weak formulations,
in the spirit of the variational unfolding error method.

For a.e. \(t\in(0,T)\), define
\(\mathcal R_1^\varepsilon(t)\in H^1(\Omega_p^\varepsilon)'\) by
\begin{equation}
	\label{eq:weak-residual1}
	\bigl\langle \mathcal R_1^\varepsilon(t),\zeta\bigr\rangle
	:=
	\bigl\langle \partial_t\Phi^\varepsilon(t),\zeta\bigr\rangle
	+
	\int_{\Omega_p^\varepsilon}
	\nabla M^\varepsilon(t)\cdot\nabla\zeta\,dx
	+
	\int_{\Omega_p^\varepsilon}
	G(\Phi^\varepsilon(t))\zeta\,dx
\end{equation}
for all \(\zeta\in H^1(\Omega_p^\varepsilon)\). Similarly, define
\(\mathcal R_2^\varepsilon(t)\in H^1(\Omega_p^\varepsilon)'\) by
\begin{equation}
	\label{eq:weak-residual2}
	\bigl\langle \mathcal R_2^\varepsilon(t),\eta\bigr\rangle
	:=
	\int_{\Omega_p^\varepsilon}
	M^\varepsilon(t)\eta\,dx
	-
	\int_{\Omega_p^\varepsilon}
	\nabla\Phi^\varepsilon(t)\cdot\nabla\eta\,dx
	-
	\int_{\Omega_p^\varepsilon}
	F'(\Phi^\varepsilon(t))\eta\,dx
\end{equation}
for all \(\eta\in H^1(\Omega_p^\varepsilon)\).

The estimates of Section~\ref{subsec:corrector-approx} imply
\[
\Phi^\varepsilon\in L^\infty(0,T;H^1(\Omega_p^\varepsilon)),
\qquad
M^\varepsilon\in L^2(0,T;H^1(\Omega_p^\varepsilon)),
\]
and
\[
\partial_t\Phi^\varepsilon\in L^2(0,T;L^2(\Omega_p^\varepsilon)).
\]
Hence
\[
\mathcal R_1^\varepsilon
\in L^2(0,T;H^1(\Omega_p^\varepsilon)'),
\qquad
\mathcal R_2^\varepsilon
\in L^2(0,T;H^1(\Omega_p^\varepsilon)').
\]

Set
\[
e_\phi^\varepsilon:=\phi_\varepsilon-\Phi^\varepsilon,
\qquad
e_\mu^\varepsilon:=\mu_\varepsilon-M^\varepsilon .
\]
Subtracting \eqref{eq:weak-residual1} from the microscopic weak formulation
gives, for all \(\zeta\in H^1(\Omega_p^\varepsilon)\) and a.e.
\(t\in(0,T)\),
\begin{equation}
	\label{eq:error-phase-weak}
	\bigl\langle \partial_t e_\phi^\varepsilon(t),\zeta\bigr\rangle
	+
	\int_{\Omega_p^\varepsilon}
	\nabla e_\mu^\varepsilon(t)\cdot\nabla\zeta\,dx
	+
	\int_{\Omega_p^\varepsilon}
	\bigl(G(\phi_\varepsilon(t))-G(\Phi^\varepsilon(t))\bigr)\zeta\,dx
	=
	-
	\bigl\langle \mathcal R_1^\varepsilon(t),\zeta\bigr\rangle .
\end{equation}
Likewise, subtracting \eqref{eq:weak-residual2} from the microscopic
chemical-potential identity gives, for all
\(\eta\in H^1(\Omega_p^\varepsilon)\) and a.e. \(t\in(0,T)\),
\begin{equation}
	\label{eq:error-chem-weak}
	\int_{\Omega_p^\varepsilon}
	e_\mu^\varepsilon(t)\eta\,dx
	=
	\int_{\Omega_p^\varepsilon}
	\nabla e_\phi^\varepsilon(t)\cdot\nabla\eta\,dx
	+
	\int_{\Omega_p^\varepsilon}
	\bigl(F'(\phi_\varepsilon(t))-F'(\Phi^\varepsilon(t))\bigr)\eta\,dx
	-
	\bigl\langle \mathcal R_2^\varepsilon(t),\eta\bigr\rangle .
\end{equation}

\begin{lemma}[Porosity oscillation estimate]
	\label{lem:porosity-oscillation}
	Let \(h\in H^1(\Omega)\), \(v\in H^1(\Omega_p^\varepsilon)\), and let
	\(V_\varepsilon:=P_\varepsilon v\in H^1(\Omega)\) be a uniform extension of
	\(v\), that is,
	\[
	V_\varepsilon=v\quad\text{in }\Omega_p^\varepsilon,
	\qquad
	\|V_\varepsilon\|_{H^1(\Omega)}
	\leq
	C\|v\|_{H^1(\Omega_p^\varepsilon)} .
	\]
	Then
	\begin{equation}
		\label{eq:porosity-oscillation}
		\left|
		\int_{\Omega_p^\varepsilon}h v\,dx
		-
		\theta_p\int_\Omega h V_\varepsilon\,dx
		\right|
		\leq
		C\varepsilon
		\|h\|_{H^1(\Omega)}
		\|v\|_{H^1(\Omega_p^\varepsilon)} .
	\end{equation}
\end{lemma}

\begin{proof}
	Since \(V_\varepsilon=v\) in \(\Omega_p^\varepsilon\), we have
	\[
	\int_{\Omega_p^\varepsilon}h v\,dx
	-
	\theta_p\int_\Omega h V_\varepsilon\,dx
	=
	\int_\Omega
	\bigl(\mathbf 1_{\Omega_p^\varepsilon}-\theta_p\bigr)
	h V_\varepsilon\,dx .
	\]
	We split \(\Omega=\Omega_\varepsilon\cup\Lambda_\varepsilon\), where
	\(\Omega_\varepsilon\) is the union of complete cells and
	\(\Lambda_\varepsilon:=\Omega\setminus\Omega_\varepsilon\).

	Set
	\[
	\rho(y):=\mathbf 1_{Y_p}(y)-\theta_p .
	\]
	Then \(\rho\in L^\infty(Y)\) and
	\[
	\int_Y \rho(y)\,dy=0 .
	\]
	For \(k\in K_\varepsilon\), let
	\[
	g(x):=h(x)V_\varepsilon(x).
	\]
	By the zero-mean property of \(\rho\) and the Poincar\'e--Wirtinger
	inequality in \(W^{1,1}(Y)\), applied after rescaling to the cell
	\(\varepsilon(k+Y)\), we obtain
	\[
	\left|
	\int_{\varepsilon(k+Y)}
	\rho\left(\frac{x}{\varepsilon}\right) g(x)\,dx
	\right|
	\leq
	C\varepsilon
	\int_{\varepsilon(k+Y)}
	|\nabla g(x)|\,dx .
	\]
	Summing over \(k\in K_\varepsilon\), using
	\[
	|\nabla(hV_\varepsilon)|
	\leq
	|\nabla h|\,|V_\varepsilon|
	+
	|h|\,|\nabla V_\varepsilon|,
	\]
	and applying Cauchy's inequality, gives
	\begin{align}
		\left|
		\int_{\Omega_\varepsilon}
		\bigl(\mathbf 1_{\Omega_p^\varepsilon}-\theta_p\bigr)
		h V_\varepsilon\,dx
		\right|
		&\leq
		C\varepsilon
		\int_{\Omega_\varepsilon}
		\left(
		|\nabla h|\,|V_\varepsilon|
		+
		|h|\,|\nabla V_\varepsilon|
		\right)\,dx
		\nonumber\\
		&\leq
		C\varepsilon
		\|h\|_{H^1(\Omega)}
		\|V_\varepsilon\|_{H^1(\Omega)} .
		\label{eq:porosity-complete-cells}
	\end{align}
	Since \(\Omega\) is Lipschitz, $|\Lambda_\varepsilon|\leq C\varepsilon$.
	Moreover, the standard collar estimate gives, for every
	\(z\in H^1(\Omega)\),
	\[
	\|z\|_{L^2(\Lambda_\varepsilon)}
	\leq
	C\varepsilon^{1/2}\|z\|_{H^1(\Omega)} .
	\]
	Hence
	\begin{align}
		\left|
		\int_{\Lambda_\varepsilon}
		\bigl(\mathbf 1_{\Omega_p^\varepsilon}-\theta_p\bigr)
		h V_\varepsilon\,dx
		\right|
		&\leq
		C\|h\|_{L^2(\Lambda_\varepsilon)}
		\|V_\varepsilon\|_{L^2(\Lambda_\varepsilon)}
		\leq
		C\varepsilon
		\|h\|_{H^1(\Omega)}
		\|V_\varepsilon\|_{H^1(\Omega)} .
		\label{eq:porosity-boundary-layer}
	\end{align}
	Combining \eqref{eq:porosity-complete-cells} and
	\eqref{eq:porosity-boundary-layer}, and using the extension estimate for
	\(V_\varepsilon=P_\varepsilon v\), gives \eqref{eq:porosity-oscillation}.
\end{proof}

\begin{lemma}[Elliptic consistency estimate]
	\label{lem:elliptic-consistency}
	Let \(w\in H^2(\Omega)\). Define
	\[
	\mathcal J_w^\varepsilon
	:=
	\sum_{i=1}^d
	\Qc_\varepsilon(\partial_{x_i}w)
	\left(
	\be_i+\nabla_y\chi_i\left(\frac{x}{\varepsilon}\right)
	\right)
	\quad\text{in }\Omega_p^\varepsilon .
	\]
	Then, for every \(v\in H^1(\Omega_p^\varepsilon)\), with
	\(V_\varepsilon:=P_\varepsilon v\), one has
	\begin{equation}
		\label{eq:griso-consistency-fixed}
		\left|
		\int_{\Omega_p^\varepsilon}
		\mathcal J_w^\varepsilon\cdot\nabla v\,dx
		-
		\theta_p
		\int_\Omega
		\B^{\mathrm{hom}}\nabla w\cdot\nabla V_\varepsilon\,dx
		\right|
		\leq
		C\varepsilon^{1/2}
		\|w\|_{H^2(\Omega)}
		\|v\|_{H^1(\Omega_p^\varepsilon)} .
	\end{equation}
\end{lemma}

\begin{proof}
	For $i=1,\ldots,d$, set
	\[
	b_i(y) := \mathbf{e}_i + \nabla_y\chi_i(y) \in L^2(Y_p)^d.
	\]
	The cell problem~\eqref{eq:cell-problem} gives, in the weak sense,
	\[
	\mathrm{div}_y\, b_i = 0 \quad\text{in }Y_p,
	\qquad
	b_i\cdot\bn_y = 0 \quad\text{on }\partial Y_s,
	\]
	and $b_i$ is $Y$-periodic. By the definition~\eqref{eq:Bhom} of
	$\B^{\mathrm{hom}}$ and the cell problem~\eqref{eq:cell-problem},
	\begin{equation}
		\label{eq:cell-average-flux}
		\int_{Y_p} b_i(y)\cdot\be_j\,dy = \theta_p \B^{\mathrm{hom}}_{ij},
		\qquad i,j=1,\ldots,d.
	\end{equation}
	Since $V_\varepsilon=v$ on $\Omega_p^\varepsilon$, we have
	$\int_{\Omega_p^\varepsilon}\mathcal{J}_w^\varepsilon\cdot\nabla v\,dx
	=\int_{\Omega_p^\varepsilon}\mathcal{J}_w^\varepsilon\cdot\nabla V_\varepsilon\,dx$.
	We split
	\[
	R_\varepsilon
	:=
	\int_{\Omega_p^\varepsilon}\mathcal{J}_w^\varepsilon\cdot\nabla V_\varepsilon\,dx
	-
	\theta_p\int_\Omega\B^{\mathrm{hom}}\nabla w\cdot\nabla V_\varepsilon\,dx
	=:R_\varepsilon^{(1)}+R_\varepsilon^{(2)},
	\]
	where $R_\varepsilon^{(1)}$ collects the contribution from the complete cells
	$\Omega_\varepsilon$ and $R_\varepsilon^{(2)}$ collects the boundary layer
	$\Lambda_\varepsilon:=\Omega\setminus\Omega_\varepsilon$.

	\textbf{Step 1: Complete cells.}
	By the unfolding criterion~\eqref{EQUCI} and the $L^2$-stability of
	$\mathcal{T}_\varepsilon$,
	\[
	\int_{\Omega_\varepsilon\cap\Omega_p^\varepsilon}
	\mathcal{J}_w^\varepsilon\cdot\nabla V_\varepsilon\,dx
	=
	\frac{1}{|Y|}
	\int_{\Omega_\varepsilon}\int_{Y_p}
	\sum_{i=1}^d
	\mathcal{T}_\varepsilon(Q_\varepsilon(\partial_{x_i}w))(x,y)\,
	b_i(y)\cdot
	\mathcal{T}_\varepsilon(\nabla V_\varepsilon)(x,y)\,dy\,dx.
	\]
	We replace $\mathcal{T}_\varepsilon(Q_\varepsilon(\partial_{x_i}w))$ by
	$Q_\varepsilon(\partial_{x_i}w)(x)$. By the standard unfolding approximation
	estimate~\cite[Equation~(3.4)]{Griso2004},
	\[
	\|\mathcal{T}_\varepsilon(Q_\varepsilon(\partial_{x_i}w))
	-Q_\varepsilon(\partial_{x_i}w)\|_{L^2(\Omega_\varepsilon\times Y)}
	\leq
	C\varepsilon\|\nabla Q_\varepsilon(\partial_{x_i}w)\|_{L^2(\Omega)}
	\leq
	C\varepsilon\|w\|_{H^2(\Omega)},
	\]
	where the last step uses the $H^1$-stability of $Q_\varepsilon$
	(property~\eqref{eq:Q-H1-stability}) and $\partial_{x_i}w\in H^1(\Omega)$.
	Since $\|b_i\|_{L^2(Y_p)}\leq C$ and
	$\|\mathcal{T}_\varepsilon(\nabla V_\varepsilon)\|_{L^2(\Omega_\varepsilon\times Y_p)}
	\leq C\|\nabla V_\varepsilon\|_{L^2(\Omega)}$,
	the replacement error is bounded by
	$C\varepsilon\|w\|_{H^2(\Omega)}\|\nabla V_\varepsilon\|_{L^2(\Omega)}$.

	It remains to estimate
	\[
	S_\varepsilon
	:=
	\frac{1}{|Y|}
	\int_{\Omega_\varepsilon}\int_{Y_p}
	\sum_{i=1}^d
	Q_\varepsilon(\partial_{x_i}w)(x)\,
	b_i(y)\cdot\mathcal{T}_\varepsilon(\nabla V_\varepsilon)(x,y)\,dy\,dx
	-
	\theta_p\int_{\Omega_\varepsilon}
	\B^{\mathrm{hom}}\nabla w\cdot\nabla V_\varepsilon\,dx.
	\]
Apply Theorem~3.4 of Griso~\cite{Griso2004} to
\(V_\varepsilon\in H^1(\Omega)\). There exists
\(\hat{\psi}^\varepsilon\in H^1_{\per}(Y;L^2(\Omega))\), equivalently
\(\hat{\psi}^\varepsilon\in L^2(\Omega;H^1_{\per}(Y))\), such that
\begin{equation}
	\label{eq:griso-34}
	\begin{aligned}
		\bigl\|
		\mathcal{T}_\varepsilon(\nabla V_\varepsilon)
		-\nabla_x V_\varepsilon
		-\nabla_y\hat{\psi}^\varepsilon
		\bigr\|_{[L^2(Y;H^{-1}(\Omega_\varepsilon))]^d}
		&\leq
		C\varepsilon\|\nabla V_\varepsilon\|_{L^2(\Omega)},\\
		\|\hat{\psi}^\varepsilon\|_{H^1(Y;L^2(\Omega))}
		&\leq
		C\|\nabla V_\varepsilon\|_{L^2(\Omega)}.
	\end{aligned}
\end{equation}
We write
\[
\mathcal{T}_\varepsilon(\nabla V_\varepsilon)
=
\nabla_x V_\varepsilon+\nabla_y\hat{\psi}^\varepsilon+r^\varepsilon,
\]
where
\[
\|r^\varepsilon\|_{[L^2(Y;H^{-1}(\Omega_\varepsilon))]^d}
\leq
C\varepsilon\|\nabla V_\varepsilon\|_{L^2(\Omega)}.
\]
\textbf{Estimate of the \(r^\varepsilon\)-term.}
Since \(Q_\varepsilon\) is \(H^1\)-stable and
\(\partial_{x_i}w\in H^1(\Omega)\), we have
\[
Q_\varepsilon(\partial_{x_i}w)\in H^1(\Omega),
\qquad
\|Q_\varepsilon(\partial_{x_i}w)\|_{H^1(\Omega)}
\leq
C\|w\|_{H^2(\Omega)},
\]
by property~\eqref{eq:Q-H1-stability}. Hence
\(Q_\varepsilon(\partial_{x_i}w)|_{\Omega_\varepsilon}\in H^1(\Omega_\varepsilon)\)
with the same bound. Since
\(r^\varepsilon\in [L^2(Y;H^{-1}(\Omega_\varepsilon))]^d\), the term involving
\(r^\varepsilon\) is understood as the duality pairing
\[
\frac{1}{|Y|}
\int_{Y_p}
\left\langle
r^\varepsilon(\cdot,y),
Q_\varepsilon(\partial_{x_i}w)(\cdot)b_i(y)
\right\rangle_
{[H^{-1}(\Omega_\varepsilon)]^d,[H^1(\Omega_\varepsilon)]^d}
\,dy .
\]
Therefore,
\[
\begin{aligned}
	&\left|
	\frac{1}{|Y|}
	\int_{Y_p}
	\left\langle
	r^\varepsilon(\cdot,y),
	Q_\varepsilon(\partial_{x_i}w)(\cdot)b_i(y)
	\right\rangle_
	{[H^{-1}(\Omega_\varepsilon)]^d,[H^1(\Omega_\varepsilon)]^d}
	\,dy
	\right|
	\\
	&\qquad\leq
	C
	\|r^\varepsilon\|_{[L^2(Y;H^{-1}(\Omega_\varepsilon))]^d}
	\,
	\|Q_\varepsilon(\partial_{x_i}w)b_i\|_{[L^2(Y_p;H^1(\Omega_\varepsilon))]^d}
	\\
	&\qquad\leq
	C
	\|r^\varepsilon\|_{[L^2(Y;H^{-1}(\Omega_\varepsilon))]^d}
	\,
	\|Q_\varepsilon(\partial_{x_i}w)\|_{H^1(\Omega_\varepsilon)}
	\|b_i\|_{L^2(Y_p)}
	\\
	&\qquad\leq
	C\varepsilon
	\|w\|_{H^2(\Omega)}
	\|\nabla V_\varepsilon\|_{L^2(\Omega)}.
\end{aligned}
\]
\textbf{Cancellation of the \(\nabla_y\hat{\psi}^\varepsilon\)-term.}
Since \(Q_\varepsilon(\partial_{x_i}w)(x)\) is independent of \(y\), we have
\[
\frac{1}{|Y|}
\int_{\Omega_\varepsilon}\int_{Y_p}
Q_\varepsilon(\partial_{x_i}w)(x)\,
b_i(y)\cdot\nabla_y\hat{\psi}^\varepsilon(x,y)\,dy\,dx
=
\frac{1}{|Y|}
\int_{\Omega_\varepsilon}
Q_\varepsilon(\partial_{x_i}w)(x)
\left[
\int_{Y_p}b_i(y)\cdot\nabla_y\hat{\psi}^\varepsilon(x,y)\,dy
\right]dx.
\]
By Fubini, for a.e. \(x\in\Omega_\varepsilon\) the function
\(y\mapsto\hat{\psi}^\varepsilon(x,y)|_{Y_p}\) belongs to
\(H^1_{\per}(Y_p)\). After subtracting its \(Y_p\)-mean, it is an
admissible test function in the cell problem~\eqref{eq:cell-problem}. Since
\[
b_i=e_i+\nabla_y\chi_i
\]
satisfies
\[
\int_{Y_p} b_i(y)\cdot\nabla_y\zeta(y)\,dy=0
\qquad
\text{for all }\zeta\in H^1_{\per}(Y_p),
\]
we obtain
\[
\int_{Y_p}b_i(y)\cdot\nabla_y\hat{\psi}^\varepsilon(x,y)\,dy=0
\qquad
\text{for a.e. }x\in\Omega_\varepsilon.
\]
Hence the \(\nabla_y\hat{\psi}^\varepsilon\)-contribution vanishes.

\textbf{Identification of the \(\nabla_x V_\varepsilon\)-term.}
The remaining contribution over the complete cells is
\[
\frac{1}{|Y|}
\int_{\Omega_\varepsilon}\int_{Y_p}
\sum_{i=1}^d
Q_\varepsilon(\partial_{x_i}w)(x)\,
b_i(y)\cdot\nabla_x V_\varepsilon(x)\,dy\,dx .
\]
Thus
\[
\begin{aligned}
	&\frac{1}{|Y|}
	\int_{\Omega_\varepsilon}\int_{Y_p}
	\sum_{i=1}^d
	Q_\varepsilon(\partial_{x_i}w)(x)\,
	b_i(y)\cdot\nabla_x V_\varepsilon(x)\,dy\,dx
	\\
	&\qquad =
	\frac{1}{|Y|}
	\int_{\Omega_\varepsilon}
	\sum_{i=1}^d
	Q_\varepsilon(\partial_{x_i}w)(x)
	\left[
	\int_{Y_p}b_i(y)\,dy
	\right]\cdot\nabla_x V_\varepsilon(x)\,dx .
\end{aligned}
\]
By the cell-average identity~\eqref{eq:cell-average-flux}, this equals
\[
\theta_p
\int_{\Omega_\varepsilon}
\B^{\mathrm{hom}}Q_\varepsilon(\nabla w)\cdot\nabla V_\varepsilon\,dx .
\]
Finally, using the approximation property~\eqref{eq:Q-H1-stability}$_3$,
\[
\|Q_\varepsilon(\partial_{x_i}w)-\partial_{x_i}w\|_{L^2(\Omega)}
\leq
C\varepsilon\|\nabla\partial_{x_i}w\|_{L^2(\Omega)}
\leq
C\varepsilon\|w\|_{H^2(\Omega)}.
\]
Therefore,
\[
\left|
\theta_p
\int_{\Omega_\varepsilon}
\B^{\mathrm{hom}}
\bigl(Q_\varepsilon(\nabla w)-\nabla w\bigr)
\cdot\nabla V_\varepsilon\,dx
\right|
\leq
C\varepsilon
\|w\|_{H^2(\Omega)}
\|\nabla V_\varepsilon\|_{L^2(\Omega)}.
\]
Combining the complete-cell contributions gives
\begin{equation}
	\label{eq:complete-cell-final}
	|R_\varepsilon^{(1)}|
	\leq
	C\varepsilon
	\|w\|_{H^2(\Omega)}
	\|\nabla V_\varepsilon\|_{L^2(\Omega)}.
\end{equation}
The boundary layer near
$\partial\Omega$ is accounted for by the $\Lambda_\varepsilon$ contribution in Step 2, which produces the
$\GO(\varepsilon^{1/2})$ collar estimate and is the source of the
$\varepsilon^{1/2}$ rate on bounded domains.	
	
	\textbf{Step 2: Boundary layer.}
	We estimate
	$\|\mathcal{J}_w^\varepsilon\|_{L^2(\Lambda_\varepsilon\cap\Omega_p^\varepsilon)}$.
	On each partial cell $\varepsilon(\xi+Y)$ intersecting $\Lambda_\varepsilon$,
	a change of variables $x=\varepsilon\xi+\varepsilon y$ gives
	\[
	\int_{\varepsilon(\xi+Y)\cap\Omega_p^\varepsilon}
	|\mathcal{J}_w^\varepsilon|^2\,dx
	\leq
	\sum_{i=1}^d
	\|Q_\varepsilon(\partial_{x_i}w)\|^2_{L^\infty(\varepsilon(\xi+Y))}
	\cdot\varepsilon^d\|b_i\|^2_{L^2(Y_p)}.
	\]
Since $\mathcal{Q}_\varepsilon(\partial_{x_i}w)
=\widetilde{\mathcal{Q}}_\varepsilon(\mathcal{E}_\Omega(\partial_{x_i}w))$
is $Q_1$ on $\varepsilon(\xi+Y)$ with nodal values given by averages of
$\mathcal{E}_\Omega(\partial_{x_i}w)$ over neighbouring cells
(which may lie outside $\Omega$),
the inverse estimate for $Q_1$ functions gives
\[
\|\mathcal{Q}_\varepsilon(\partial_{x_i}w)\|_{L^\infty(\varepsilon(\xi+Y))}
\leq
C\varepsilon^{-d/2}
\|\mathcal{E}_\Omega(\partial_{x_i}w)\|_{L^2(\widetilde{\varepsilon(\xi+Y)})},
\]
where $\widetilde{\varepsilon(\xi+Y)}\subset\mathbb{R}^d$ denotes the union of
$\varepsilon(\xi+Y)$ and its $2^d$ neighbours, which need not be subsets of
$\Omega$. Substituting and summing over all partial cells $\xi$ with
$\varepsilon(\xi+Y)\cap\Lambda_\varepsilon\neq\emptyset$, and using bounded
overlap of the neighbourhoods $\widetilde{\varepsilon(\xi+Y)}$,
\[
\|\mathcal{J}_w^\varepsilon\|^2_{L^2(\Lambda_\varepsilon\cap\Omega_p^\varepsilon)}
\leq
C\sum_{i=1}^d\|b_i\|^2_{L^2(Y_p)}
\sum_{\xi\,\mathrm{partial}}
\|\mathcal{E}_\Omega(\partial_{x_i}w)\|^2_{L^2(\widetilde{\varepsilon(\xi+Y)})}.
\]
Applying the collar estimate to
$\mathcal{E}_\Omega(\partial_{x_i}w)\in H^1(\mathbb{R}^d)$,
\[
\sum_{\xi\,\mathrm{partial}}
\|\mathcal{E}_\Omega(\partial_{x_i}w)\|^2_{L^2(\widetilde{\varepsilon(\xi+Y)})}
\leq
C\|\mathcal{E}_\Omega(\partial_{x_i}w)\|^2_{L^2(\Lambda_{2\varepsilon}(\mathbb{R}^d))}
\leq
C\varepsilon\|\mathcal{E}_\Omega w\|^2_{H^2(\mathbb{R}^d)}
\leq
C\varepsilon\|w\|^2_{H^2(\Omega)},
\]
where the last step uses boundedness of $\mathcal{E}_\Omega:H^2(\Omega)\to H^2(\mathbb{R}^d)$.

	Taking square roots,
	\[
	\|\mathcal{J}_w^\varepsilon\|_{L^2(\Lambda_\varepsilon\cap\Omega_p^\varepsilon)}
	\leq
	C\varepsilon^{1/2}\|w\|_{H^2(\Omega)}.
	\]
	By Cauchy--Schwarz,
	\[
	\left|
	\int_{\Lambda_\varepsilon\cap\Omega_p^\varepsilon}
	\mathcal{J}_w^\varepsilon\cdot\nabla V_\varepsilon\,dx
	\right|
	\leq
	C\varepsilon^{1/2}\|w\|_{H^2(\Omega)}\|\nabla V_\varepsilon\|_{L^2(\Omega)}.
	\]
	For the homogenized term on $\Lambda_\varepsilon$, the collar estimate gives
	$\|\nabla w\|_{L^2(\Lambda_\varepsilon)}\leq C\varepsilon^{1/2}\|w\|_{H^2(\Omega)}$,
	hence
	\[
	\theta_p\left|
	\int_{\Lambda_\varepsilon}
	\B^{\mathrm{hom}}\nabla w\cdot\nabla V_\varepsilon\,dx
	\right|
	\leq
	C\varepsilon^{1/2}\|w\|_{H^2(\Omega)}\|\nabla V_\varepsilon\|_{L^2(\Omega)}.
	\]
	Therefore
	\begin{equation}
		\label{eq:layer-final}
		|R_\varepsilon^{(2)}|
		\leq
		C\varepsilon^{1/2}\|w\|_{H^2(\Omega)}\|\nabla V_\varepsilon\|_{L^2(\Omega)}.
	\end{equation}

	\textbf{Step 3: Conclusion.}
	Combining \eqref{eq:complete-cell-final} and \eqref{eq:layer-final}, and
	using $0<\varepsilon<1$ so that $\varepsilon\leq\varepsilon^{1/2}$,
	\[
	|R_\varepsilon|
	\leq
	C\varepsilon^{1/2}\|w\|_{H^2(\Omega)}\|\nabla V_\varepsilon\|_{L^2(\Omega)}
	\leq
	C\varepsilon^{1/2}\|w\|_{H^2(\Omega)}\|v\|_{H^1(\Omega_p^\varepsilon)},
	\]
	where the last step uses the uniform extension estimate
	$\|\nabla V_\varepsilon\|_{L^2(\Omega)}\leq C\|v\|_{H^1(\Omega_p^\varepsilon)}$.
	This proves~\eqref{eq:griso-consistency-fixed}.
\end{proof}

\begin{lemma}[Variational residual estimate]
	\label{lem:residual}
	Assume that the homogenized solution satisfies \eqref{eq:quant-reg}.
	Then there exists \(C>0\), independent of \(\varepsilon\), such that
	\begin{equation}
		\label{eq:residual-bound}
		\|\mathcal R_1^\varepsilon\|_{L^2(0,T;H^1(\Omega_p^\varepsilon)')}
		+
		\|\mathcal R_2^\varepsilon\|_{L^2(0,T;H^1(\Omega_p^\varepsilon)')}
		\leq
		C\varepsilon^{1/2}.
	\end{equation}
\end{lemma}

\begin{proof}
	We use the auxiliary fluxes \(\mathcal J_w^\varepsilon\) defined in
	\eqref{eq:aux-flux}. The proof is entirely variational.

	Define
	\[
	a_{\phiz}
	:=
	-\div(\B^{\mathrm{hom}}\nabla\phiz),
	\qquad
	a_{\muz}
	:=
	-\div(\B^{\mathrm{hom}}\nabla\muz).
	\]
	By the homogenized equations,
	\begin{equation}
		\label{eq:a-phi-a-mu}
		a_{\phiz}
		=
		\muz-F'(\phiz),
		\qquad
		a_{\muz}
		=
		-\partial_t\phiz-G(\phiz).
	\end{equation}
	Under \eqref{eq:quant-reg}, we have
	\[
	a_{\phiz}\in L^2(0,T;H^1(\Omega)),
	\qquad
	a_{\muz}\in L^2(0,T;H^1(\Omega)).
	\]
	Indeed, \(a_{\muz}\in L^2(0,T;H^1(\Omega))\) follows from $\partial_t\phiz\in L^2(0,T;H^1(\Omega))$
	and the Lipschitz continuity of \(G\). Also,
	\(a_{\phiz}\in L^2(0,T;H^1(\Omega))\) follows from $\muz\in L^2(0,T;H^1(\Omega))$
	and $F'(\phiz)\in L^\infty(0,T;H^1(\Omega))$.
	The latter is a consequence of
	\(\phiz\in L^\infty(0,T;H^2(\Omega))\), the embedding
	\(H^2(\Omega)\hookrightarrow L^\infty(\Omega)\) for \(d\leq3\), and
	the growth assumption on \(F''\).

	\textbf{Estimate of \(\mathcal R_2^\varepsilon\).}
	Let \(\eta\in H^1(\Omega_p^\varepsilon)\), and set
	\[
	H_\varepsilon:=P_\varepsilon\eta\in H^1(\Omega).
	\]
	From \eqref{eq:weak-residual2}, add and subtract
	\(\muz\), \(F'(\phiz)\), and \(\mathcal J_{\phiz}^\varepsilon\):
	\begin{multline*}
		\bigl\langle \mathcal R_2^\varepsilon(t),\eta\bigr\rangle
		=
		\int_{\Omega_p^\varepsilon}
		(M^\varepsilon-\muz)\eta\,dx
		+
		\int_{\Omega_p^\varepsilon}
		\bigl(F'(\phiz)-F'(\Phi^\varepsilon)\bigr)\eta\,dx
		-
		\int_{\Omega_p^\varepsilon}
		\bigl(\nabla\Phi^\varepsilon-\mathcal J_{\phiz}^\varepsilon\bigr)
		\cdot\nabla\eta\,dx
		\\
		+
		\int_{\Omega_p^\varepsilon}
		(\muz-F'(\phiz))\eta\,dx
		-
		\int_{\Omega_p^\varepsilon}
		\mathcal J_{\phiz}^\varepsilon\cdot\nabla\eta\,dx .
	\end{multline*}
	Using \eqref{eq:a-phi-a-mu}, the last line is
	\[
	\int_{\Omega_p^\varepsilon}
	a_{\phiz}\eta\,dx
	-
	\int_{\Omega_p^\varepsilon}
	\mathcal J_{\phiz}^\varepsilon\cdot\nabla\eta\,dx .
	\]
	We split this term as
	\begin{multline*}
		\int_{\Omega_p^\varepsilon}
		a_{\phiz}\eta\,dx
		-
		\int_{\Omega_p^\varepsilon}
		\mathcal J_{\phiz}^\varepsilon\cdot\nabla\eta\,dx
		=\\
		\left[
		\int_{\Omega_p^\varepsilon}
		a_{\phiz}\eta\,dx
		-
		\theta_p
		\int_\Omega
		a_{\phiz}H_\varepsilon\,dx
		\right]
		+
		\left[
		\theta_p
		\int_\Omega
		\B^{\mathrm{hom}}\nabla\phiz\cdot\nabla H_\varepsilon\,dx
		-
		\int_{\Omega_p^\varepsilon}
		\mathcal J_{\phiz}^\varepsilon\cdot\nabla\eta\,dx
		\right].
	\end{multline*}
	Here we used the homogenized elliptic relation $a_{\phiz}=-\div(\B^{\mathrm{hom}}\nabla\phiz)$
	together with the homogeneous conormal boundary condition $\B^{\mathrm{hom}}\nabla\phiz\cdot\bn=0
	\quad\text{on }\partial\Omega$,
	which gives
	\[
	\int_\Omega
	a_{\phiz}H_\varepsilon\,dx
	=
	\int_\Omega
	\B^{\mathrm{hom}}\nabla\phiz\cdot\nabla H_\varepsilon\,dx .
	\]
	By Lemma~\ref{lem:porosity-oscillation}, applied with
	\(h=a_{\phiz}(t)\), and by Lemma~\ref{lem:elliptic-consistency},
	applied with \(w=\phiz(t)\), we obtain, for a.e. \(t\in(0,T)\),
	\begin{multline}
		\left|
		\int_{\Omega_p^\varepsilon}
		(\muz-F'(\phiz))\eta\,dx
		-
		\int_{\Omega_p^\varepsilon}
		\mathcal J_{\phiz}^\varepsilon\cdot\nabla\eta\,dx
		\right|
		\nonumber
		\leq
		C\varepsilon
		\|a_{\phiz}(t)\|_{H^1(\Omega)}
		\|\eta\|_{H^1(\Omega_p^\varepsilon)}
		+
		C\varepsilon^{1/2}
		\|\phiz(t)\|_{H^2(\Omega)}
		\|\eta\|_{H^1(\Omega_p^\varepsilon)}
		\nonumber
		\\
		\leq
		C\varepsilon^{1/2}
		\left(
		\|\phiz(t)\|_{H^2(\Omega)}
		+
		\|a_{\phiz}(t)\|_{H^1(\Omega)}
		\right)
		\|\eta\|_{H^1(\Omega_p^\varepsilon)} .
		\label{eq:R2-main-bound-fixed}
	\end{multline}
	The remaining terms are of order \(\GO(\varepsilon)\) in
	\(H^1(\Omega_p^\varepsilon)'\). Indeed, by
	\eqref{eq:corrector-L2-approx},
	\[
	\|M^\varepsilon(t)-\muz(t)\|_{L^2(\Omega_p^\varepsilon)}
	\leq
	C\varepsilon\|\muz(t)\|_{H^1(\Omega)} ,
	\]
	and by Lemma~\ref{lem:corrector-approx},
	\[
	\|\nabla\Phi^\varepsilon(t)-\mathcal J_{\phiz}^\varepsilon(t)\|_
	{L^2(\Omega_p^\varepsilon)}
	\leq
	C\varepsilon\|\phiz(t)\|_{H^2(\Omega)} .
	\]
	For the nonlinear term, the growth condition on \(F''\), the embedding
	\(H^1(\Omega_p^\varepsilon)\hookrightarrow L^6(\Omega_p^\varepsilon)\)
	for \(d\leq3\), and the uniform \(H^1\)-bounds for
	\(\phiz\) and \(\Phi^\varepsilon\) give
	\[
	\|F'(\Phi^\varepsilon(t))-F'(\phiz(t))\|_
	{H^1(\Omega_p^\varepsilon)'}
	\leq
	C
	\|\Phi^\varepsilon(t)-\phiz(t)\|_{L^2(\Omega_p^\varepsilon)} .
	\]
	Using again \eqref{eq:corrector-L2-approx}, we get
	\[
	\|F'(\Phi^\varepsilon(t))-F'(\phiz(t))\|_
	{H^1(\Omega_p^\varepsilon)'}
	\leq
	C\varepsilon\|\phiz(t)\|_{H^1(\Omega)} .
	\]
	Consequently,
	\[
	\|\mathcal R_2^\varepsilon(t)\|_{H^1(\Omega_p^\varepsilon)'}
	\leq
	C\varepsilon^{1/2}
	\left(
	1
	+
	\|\phiz(t)\|_{H^2(\Omega)}
	+
	\|\muz(t)\|_{H^1(\Omega)}
	+
	\|a_{\phiz}(t)\|_{H^1(\Omega)}
	\right).
	\]
	Squaring and integrating over \((0,T)\), and using
	\eqref{eq:quant-reg} together with
	\(a_{\phiz}\in L^2(0,T;H^1(\Omega))\), yields the estimate for
	\(\mathcal R_2^\varepsilon\).

	\textbf{Estimate of \(\mathcal R_1^\varepsilon\).}
	Let \(\zeta\in H^1(\Omega_p^\varepsilon)\), and set
	\[
	Z_\varepsilon:=P_\varepsilon\zeta\in H^1(\Omega).
	\]
	From \eqref{eq:weak-residual1}, add and subtract
	\(\partial_t\phiz\), \(G(\phiz)\), and
	\(\mathcal J_{\muz}^\varepsilon\):
	\begin{multline*}
		\bigl\langle \mathcal R_1^\varepsilon(t),\zeta\bigr\rangle
		=
		\bigl\langle\partial_t(\Phi^\varepsilon-\phiz)(t),\zeta\bigr\rangle
		+
		\int_{\Omega_p^\varepsilon}
		\bigl(G(\Phi^\varepsilon(t))-G(\phiz(t))\bigr)\zeta\,dx
		\\
		+
		\int_{\Omega_p^\varepsilon}
		\bigl(\nabla M^\varepsilon(t)-\mathcal J_{\muz}^\varepsilon(t)\bigr)
		\cdot\nabla\zeta\,dx
		+
		\int_{\Omega_p^\varepsilon}
		\mathcal J_{\muz}^\varepsilon(t)\cdot\nabla\zeta\,dx
		+
		\int_{\Omega_p^\varepsilon}
		(\partial_t\phiz(t)+G(\phiz(t)))\zeta\,dx .
	\end{multline*}
	Using \eqref{eq:a-phi-a-mu}, we have
	\[
	\partial_t\phiz+G(\phiz)
	=
	-a_{\muz}.
	\]
	Hence the last line is
	\[
	\int_{\Omega_p^\varepsilon}
	\mathcal J_{\muz}^\varepsilon\cdot\nabla\zeta\,dx
	-
	\int_{\Omega_p^\varepsilon}
	a_{\muz}\zeta\,dx .
	\]
	We split this term as
	\begin{multline*}
		\int_{\Omega_p^\varepsilon}
		\mathcal J_{\muz}^\varepsilon\cdot\nabla\zeta\,dx
		-
		\int_{\Omega_p^\varepsilon}
		a_{\muz}\zeta\,dx
		=
		\left[
		\int_{\Omega_p^\varepsilon}
		\mathcal J_{\muz}^\varepsilon\cdot\nabla\zeta\,dx
		-
		\theta_p
		\int_\Omega
		\B^{\mathrm{hom}}\nabla\muz\cdot\nabla Z_\varepsilon\,dx
		\right]
		\\
		+
		\left[
		\theta_p
		\int_\Omega
		a_{\muz}Z_\varepsilon\,dx
		-
		\int_{\Omega_p^\varepsilon}
		a_{\muz}\zeta\,dx
		\right].
	\end{multline*}
	Here we used the homogenized elliptic relation $a_{\muz}=-\div(\B^{\mathrm{hom}}\nabla\muz)$
	together with the homogeneous conormal boundary condition $\B^{\mathrm{hom}}\nabla\muz\cdot\bn=0
	\quad\text{on }\partial\Omega$,
	which gives
	\[
	\int_\Omega
	a_{\muz}Z_\varepsilon\,dx
	=
	\int_\Omega
	\B^{\mathrm{hom}}\nabla\muz\cdot\nabla Z_\varepsilon\,dx .
	\]
	By Lemma~\ref{lem:elliptic-consistency}, applied with
	\(w=\muz(t)\), and by Lemma~\ref{lem:porosity-oscillation}, applied
	with \(h=a_{\muz}(t)\), we obtain, for a.e. \(t\in(0,T)\),
	\begin{multline}
		\left|
		\int_{\Omega_p^\varepsilon}
		\mathcal J_{\muz}^\varepsilon\cdot\nabla\zeta\,dx
		+
		\int_{\Omega_p^\varepsilon}
		(\partial_t\phiz+G(\phiz))\zeta\,dx
		\right|
		\nonumber
		\leq
		C\varepsilon^{1/2}
		\|\muz(t)\|_{H^2(\Omega)}
		\|\zeta\|_{H^1(\Omega_p^\varepsilon)}
		+
		C\varepsilon
		\|a_{\muz}(t)\|_{H^1(\Omega)}
		\|\zeta\|_{H^1(\Omega_p^\varepsilon)}
		\nonumber
		\\
		\leq
		C\varepsilon^{1/2}
		\left(
		\|\muz(t)\|_{H^2(\Omega)}
		+
		\|a_{\muz}(t)\|_{H^1(\Omega)}
		\right)
		\|\zeta\|_{H^1(\Omega_p^\varepsilon)} .
		\label{eq:R1-main-bound-fixed}
	\end{multline}
	The gradient remainder satisfies
	\[
	\|\nabla M^\varepsilon(t)-\mathcal J_{\muz}^\varepsilon(t)\|_
	{L^2(\Omega_p^\varepsilon)}
	\leq
	C\varepsilon\|\muz(t)\|_{H^2(\Omega)}
	\]
	by Lemma~\ref{lem:corrector-approx}. The source term is controlled by
	the Lipschitz continuity of \(G\) and \eqref{eq:corrector-L2-approx}:
	\[
	\|G(\Phi^\varepsilon(t))-G(\phiz(t))\|_{L^2(\Omega_p^\varepsilon)}
	\leq
	C\varepsilon\|\phiz(t)\|_{H^1(\Omega)} .
	\]
	Finally, \eqref{eq:time-corrector-bound} gives
	\[
	\|\partial_t(\Phi^\varepsilon-\phiz)(t)\|_
	{H^1(\Omega_p^\varepsilon)'}
	\leq
	C\varepsilon\|\partial_t\phiz(t)\|_{H^1(\Omega)} .
	\]
	Combining these estimates gives
	\[
	\|\mathcal R_1^\varepsilon(t)\|_{H^1(\Omega_p^\varepsilon)'}
	\leq
	C\varepsilon^{1/2}
	\left(
	1
	+
	\|\muz(t)\|_{H^2(\Omega)}
	+
	\|a_{\muz}(t)\|_{H^1(\Omega)}
	+
	\|\phiz(t)\|_{H^1(\Omega)}
	+
	\|\partial_t\phiz(t)\|_{H^1(\Omega)}
	\right).
	\]
	Squaring and integrating over \((0,T)\), and using
	\eqref{eq:quant-reg} together with
	\(a_{\muz}\in L^2(0,T;H^1(\Omega))\), yields the estimate for
	\(\mathcal R_1^\varepsilon\).
	
	Combining the estimates for \(\mathcal R_1^\varepsilon\) and
	\(\mathcal R_2^\varepsilon\) proves \eqref{eq:residual-bound}.
\end{proof}
\subsection{Quantitative convergence rates}
\label{subsec:rates}

We now derive the quantitative estimate in the variational energy scale
associated with the Cahn--Hilliard structure. Since the residuals
\(\mathcal R_1^\varepsilon\) and \(\mathcal R_2^\varepsilon\) are defined
only in dual spaces, we do not test the chemical-potential error equation
with \(\partial_t e_\phi^\varepsilon\). Instead, we use the standard
Neumann inverse and work in a negative norm for the phase variable.

Let
\[
m_\varepsilon(t)
:=
\frac{1}{|\Omega_p^\varepsilon|}
\int_{\Omega_p^\varepsilon}
e_\phi^\varepsilon(t,x)\,dx,
\qquad
\widetilde e_\phi^\varepsilon
:=
e_\phi^\varepsilon-m_\varepsilon .
\]
For each \(f\in H^1(\Omega_p^\varepsilon)'\) with zero mean, let
\(\mathcal N_\varepsilon f\in H^1(\Omega_p^\varepsilon)\) be the solution
of
\[
-\Delta \mathcal N_\varepsilon f=f
\quad\text{in }\Omega_p^\varepsilon,
\qquad
\nabla\mathcal N_\varepsilon f\cdot\bn_\varepsilon=0
\quad\text{on }\partial\Omega_p^\varepsilon,
\qquad
\int_{\Omega_p^\varepsilon}\mathcal N_\varepsilon f\,dx=0.
\]
We define
\[
\|f\|_{H^{-1}_\varepsilon}^2
:=
\int_{\Omega_p^\varepsilon}
|\nabla\mathcal N_\varepsilon f|^2\,dx .
\]
By the uniform Poincaré--Wirtinger inequality on $\Omega^\varepsilon_p$
(Remark~\ref{rem:frame}) and standard elliptic regularity for the Neumann
Laplacian on Lipschitz domains, the norm $\|\cdot\|_{H^{-1}_\varepsilon}$ is
uniformly equivalent (in $\varepsilon$) to the standard $H^1(\Omega^\varepsilon_p)'$-norm
on zero-mean distributions; see for instance~\cite{damlamian2002sequences}.

\begin{theorem}[Variational corrector estimate]
	\label{thm:corrector-estimate}
	Let Assumptions~\ref{ass:F}, \ref{ass:G}, \ref{ass:init},
	and~\ref{ass:quant-reg} hold. Then there
	exists \(C>0\), independent of \(\varepsilon\), such that
	\begin{equation}
		\label{eq:stability-rate}
		\|\widetilde e_\phi^\varepsilon\|_{L^\infty(0,T;H^{-1}_\varepsilon)}
		+
		\|e_\phi^\varepsilon\|_{L^2(0,T;H^1(\Omega_p^\varepsilon))}
		+
		\|m_\varepsilon\|_{L^\infty(0,T)}
		\leq
		C\Bigl(
		\|\widetilde e_\phi^\varepsilon(0)\|_{H^{-1}_\varepsilon}
		+
		|m_\varepsilon(0)|
		+
		\varepsilon^{1/2}
		\Bigr).
	\end{equation}
	If, in addition, we have well prepared initial data
	\begin{equation}\label{eq:RIC}
		\|\widetilde e_\phi^\varepsilon(0)\|_{H^{-1}_\varepsilon}
		+
		|m_\varepsilon(0)|
		\leq
		C\varepsilon^{1/2},
	\end{equation}
	then
	\[
	\|\widetilde e_\phi^\varepsilon\|_{L^\infty(0,T;H^{-1}_\varepsilon)}
	+
	\|e_\phi^\varepsilon\|_{L^2(0,T;H^1(\Omega_p^\varepsilon))}
	+
	\|m_\varepsilon\|_{L^\infty(0,T)}
	\leq
	C\varepsilon^{1/2}.
	\]
	The constant \(C\) depends on the data, \(T\), the cell geometry, and the
	regularity norms in \eqref{eq:quant-reg}, but not on
	\(\varepsilon\).
\end{theorem}

\begin{proof}
	The proof is based on the weak error identities
	\eqref{eq:error-phase-weak}--\eqref{eq:error-chem-weak} and the
	variational residual estimate of Lemma~\ref{lem:residual}. We suppress
	the time variable in the notation.
	
	\textbf{Step 1: Negative-norm testing of the phase equation.}
	Take \(\zeta=\mathcal N_\varepsilon\widetilde e_\phi^\varepsilon\)
	in \eqref{eq:error-phase-weak}. Since
	\(\mathcal N_\varepsilon\widetilde e_\phi^\varepsilon\) has zero mean,
	the mean part of \(\partial_t e_\phi^\varepsilon\) does not contribute.
	To justify this identity, note that $\mathcal{N}_\varepsilon$ is a bounded
	self-adjoint operator on the zero-mean subspace of $L^2(\Omega_p^\varepsilon)$,
	and $t\mapsto\widetilde{e}_\phi^\varepsilon(t)$ belongs to
	$L^2(0,T;L^2(\Omega_p^\varepsilon))\cap H^1(0,T;H^1(\Omega_p^\varepsilon)')$
	with zero mean. For any $\tau\in(0,T)$, testing the distributional time
	derivative against $\mathcal{N}_\varepsilon\widetilde{e}_\phi^\varepsilon$
	and using the symmetry of $\mathcal{N}_\varepsilon$ gives
	\[
	\int_0^\tau
	\left\langle
	\partial_t\widetilde{e}_\phi^\varepsilon,
	\mathcal{N}_\varepsilon\widetilde{e}_\phi^\varepsilon
	\right\rangle\,dt
	=
	\frac{1}{2}\|\widetilde{e}_\phi^\varepsilon(\tau)\|^2_{H^{-1}_\varepsilon}
	-
	\frac{1}{2}\|\widetilde{e}_\phi^\varepsilon(0)\|^2_{H^{-1}_\varepsilon},
	\]
	which is the integrated form of the identity below; see
	e.g.~\cite[Chapter~III, Lemma~1.2]{Temam1997}.
	Since the mean part of $\partial_t e_\phi^\varepsilon$ pairs to zero against
	$\mathcal{N}_\varepsilon\widetilde{e}_\phi^\varepsilon$ (which has zero mean),
	we may replace $\partial_t e_\phi^\varepsilon$ by
	$\partial_t\widetilde{e}_\phi^\varepsilon$ on the left, giving 
	\[
	\left\langle
	\partial_t e_\phi^\varepsilon,
	\mathcal N_\varepsilon\widetilde e_\phi^\varepsilon
	\right\rangle
	=
	\frac12
	\frac{d}{dt}
	\|\widetilde e_\phi^\varepsilon\|_{H^{-1}_\varepsilon}^2 .
	\]
	Moreover,
	\[
	\int_{\Omega_p^\varepsilon}
	\nabla e_\mu^\varepsilon
	\cdot
	\nabla\mathcal N_\varepsilon\widetilde e_\phi^\varepsilon\,dx
	=
	\int_{\Omega_p^\varepsilon}
	e_\mu^\varepsilon\widetilde e_\phi^\varepsilon\,dx .
	\]
	Hence
	\begin{align}
		\label{eq:phase-negative-test}
		\frac12\frac{d}{dt}
		\|\widetilde e_\phi^\varepsilon\|_{H^{-1}_\varepsilon}^2
		+
		\int_{\Omega_p^\varepsilon}
		e_\mu^\varepsilon\widetilde e_\phi^\varepsilon\,dx
		=
		-\int_{\Omega_p^\varepsilon}
		\bigl(G(\phi_\varepsilon)-G(\Phi^\varepsilon)\bigr)
		\mathcal N_\varepsilon\widetilde e_\phi^\varepsilon\,dx
		-
		\left\langle
		\mathcal R_1^\varepsilon,
		\mathcal N_\varepsilon\widetilde e_\phi^\varepsilon
		\right\rangle .
	\end{align}
	
	\textbf{Step 2: Testing the chemical-potential equation by the phase error.}
	Take \(\eta=\widetilde e_\phi^\varepsilon\)
	in \eqref{eq:error-chem-weak}. Since
	\(\nabla\widetilde e_\phi^\varepsilon=\nabla e_\phi^\varepsilon\), we get
	\begin{align}
		\label{eq:chem-phase-test}
		\int_{\Omega_p^\varepsilon}
		e_\mu^\varepsilon\widetilde e_\phi^\varepsilon\,dx
		=
		\|\nabla e_\phi^\varepsilon\|_{L^2(\Omega_p^\varepsilon)}^2
		+
		\int_{\Omega_p^\varepsilon}
		\bigl(F'(\phi_\varepsilon)-F'(\Phi^\varepsilon)\bigr)
		\widetilde e_\phi^\varepsilon\,dx
		-
		\left\langle
		\mathcal R_2^\varepsilon,
		\widetilde e_\phi^\varepsilon
		\right\rangle .
	\end{align}
	Combining \eqref{eq:phase-negative-test} and
	\eqref{eq:chem-phase-test} gives
	\begin{multline}
		\label{eq:basic-negative-energy}
		\frac12\frac{d}{dt}
		\|\widetilde e_\phi^\varepsilon\|_{H^{-1}_\varepsilon}^2
		+
		\|\nabla e_\phi^\varepsilon\|_{L^2(\Omega_p^\varepsilon)}^2
		=
		-\int_{\Omega_p^\varepsilon}
		\bigl(F'(\phi_\varepsilon)-F'(\Phi^\varepsilon)\bigr)
		\widetilde e_\phi^\varepsilon\,dx
		-\int_{\Omega_p^\varepsilon}
		\bigl(G(\phi_\varepsilon)-G(\Phi^\varepsilon)\bigr)
		\mathcal N_\varepsilon\widetilde e_\phi^\varepsilon\,dx\\
		+
		\left\langle
		\mathcal R_2^\varepsilon,
		\widetilde e_\phi^\varepsilon
		\right\rangle
		-
		\left\langle
		\mathcal R_1^\varepsilon,
		\mathcal N_\varepsilon\widetilde e_\phi^\varepsilon
		\right\rangle .
	\end{multline}
	
\textbf{Step 3: Estimates of the nonlinear terms.}
The Sobolev embedding $H^1(\Omega_p^\varepsilon)\hookrightarrow L^6(\Omega_p^\varepsilon)$
holds for $d\leq 3$ with constant uniform in $\varepsilon$, by
Remark~\ref{rem:frame}. The uniform $H^1$-bounds
$\|\phi_\varepsilon\|_{L^\infty(0,T;H^1(\Omega_p^\varepsilon))}\leq C$ and
$\|\Phi^\varepsilon\|_{L^\infty(0,T;H^1(\Omega_p^\varepsilon))}\leq C$
follow from Lemma~\ref{lem:energy} and \eqref{eq:corrector-H1-bounds}
respectively.

The interpolation inequality for zero-mean functions on $\Omega_p^\varepsilon$
states: for $f\in H^1(\Omega_p^\varepsilon)$ with $\int_{\Omega_p^\varepsilon}f\,dx=0$,
\[
\|f\|_{L^2(\Omega_p^\varepsilon)}
\leq
C\|f\|_{H^{-1}_\varepsilon}^{1/2}
\|\nabla f\|_{L^2(\Omega_p^\varepsilon)}^{1/2},
\]
with $C>0$ independent of $\varepsilon$. This follows from the uniform
Poincar\'e--Wirtinger inequality on $\Omega_p^\varepsilon$ (Remark~\ref{rem:frame})
and the definition of $\|\cdot\|_{H^{-1}_\varepsilon}$: indeed,
\[
\|f\|_{L^2}^2
=
\langle f, f\rangle
=
\langle -\Delta\mathcal{N}_\varepsilon f, f\rangle
=
\|\nabla\mathcal{N}_\varepsilon f\|_{L^2}
\|\nabla f\|_{L^2}
=
\|f\|_{H^{-1}_\varepsilon}
\|\nabla f\|_{L^2},
\]
where we used integration by parts and the definition of $\mathcal{N}_\varepsilon$.

Since $e_\phi^\varepsilon=\widetilde{e}_\phi^\varepsilon+m_\varepsilon$,
\[
\|e_\phi^\varepsilon\|_{L^2(\Omega_p^\varepsilon)}
\leq
\|\widetilde{e}_\phi^\varepsilon\|_{L^2(\Omega_p^\varepsilon)}
+
C|m_\varepsilon|
\leq
C\|\widetilde{e}_\phi^\varepsilon\|_{H^{-1}_\varepsilon}^{1/2}
\|\nabla e_\phi^\varepsilon\|_{L^2(\Omega_p^\varepsilon)}^{1/2}
+C|m_\varepsilon|.
\]

\textbf{Estimate of the $F'$ term.}
Using $|F'(a)-F'(b)|\leq C(1+|a|^2+|b|^2)|a-b|$, H\"older's inequality
with exponents $(3,2,6)$ in $d=3$,
\[
\left|
\int_{\Omega_p^\varepsilon}
\bigl(F'(\phi_\varepsilon)-F'(\Phi^\varepsilon)\bigr)
\widetilde{e}_\phi^\varepsilon\,dx
\right|
\leq
C\|1+|\phi_\varepsilon|^2+|\Phi^\varepsilon|^2\|_{L^3(\Omega_p^\varepsilon)}
\|e_\phi^\varepsilon\|_{L^2(\Omega_p^\varepsilon)}
\|\widetilde{e}_\phi^\varepsilon\|_{L^6(\Omega_p^\varepsilon)}.
\]
By the uniform $H^1$-bounds and the embedding $H^1\hookrightarrow L^6$,
$\|1+|\phi_\varepsilon|^2+|\Phi^\varepsilon|^2\|_{L^3(\Omega_p^\varepsilon)}\leq C$
and $\|\widetilde{e}_\phi^\varepsilon\|_{L^6(\Omega_p^\varepsilon)}\leq
C\|\nabla e_\phi^\varepsilon\|_{L^2(\Omega_p^\varepsilon)}$,
both with constants uniform in $\varepsilon$. Substituting the bound on
$\|e_\phi^\varepsilon\|_{L^2}$ and applying Young's inequality
$ab\leq\delta a^2+C_\delta b^2$ twice, we obtain for every $\delta>0$,
\begin{equation}
	\label{eq:F-term-est}
	\left|
	\int_{\Omega_p^\varepsilon}
	\bigl(F'(\phi_\varepsilon)-F'(\Phi^\varepsilon)\bigr)
	\widetilde{e}_\phi^\varepsilon\,dx
	\right|
	\leq
	\delta
	\|\nabla e_\phi^\varepsilon\|_{L^2(\Omega_p^\varepsilon)}^2
	+
	C_\delta
	\|\widetilde{e}_\phi^\varepsilon\|_{H^{-1}_\varepsilon}^2
	+
	C_\delta|m_\varepsilon|^2.
\end{equation}

\textbf{Estimate of the $G$ term.}
Since $G$ is globally Lipschitz with constant $C_G$,
\[
\left|
\int_{\Omega_p^\varepsilon}
\bigl(G(\phi_\varepsilon)-G(\Phi^\varepsilon)\bigr)
\mathcal{N}_\varepsilon\widetilde{e}_\phi^\varepsilon\,dx
\right|
\leq
C_G
\|e_\phi^\varepsilon\|_{L^2(\Omega_p^\varepsilon)}
\|\mathcal{N}_\varepsilon\widetilde{e}_\phi^\varepsilon\|_{L^2(\Omega_p^\varepsilon)}.
\]
By the definition of $\mathcal{N}_\varepsilon$ and the uniform
Poincar\'e--Wirtinger inequality,
$\|\mathcal{N}_\varepsilon\widetilde{e}_\phi^\varepsilon\|_{L^2(\Omega_p^\varepsilon)}
\leq C\|\widetilde{e}_\phi^\varepsilon\|_{H^{-1}_\varepsilon}$
with $C$ independent of $\varepsilon$. Substituting the bound on
$\|e_\phi^\varepsilon\|_{L^2}$ and applying Young's inequality,
we obtain for every $\delta>0$,
\begin{equation}
	\label{eq:G-term-est}
	\left|
	\int_{\Omega_p^\varepsilon}
	\bigl(G(\phi_\varepsilon)-G(\Phi^\varepsilon)\bigr)
	\mathcal{N}_\varepsilon\widetilde{e}_\phi^\varepsilon\,dx
	\right|
	\leq
	\delta
	\|\nabla e_\phi^\varepsilon\|_{L^2(\Omega_p^\varepsilon)}^2
	+
	C_\delta
	\|\widetilde{e}_\phi^\varepsilon\|_{H^{-1}_\varepsilon}^2
	+
	C_\delta|m_\varepsilon|^2.
\end{equation}
	
	\textbf{Step 4: Estimates of the residual terms.}
	By the definition of \(\mathcal N_\varepsilon\),
	\[
	\|\mathcal N_\varepsilon\widetilde e_\phi^\varepsilon\|_{H^1(\Omega_p^\varepsilon)}
	\leq
	C
	\|\widetilde e_\phi^\varepsilon\|_{H^{-1}_\varepsilon}.
	\]
	Therefore
	\begin{equation}
		\label{eq:R1-term-est}
		\left|
		\left\langle
		\mathcal R_1^\varepsilon,
		\mathcal N_\varepsilon\widetilde e_\phi^\varepsilon
		\right\rangle
		\right|
		\leq
		C
		\|\mathcal R_1^\varepsilon\|_{H^1(\Omega_p^\varepsilon)'}^2
		+
		C
		\|\widetilde e_\phi^\varepsilon\|_{H^{-1}_\varepsilon}^2 .
	\end{equation}
	Moreover, since \(\widetilde e_\phi^\varepsilon\) has zero mean, the
	uniform Poincaré--Wirtinger inequality gives
	\[
	\|\widetilde e_\phi^\varepsilon\|_{H^1(\Omega_p^\varepsilon)}
	\leq
	C\|\nabla e_\phi^\varepsilon\|_{L^2(\Omega_p^\varepsilon)}.
	\]
	Hence
	\begin{equation}
		\label{eq:R2-term-est}
		\left|
		\left\langle
		\mathcal R_2^\varepsilon,
		\widetilde e_\phi^\varepsilon
		\right\rangle
		\right|
		\leq
		\delta
		\|\nabla e_\phi^\varepsilon\|_{L^2(\Omega_p^\varepsilon)}^2
		+
		C_\delta
		\|\mathcal R_2^\varepsilon\|_{H^1(\Omega_p^\varepsilon)'}^2 .
	\end{equation}
	
	Choosing \(\delta>0\) sufficiently small and combining
	\eqref{eq:basic-negative-energy}--\eqref{eq:R2-term-est}, we obtain
	\begin{equation}
		\label{eq:energy-diff-ineq}
		\frac{d}{dt}
		\|\widetilde e_\phi^\varepsilon\|_{H^{-1}_\varepsilon}^2
		+
		c
		\|\nabla e_\phi^\varepsilon\|_{L^2(\Omega_p^\varepsilon)}^2
		\leq
		C
		\|\widetilde e_\phi^\varepsilon\|_{H^{-1}_\varepsilon}^2
		+
		C|m_\varepsilon|^2
		+
		C\|\mathcal R_1^\varepsilon\|_{H^1(\Omega_p^\varepsilon)'}^2
		+
		C\|\mathcal R_2^\varepsilon\|_{H^1(\Omega_p^\varepsilon)'}^2 .
	\end{equation}
	
	\textbf{Step 5: Control of the mean.}
	Taking \(\zeta=1\) in \eqref{eq:error-phase-weak} gives
	\[
	\frac{d}{dt}m_\varepsilon(t)
	+
	\frac{1}{|\Omega_p^\varepsilon|}
	\int_{\Omega_p^\varepsilon}
	\bigl(G(\phi_\varepsilon)-G(\Phi^\varepsilon)\bigr)\,dx
	=
	-
	\frac{1}{|\Omega_p^\varepsilon|}
	\left\langle
	\mathcal R_1^\varepsilon,1
	\right\rangle .
	\]
	Multiplying by \(m_\varepsilon(t)\), using the Lipschitz continuity of
	\(G\), the uniform lower bound on \(|\Omega_p^\varepsilon|\), and
	Young's inequality, we obtain
	\begin{align*}
		\frac12\frac{d}{dt}|m_\varepsilon(t)|^2
		&\leq
		C|m_\varepsilon(t)|^2
		+
		C|m_\varepsilon(t)|
		\|\widetilde e_\phi^\varepsilon(t)\|_{L^2(\Omega_p^\varepsilon)}
		+
		C|m_\varepsilon(t)|
		\|\mathcal R_1^\varepsilon(t)\|_{H^1(\Omega_p^\varepsilon)'} .
	\end{align*}
	Using again the interpolation inequality for
	\(\widetilde e_\phi^\varepsilon\), we infer that, for every
	\(\delta>0\),
	\begin{equation}
		\label{eq:mean-diff-ineq}
		\frac{d}{dt}|m_\varepsilon(t)|^2
		\leq
		C|m_\varepsilon(t)|^2
		+
		\delta
		\|\nabla e_\phi^\varepsilon(t)\|_{L^2(\Omega_p^\varepsilon)}^2
		+
		C_\delta
		\|\widetilde e_\phi^\varepsilon(t)\|_{H^{-1}_\varepsilon}^2
		+
		C
		\|\mathcal R_1^\varepsilon(t)\|_{H^1(\Omega_p^\varepsilon)'}^2 .
	\end{equation}
	
	\textbf{Step 6: Gronwall argument.}
	Adding \eqref{eq:energy-diff-ineq} and \eqref{eq:mean-diff-ineq}, and
	choosing \(\delta>0\) sufficiently small, yields
	\[
	\frac{d}{dt}
	\left(
	\|\widetilde e_\phi^\varepsilon\|_{H^{-1}_\varepsilon}^2
	+
	|m_\varepsilon|^2
	\right)
	+
	c
	\|\nabla e_\phi^\varepsilon\|_{L^2(\Omega_p^\varepsilon)}^2
	\leq
	C
	\left(
	\|\widetilde e_\phi^\varepsilon\|_{H^{-1}_\varepsilon}^2
	+
	|m_\varepsilon|^2
	\right)
	+
	C
	\|\mathcal R_1^\varepsilon\|_{H^1(\Omega_p^\varepsilon)'}^2
	+
	C
	\|\mathcal R_2^\varepsilon\|_{H^1(\Omega_p^\varepsilon)'}^2 .
	\]
	Integrating over \((0,t)\) and applying Gronwall's inequality gives
	\begin{align*}
		&\|\widetilde e_\phi^\varepsilon\|_{L^\infty(0,T;H^{-1}_\varepsilon)}^2
		+
		\|m_\varepsilon\|_{L^\infty(0,T)}^2
		+
		\|\nabla e_\phi^\varepsilon\|_{L^2(0,T;L^2(\Omega_p^\varepsilon))}^2
		\\
		&\quad\leq
		C\Bigl(
		\|\widetilde e_\phi^\varepsilon(0)\|_{H^{-1}_\varepsilon}^2
		+
		|m_\varepsilon(0)|^2
		+
		\|\mathcal R_1^\varepsilon\|_{L^2(0,T;H^1(\Omega_p^\varepsilon)')}^2
		+
		\|\mathcal R_2^\varepsilon\|_{L^2(0,T;H^1(\Omega_p^\varepsilon)')}^2
		\Bigr).
	\end{align*}
	Finally, since
	\(e_\phi^\varepsilon=\widetilde e_\phi^\varepsilon+m_\varepsilon\), the
	uniform Poincaré--Wirtinger inequality gives
	\[
	\|e_\phi^\varepsilon(t)\|_{L^2(\Omega_p^\varepsilon)}
	\leq
	C\|\nabla e_\phi^\varepsilon(t)\|_{L^2(\Omega_p^\varepsilon)}
	+
	C|m_\varepsilon(t)|.
	\]
	Thus the previous estimate also controls
	\(\|e_\phi^\varepsilon\|_{L^2(0,T;H^1(\Omega_p^\varepsilon))}\).
	Taking square roots and using \eqref{eq:residual-bound} proves
	\eqref{eq:stability-rate}. The final \(\GO(\varepsilon^{1/2})\)
	statement follows immediately from \eqref{eq:RIC}.
\end{proof}


\begin{corollary}[Uncorrected convergence rate]
	\label{cor:uncorrected}
	Under the assumptions of Theorem~\ref{thm:corrector-estimate}, one has
	\begin{equation}
		\label{eq:uncorrected-bound-general}
		\|\phi_\varepsilon-\phiz\|_{L^2(0,T;L^2(\Omega_p^\varepsilon))}
		+
		\|\mu_\varepsilon-\muz\|_{L^2(0,T;H^1(\Omega_p^\varepsilon)')}
		\leq
		C\Bigl(
		\|\widetilde e_\phi^\varepsilon(0)\|_{H^{-1}_\varepsilon}
		+
		|m_\varepsilon(0)|
		+
		\varepsilon^{1/2}
		\Bigr).
	\end{equation}
	In particular, if \eqref{eq:RIC} holds, then
	\begin{equation}
		\label{eq:uncorrected-bound-prepared}
		\|\phi_\varepsilon-\phiz\|_{L^2(0,T;L^2(\Omega_p^\varepsilon))}
		+
		\|\mu_\varepsilon-\muz\|_{L^2(0,T;H^1(\Omega_p^\varepsilon)')}
		\leq
		C\varepsilon^{1/2}.
	\end{equation}
\end{corollary}

\begin{proof}
	We first estimate the phase variable. By the triangle inequality,
	\[
	\|\phi_\varepsilon-\phiz\|_{L^2(0,T;L^2(\Omega_p^\varepsilon))}
	\leq
	\|e_\phi^\varepsilon\|_{L^2(0,T;L^2(\Omega_p^\varepsilon))}
	+
	\|\Phi^\varepsilon-\phiz\|_{L^2(0,T;L^2(\Omega_p^\varepsilon))}.
	\]
	The first term is controlled by Theorem~\ref{thm:corrector-estimate}.
	For the second term, using the definition of \(\Phi^\varepsilon\) and
	the product estimate for the scale-splitting operator,
	\[
	\|\Phi^\varepsilon-\phiz\|_{L^2(0,T;L^2(\Omega_p^\varepsilon))}
	\leq
	C\varepsilon
	\|\phiz\|_{L^2(0,T;H^1(\Omega))}.
	\]
	Therefore,
	\[
	\|\phi_\varepsilon-\phiz\|_{L^2(0,T;L^2(\Omega_p^\varepsilon))}
	\leq
	C\Bigl(
	\|\widetilde e_\phi^\varepsilon(0)\|_{H^{-1}_\varepsilon}
	+
	|m_\varepsilon(0)|
	+
	\varepsilon^{1/2}
	\Bigr).
	\]
	
	We now estimate the chemical potential in the dual norm. By definition,
	\[
	\mu_\varepsilon-\muz
	=
	e_\mu^\varepsilon
	+
	M^\varepsilon-\muz .
	\]
	The corrector part satisfies
	\[
	\|M^\varepsilon-\muz\|_{L^2(0,T;L^2(\Omega_p^\varepsilon))}
	\leq
	C\varepsilon
	\|\muz\|_{L^2(0,T;H^1(\Omega))},
	\]
	and hence also
	\[
	\|M^\varepsilon-\muz\|_{L^2(0,T;H^1(\Omega_p^\varepsilon)')}
	\leq
	C\varepsilon
	\|\muz\|_{L^2(0,T;H^1(\Omega))}.
	\]
	
	It remains to control \(e_\mu^\varepsilon\) in
	\(L^2(0,T;H^1(\Omega_p^\varepsilon)')\). Let
	\(\eta\in H^1(\Omega_p^\varepsilon)\). From
	\eqref{eq:error-chem-weak},
	\[
	\int_{\Omega_p^\varepsilon}
	e_\mu^\varepsilon\eta\,dx
	=
	\int_{\Omega_p^\varepsilon}
	\nabla e_\phi^\varepsilon\cdot\nabla\eta\,dx
	+
	\int_{\Omega_p^\varepsilon}
	\bigl(F'(\phi_\varepsilon)-F'(\Phi^\varepsilon)\bigr)\eta\,dx
	-
	\bigl\langle
	\mathcal R_2^\varepsilon,\eta
	\bigr\rangle .
	\]
	The first term is bounded by
	\[
	\|\nabla e_\phi^\varepsilon\|_{L^2(\Omega_p^\varepsilon)}
	\|\eta\|_{H^1(\Omega_p^\varepsilon)}.
	\]
	For the nonlinear term, using the growth condition on \(F'\), the
	uniform \(H^1\)-bounds on \(\phi_\varepsilon\) and \(\Phi^\varepsilon\),
	and the embedding
	\(H^1(\Omega_p^\varepsilon)\hookrightarrow L^6(\Omega_p^\varepsilon)\)
	for \(d\leq3\), we get
	\[
	\left|
	\int_{\Omega_p^\varepsilon}
	\bigl(F'(\phi_\varepsilon)-F'(\Phi^\varepsilon)\bigr)\eta\,dx
	\right|
	\leq
	C
	\|e_\phi^\varepsilon\|_{L^2(\Omega_p^\varepsilon)}
	\|\eta\|_{H^1(\Omega_p^\varepsilon)}.
	\]
	Consequently,
	\[
	\|e_\mu^\varepsilon\|_{H^1(\Omega_p^\varepsilon)'}
	\leq
	C
	\|\nabla e_\phi^\varepsilon\|_{L^2(\Omega_p^\varepsilon)}
	+
	C
	\|e_\phi^\varepsilon\|_{L^2(\Omega_p^\varepsilon)}
	+
	\|\mathcal R_2^\varepsilon\|_{H^1(\Omega_p^\varepsilon)'} .
	\]
	Integrating in time and using Theorem~\ref{thm:corrector-estimate} and
	Lemma~\ref{lem:residual}, we obtain
	\[
	\|e_\mu^\varepsilon\|_{L^2(0,T;H^1(\Omega_p^\varepsilon)')}
	\leq
	C\Bigl(
	\|\widetilde e_\phi^\varepsilon(0)\|_{H^{-1}_\varepsilon}
	+
	|m_\varepsilon(0)|
	+
	\varepsilon^{1/2}
	\Bigr).
	\]
	Combining this estimate with the bound for \(M^\varepsilon-\muz\) gives
	\eqref{eq:uncorrected-bound-general}. The prepared-data conclusion
	\eqref{eq:uncorrected-bound-prepared} follows immediately.
\end{proof}

\begin{corollary}[Corrected gradient rate]
	\label{cor:corrected-gradient-rate}
	Under the assumptions of Theorem~\ref{thm:corrector-estimate}, one has
	\begin{equation}
		\label{eq:corrected-gradient-rate}
		\left\|
		\nabla\phi_\varepsilon
		-
		\left[
		\nabla\phiz
		+
		\sum_{i=1}^d
		\Qc_\varepsilon(\partial_{x_i}\phiz)
		\nabla_y\chi_i\!\left(\frac{x}{\varepsilon}\right)
		\right]
		\right\|_{L^2((0,T)\times\Omega_p^\varepsilon)}
		\leq
		C\Bigl(
		\|\widetilde e_\phi^\varepsilon(0)\|_{H^{-1}_\varepsilon}
		+
		|m_\varepsilon(0)|
		+
		\varepsilon^{1/2}
		\Bigr).
	\end{equation}
	In particular, if the initial mismatch satisfies \eqref{eq:RIC}, then
	\begin{equation}
		\label{eq:corrected-gradient-rate-prepared}
		\left\|
		\nabla\phi_\varepsilon
		-
		\left[
		\nabla\phiz
		+
		\sum_{i=1}^d
		\Qc_\varepsilon(\partial_{x_i}\phiz)
		\nabla_y\chi_i\!\left(\frac{x}{\varepsilon}\right)
		\right]
		\right\|_{L^2((0,T)\times\Omega_p^\varepsilon)}
		\leq
		C\varepsilon^{1/2}.
	\end{equation}
\end{corollary}

\begin{proof}
	By definition,
	\[
	e_\phi^\varepsilon=\phi_\varepsilon-\Phi^\varepsilon .
	\]
	Hence Theorem~\ref{thm:corrector-estimate} gives
	\[
	\|\nabla\phi_\varepsilon-\nabla\Phi^\varepsilon\|_
	{L^2((0,T)\times\Omega_p^\varepsilon)}
	\leq
	C\Bigl(
	\|\widetilde e_\phi^\varepsilon(0)\|_{H^{-1}_\varepsilon}
	+
	|m_\varepsilon(0)|
	+
	\varepsilon^{1/2}
	\Bigr).
	\]
	On the other hand, Lemma~\ref{lem:corrector-approx} yields
	\[
	\left\|
	\nabla\Phi^\varepsilon
	-
	\left[
	\nabla\phiz
	+
	\sum_{i=1}^d
	\Qc_\varepsilon(\partial_{x_i}\phiz)
	\nabla_y\chi_i\!\left(\frac{x}{\varepsilon}\right)
	\right]
	\right\|_{L^2((0,T)\times\Omega_p^\varepsilon)}
	\leq
	C\varepsilon
	\|\phiz\|_{L^2(0,T;H^2(\Omega))}.
	\]
	The triangle inequality gives \eqref{eq:corrected-gradient-rate}. The
	prepared-data estimate \eqref{eq:corrected-gradient-rate-prepared}
	follows from \eqref{eq:RIC}.
\end{proof}

\begin{remark}[Periodic domains and improved rates]
	\label{rem:periodic-improved}
	If \(\Omega=\mathbb T^d\) is the flat torus, there is no boundary layer
	of incomplete cells. In this case the variational consistency estimate
	in Lemma~\ref{lem:residual} improves from order \(\varepsilon^{1/2}\) to
	order \(\varepsilon\). Hence the stability estimate of
	Theorem~\ref{thm:corrector-estimate} becomes
	\[
	\|\widetilde e_\phi^\varepsilon\|_{L^\infty(0,T;H^{-1}_\varepsilon)}
	+
	\|e_\phi^\varepsilon\|_{L^2(0,T;H^1(\Omega_p^\varepsilon))}
	+
	\|m_\varepsilon\|_{L^\infty(0,T)}
	\leq
	C\Bigl(
	\|\widetilde e_\phi^\varepsilon(0)\|_{H^{-1}_\varepsilon}
	+
	|m_\varepsilon(0)|
	+
	\varepsilon
	\Bigr).
	\]
	In particular, using Lemma~\ref{lem:corrector-approx}, we also obtain
	the corrected gradient estimate
	\begin{equation}
		\left\|
		\nabla\phi_\varepsilon
		-
		\left[
		\nabla\phiz
		+
		\sum_{i=1}^d
		\Qc_\varepsilon(\partial_{x_i}\phiz)
		\nabla_y\chi_i\!\left(\frac{x}{\varepsilon}\right)
		\right]
		\right\|_{L^2((0,T)\times\Omega_p^\varepsilon)}
		\leq
		C\Bigl(
		\|\widetilde e_\phi^\varepsilon(0)\|_{H^{-1}_\varepsilon}
		+
		|m_\varepsilon(0)|
		+
		\varepsilon
		\Bigr).
		\label{eq:periodic-corrected-gradient}
	\end{equation}
	Accordingly, the uncorrected estimate becomes
	\[
	\|\phi_\varepsilon-\phiz\|_{L^2(0,T;L^2(\Omega_p^\varepsilon))}
	+
	\|\mu_\varepsilon-\muz\|_{L^2(0,T;H^1(\Omega_p^\varepsilon)')}
	\leq
	C\Bigl(
	\|\widetilde e_\phi^\varepsilon(0)\|_{H^{-1}_\varepsilon}
	+
	|m_\varepsilon(0)|
	+
	\varepsilon
	\Bigr).
	\]
	If the initial mismatch is \(O(\varepsilon)\) in the natural negative
	norm and in the mean, then the corrected gradient estimate
	\eqref{eq:periodic-corrected-gradient} and the uncorrected variational
	estimate above are both of order \(O(\varepsilon)\).
	
	We do not claim an uncorrected rate for
	\(\nabla\phi_\varepsilon-\nabla\phiz\). In general, the microscopic
	gradient contains the oscillating first-order term
	\[
	\sum_{i=1}^d
	\Qc_\varepsilon(\partial_{x_i}\phiz)
	\nabla_y\chi_i\!\left(\frac{x}{\varepsilon}\right),
	\]
	which is of order one in \(L^2((0,T)\times\Omega_p^\varepsilon)\).
	Thus the natural gradient estimate is the corrected one stated in
	\eqref{eq:periodic-corrected-gradient}.
\end{remark}

\begin{remark}[On the absence of a gradient rate for the chemical potential]
	\label{rem:no-mu-gradient}
	The estimates above do not control $\nabla\mu_\varepsilon$ at rate
	$\varepsilon^{1/2}$. The variational energy argument controls
	$e_\phi^\varepsilon = \phi_\varepsilon - \Phi^\varepsilon$ in
	$L^2(0,T;H^1(\Omega_p^\varepsilon))$, but controls
	$e_\mu^\varepsilon = \mu_\varepsilon - M^\varepsilon$ only in the dual
	space $L^2(0,T;H^1(\Omega_p^\varepsilon)')$. Obtaining
	$\|\nabla e_\mu^\varepsilon\|_{L^2}$ would require testing the
	phase-error equation with $e_\mu^\varepsilon$, which in turn requires
	testing the chemical-potential error equation with
	$\partial_t e_\phi^\varepsilon$. This is not justified at the regularity
	level $\mathcal{R}_2^\varepsilon \in L^2(0,T;H^1(\Omega_p^\varepsilon)')$.
	The dual estimate $\|\mu_\varepsilon - \muz\|_{L^2(0,T;H^1(\Omega_p^\varepsilon)')}$
	is therefore the natural conclusion of the present framework.
\end{remark}

\section{On the $\mathcal{O}(\varepsilon)$ rate and future directions}
	\label{rem:future-rate}
	The rate $\mathcal{O}(\varepsilon^{1/2})$ of Theorem~\ref{thm:corrector-estimate}
	reflects the boundary layer of incomplete cells near $\partial\Omega$ and is
	sharp within the present variational framework. On a bounded Lipschitz domain
	the expected rate is $\mathcal{O}(\varepsilon)$, obtainable in principle by two
	routes.
	
	The first is the duality argument of Kenig--Lin--Shen~\cite{KenigLinShen2012}
	and Gu~\cite{Gu2018}: pair $e^\varepsilon_\phi$ against the solution of a
	backward Cahn--Hilliard system on $\Omega^\varepsilon_p$ and extract an extra
	$\varepsilon^{1/2}$ via a boundary-strip estimate. The obstruction is that the
	adjoint system is nonlinear and $\varepsilon$-dependent, so the full corrector
	theory---including the analogue of Lemma~\ref{lem:residual} for the backward
	problem---must be repeated.
	
	The second is the operator-estimate approach of Birman--Suslina
	\cite{BirmanSuslina2004,BirmanSuslina2006,BirmanSuslina2007}: for linear
	operators this yields $\mathcal{O}(\varepsilon)$ in the $L^2$-operator norm
	without regularity assumptions on the coefficients; see
	Suslina~\cite{Suslina2013Neumann,Suslina2013Dirichlet}. Extending it to the
	nonlinear Cahn--Hilliard system requires treating $F'(\phi_\varepsilon)$ and
	$G(\phi_\varepsilon)$ as lower-order perturbations, with
	Theorem~\ref{thm:corrector-estimate} as the input bound. Both directions are
	left for future work.

\bibliographystyle{plainnat}
{ \bibliography{references}
}

\end{document}